\documentclass[reqno]{amsart}
\usepackage{amssymb}
\usepackage{graphicx}
\usepackage{amsmath}

\usepackage{yfonts}
\usepackage{pbsi}

\usepackage[usenames, dvipsnames]{color}
\usepackage{verbatim}
\usepackage{mathrsfs}
\usepackage{bm}
\usepackage{cite}

\usepackage{mathtools}
\newcommand{\pvint}{\mathop{\mathrlap{\pushpv}}\!\int}
\newcommand{\pushpv}{\mathchoice
	{\mkern5mu\rule[.6ex]{.5em}{1pt}}
	{\mkern2.8mu\rule[.5ex]{.35em}{.8pt}}
	{\mkern2.5mu\rule[.29ex]{.3em}{.7pt}}
	{\mkern2mu\rule[.2ex]{.2em}{.5pt}}
}

\numberwithin{equation}{section}

\newtheorem{theorem}{Theorem}[section]
\newtheorem{corollary}[theorem]{Corollary}
\newtheorem{lemma}[theorem]{Lemma}
\newtheorem{prop}[theorem]{Proposition}

\theoremstyle{definition}

\theoremstyle{definition}
\newtheorem{remark}[theorem]{Remark}

\theoremstyle{definition}

\theoremstyle{definition}

\makeatletter
\def\dashint{\operatorname%
{\,\,\text{\bf-}\kern-.98em\DOTSI\intop\ilimits@\!\!}}
\makeatother

\def\\det{\text{det}}

\def\.5{\frac{1}{2}}

\newcommand{\laplas}{\ensuremath{\Delta}}

\newcommand{\va}{\ensuremath{\varepsilon}}
\newcommand{\ptl}{\ensuremath{\partial}}
\newcommand{\om}{\ensuremath{\Omega}}

\newcommand{\R}{\ensuremath{\mathbb{R}}}

\newcommand{\lam}{\ensuremath{\lambda}}
\newcommand{\alp}{\ensuremath{\alpha}}
\newcommand{\dt}{\ensuremath{\delta}}

\newcommand{\RN}[1]{%
  \textup{\uppercase\expandafter{\romannumeral#1}}%
}

\newcommand{\Div}{\operatorname{div}}
\newcommand{\dist}{\text{dist}}

\renewcommand{\epsilon}{\varepsilon}

\newcounter{marnote}

\begin{document}


\title[Estimates for the stress concentration]{Estimates and Asymptotics for the stress concentration between closely spaced stiff $C^{1, \gamma}$ inclusions in linear elasticity}

\author[Y. Chen]{Yu Chen}
\address[Y. Chen]{School of Mathematical Sciences, Beijing Normal University, Laboratory of Mathematics and Complex Systems, Ministry of Education, Beijing 100875, China. }
\email{chenyu@amss.ac.cn.}

\author[H.G. Li]{Haigang Li}
\address[H.G. Li]{School of Mathematical Sciences, Beijing Normal University, Laboratory of Mathematics and Complex Systems, Ministry of Education, Beijing 100875, China. }
\email{hgli@bnu.edu.cn. }

\maketitle

\begin{abstract}
This paper is concerned with the stress concentration phenomenon in elastic composite materials when the inclusions are very closely spaced. We investigate the gradient blow-up estimates for the Lam\'{e} system of linear elasticity with partially infinite coefficients to show the dependence of the degree of stress enhancement on the distance between inclusions in a composite containing densely placed stiff inclusions. In this paper we assume that the inclusions to be of $C^{1, \gamma}$, weaker than the previous $C^{2, \gamma}$ assumption. To overcome this new difficulty, we make use of $W^{1, p}$ estimates for elliptic system with right hand side in divergence form, instead of a direct $W^{2, p}$ argument for $C^{2, \gamma}$ inclusion case, and combine with the Campanato's approach to establish the optimal gradient estimates, including upper and lower bounds. Moreover, an asymptotic formula of the gradient near the blow-up point is obtained for some symmetric $C^{1, \gamma}$ inclusions.
\vspace{.5cm}



\noindent {\bf Keywords:} \, Lam\'{e} system, 
Gradient estimates, Blow-up rates.

\end{abstract}

\section{Introduction}

\subsection{Background and problem formulation}
In high-contrast composite materials, when inclusions are very closely spaced the physical fields, such as the stress field or the electric field, may be arbitrarily large in the narrow region between them. From an engineering point of view, it is quite important to precisely understand such field concentration phenomenon.
In this paper, we consider the following boundary value problem of  Lam\'{e} system with partially degenerated coefficients to model a composite with two adjacent stiff inclusions. We follow the notation used in \cite{bll,bll2}. Let $D$ be a bounded open set in $\mathbb R^{d}$, $d\geq2$, $D_1$ and $D_2$ be two adjacent convex subdomains, with $\va$-apart. Let $u=(u^{(1)}, u^{(2)},\cdots, u^{(d)})^{T}$: $D \to \R^d$  be a vector-valued function, representing the displacement field, and verify the following boundary value problem
\begin{equation}\label{lame-system-infty}
\left\{
\begin{aligned}
&\mathcal{L}_{\lam, \mu}u:=\nabla\cdot \big( \mathbb{C}^0 e( u) \big)=0&\mbox{in}&~\Omega:=D\setminus\overline{D_{1}\cup D_{2} },\\
&  u|_+= u|_-&\mbox{on}&~\ptl D_1\cup \ptl D_2,\\
&e(u)=0 &\mbox{in}&~ D_1 \cup D_2,\\
&\int_{\ptl D_i}\frac{\ptl  u}{\ptl \nu}\Big|_+\cdot  \psi^\alp=0 
& i&=1,\, 2,\,~~\alp=1, \, 2,\, \cdots, \frac{d}{2}(d+1), \\
& u=\varphi &\mbox{on}&~\partial D,
\end{aligned}\right.
\end{equation}
where the elastic tensor $\mathbb{C}^0=\mathbb{C}(\lam, \mu)$ consists of elements
\begin{equation}\label{def-C-0}
C_{ijkl}(\lam, \mu)=\lam\delta_{ij}\delta_{kl}+\mu(\delta_{ik}\delta_{jl}+\delta_{il}\delta_{jk}),\quad i,\,j,\,k,\,l\,=1,\,2,\cdots,d,
\end{equation}
$\delta_{ij}$ is Kronecker symbol: $\delta_{ij}=0$ for $i\neq j$, $\delta_{ij}=1$ for $i=j$. 
So
\begin{equation}\label{operator}
\mathcal{L}_{\lam, \mu}u=\nabla\cdot \big( \mathbb{C}^0 e( u) \big)=\mu \Delta u+(\lam+\mu)\nabla \nabla\cdot u,
\end{equation}
where
\[e(u):=\frac{1}{2}\Big(\nabla u+(\nabla u)^{T}\Big)\]
is the strain tensor. The Lam\'{e} pair $(\lam,\mu)$ satisfies the strong convexity conditions $\mu>0$ and $d\lam+2\mu>0$, and the corresponding conormal derivatives on $\ptl D_i$ are defined by
\[\frac{\ptl  u}{\ptl \nu}\Big|_+:=\big( \mathbb{C}^0 e( u) \big){\bf n}=\lam (\nabla\cdot u){\bf n}+\mu \big(\nabla u+(\nabla u)^{T}\big){\bf n},\]
where ${\bf n}$ is the unit outer normal of $\ptl D_i$, $i=1,\, 2$. Here and throughout this paper, the subscript $\pm$ indicates the limit from outside and inside the domain, respectively. $\{\psi^\alp\}_{\alp=1}^{\frac{d(d+1)}{2}}$ is a basis of
\begin{equation*}
\Psi:=\Big\{\psi \in C^1(\R^d; \R^d) ~ : ~ \nabla \psi +(\nabla \psi)^{T}=0 \Big\},
\end{equation*}
a linear space of rigid displacement in $\R^d$, with dimension $d(d+1)/2$. The solution $u$ of \eqref{lame-system-infty} variationally is the unique minimizer of the following energy functional 
\begin{equation*}
E[u]=\min_{v\in\mathcal{S}}E[v],\qquad E[v]=\frac{1}{2}\int_{\Omega}\Big(\mathbb{C}^{0}e(v),e(v)\Big)dx,
\end{equation*}
in the function space
\[\mathcal{S}:=\{u\in H^1_{\varphi}(D; \R^d)~:~e(u)=0~\text{in}~D_1\cup D_2\}.\]
By the elliptic regularity theorems, $u\in C^1(\overline{\Omega}; \R^d) \cap C^1(\overline{D_1\cup D_2}; \R^d)$ for $C^{1,\gamma}$ domain, see \cite{bll}. Moreover, the solution $u$ of \eqref{lame-system-infty} is actually a weak limit in $H^1(\Omega; \R^d)$ of the solutions $\{u_n\}$ of the following isotropic homogeneous Lam\'{e} systems with piecewise constant coefficients
\begin{equation*}
\begin{cases}
\nabla \cdot \Big((\chi_{\Omega}\mathbb{C}^0+\chi_{D_1\cup D_2}\mathbb{C}^n)e(u_n)\Big)=0& \mbox{in}~D,\\
u_n=\varphi&\mbox{on}~\partial{D},
\end{cases}
\end{equation*}
where $\chi_D$ is characteristic function of $D$ and $\mathbb{C}^n=\mathbb{C}(\lam_n, \mu_n)$, when $\min\{\mu_n, d\lam_n+2\mu_n\}\to \infty$. This limit process can refer to the appendix in \cite{bll}, with modifying the $C^{2,\gamma}$ domain to $C^{1,\gamma}$.

This problem is stimulated by the study of the initiation and growth of damage in composite materials. For the Lam\'{e} system \eqref{operator}, Babu\u{s}ka et al. \cite{basl} numerically observed that the stress still remains bounded even if the distance $\va$ beween inclusions  tends to zero. Bonnetier and Vogelius \cite{bv} studied the scalar equation
\begin{equation}\label{equ-scalar}
\nabla \cdot \big((1+(k-1)\chi_{D_1\cup D_2})\nabla u\big)=0
\end{equation}
to model a problem of antiplane shear or electric conduction, where $D_1$ and $D_2$ are two touching disks with comparable radii in two dimensions and $k\neq 1$. They proved rigorously that the gradient of the solution of \eqref{equ-scalar} is indeed bounded, by using the M\"{o}bius transformation and the maximum principle. This result was extended by Li and Vogelius \cite{lv} to a large class of divergence form second order elliptic equations with piecewise H\"{o}lder continuous coefficients in all dimensions. They established global Lipschitz estimates and piecewise $C^{1,\alp}$ estimates. Li and Nirenberg \cite{ln} further extended that to general divergence form elliptic systems including the Lam\'{e} system. As to higher order derivative estimates, we draw the attention of readers to the open problem on Page 894 of \cite{lv}. So far, some progress has been made for scalar equation \eqref{equ-scalar} in two dimensions, see \cite{lv,dongli,dongzhang}.

Notice that those estimates in \cite{ln,lv} depend on the ellipticity of the coefficients. However, when the coefficients in $D_1$ and $D_2$ degenerate to infinite, the situation becomes quite different. For the scalar case, we call it {\it perfect conductivity problem}. Keller \cite{k1} was the first to compute the effective electrical conductivity for a composite containing a dense array of perfectly conducting spheres or circular cylinders. Recently, much effort has been devoted to understanding of the blow-up mechanics when the distance $\va$ tends to zero. 
It is known that the blow-up rate of $|\nabla u|$ is $\va^{-1/2}$ in two dimensions \cite{akl,aklll,y1}, $|\va\ln \va|^{-1}$ in three dimensions \cite{bly1,bly2,ly2}, $\va^{-1}$ in four and higher dimensions \cite{bly1,bly2}. There has been a long list of literature in this direction of research, for example, \cite{ackly,adkl,akllz,bt1,bt2,ky2,l,llby,LX,y3}.
Moreover, the characterizations of the singular behavior of $\nabla u$ was further developed in \cite{g,kly0,kly,kly2,lwx}. We would like to mention that the corresponding nonlinear problems, say, for $p$-Laplacian and for Finsler Laplacian, were  studied in \cite{gs,gs1,g}. While, when the coefficients partially degenerate to zero, the blow-up gradient estimates were obtained in \cite{akl,bly2}, and the corresponding characterization was obtained in \cite{lyu} by using an explicit function in dimensions two. For more work on elliptic equations and systems related to the study of composites, see \cite{adkl,bc,bjl,dong,dongli,dongzhang,jlx,fknn,l2020,llby,ll-p,LX,lz,m,g1} and the references therein.

In this paper, we investigate the gradient blow-up estimates for the solution of the linear elasticity problem \eqref{lame-system-infty}, as $\va$ goes to zero. However, there is significant difficulty in applying the methods for scalar equations to the Lam\'{e} system. For instance, the maximum principle does not hold for the Lam\'{e} system any more. Recently, Bao, Li and Li \cite{bll,bll2} applied an energy method and an iteration technique which was first used in \cite{llby} to obtain pointwise upper bounds of $|\nabla u|$ in the narrow region between inclusions. The corresponding boundary estimates when stiff inclusions are close to matrix boundary were established in \cite{bjl}. We notice that all these estimates mentioned above are established under an assumption that the inclusions are of $C^{2,\gamma}$, $0<\gamma<1$. On the other hand, in order to extend the layer potential techniques used in \cite{aklll} for scalar case  to the Lam\'{e} system, Kang and Yu \cite{ky}  assumed the inclusions are of $C^{3, \gamma}$, and then characterized the singular behavior of $|\nabla u|$ by constructing singular functions from nuclei of strain.  They also showed that blow-up rate $\va^{-1/2}$ is optimal in some two-dimensional cases. However, based on the classical elliptic theory, a natural question is whether it is possible to obtain gradient estimates of the solution of \eqref{lame-system-infty} under a weaker smoothness assumption on the inclusions, say, $C^{1,\gamma}$. This paper is a continuation of a paper by Chen, Li and  Xu \cite{clx}, where the perfect conductivity problem for scalar-valued $u$ with $C^{1,\gamma}$ inclusions was studied. Their proof makes use of the De Giorgi-Nash estimates for scalar elliptic equations and the Campanato's approach.  In this paper we extend it to the Lam\'{e} system for vector-valued $u$. Here one of the differences is that we make use of the  $W^{1,p}$ estimates for elliptic systems with right hand side in divergence form to replace the De Giorgi-Nash estimates used in \cite{clx} for the scalar case, and establish the pointwise upper bounds of the gradient. The other difference is that there are more free constants on the boundaries of inclusions to be handled, especially when we deal with the lower bound estimates of the gradient, even though we add certain symmetric assumption on the domain. Nevertheless, we further establish an asymptotic formula of the gradient near the origin after knowing that it is actually the blow-up location, from the upper and lower bound estimates obtained above.

Our strategy to establish the gradient estimates for the solution of  system of linear elasticity \eqref{lame-system-infty} in spirit follows \cite{bll,bll2}. Here we first point out that problem \eqref{lame-system-infty} has free boundary value feature. Although $e(u)=0$ implies $u$ is linear combination of $\psi^l$, 
\begin{equation*}
u=\sum_{l=1}^{d(d+1)/2}C_i^l \psi^l\quad \text{in}~~ \overline{D}_i \quad i=1, \, 2,
\end{equation*}
these $C_i^l$ are $d(d+1)$ free constants to be dealt with. This is the biggest difference with the perfect conductivity problem studied in \cite{bly1,clx}, where only two free constants need to be  handled for all dimensions. While, in linear elasticity case, how to determine these $d(d+1)$ constants $C_i^l$ is one of our main difficulties. To this end, using the continuity of $u$ across the boundaries of $D_{i}$, we decompose the solutions of \eqref{lame-system-infty}, as in \cite{bll}, as follows:
\begin{align}\label{decomposition-u}
u=\sum\limits_{l=1}^{\frac{d(d+1)}{2}}C_1^l v_1^l + \sum\limits_{l=1}^{\frac{d(d+1)}{2}}C_2^l v_2^l+v_0& \quad \text{in}~\Omega,
\end{align}
where $v_i^l \in{C}^{1}(\overline{\Omega}; \R^d)\cap C^2(\Omega; \R^d)$,  respectively, satisfies
\begin{equation}\label{equ_v1}
\begin{cases}
\mathcal{L}_{\lam, \mu}{v}_{i}^l=0& \mathrm{in}~\Omega,\\
v_{i}^l=\delta_{ij}\psi^l & \mathrm{on}~\partial{D}_{i},\\
v_{i}^l=0& \mathrm{on}~\partial{D},
\end{cases}\quad\mbox{for}~\,i,j=1, 2,~\,l=1,2,\cdots,\frac{d(d+1)}{2}, 
\end{equation}
and $v_0\in C^1(\overline{\Omega}; \R^d)\cap C^2(\Omega; \R^d)$ satisfies
\begin{equation}\label{equ_v3}
\begin{cases}
\mathcal{L}_{\lam, \mu}{v}_0=0& \mathrm{in}~\Omega,\\
v_0=0 & \mathrm{on}~\partial{D}_{1}\cup\ptl D_2,\\
v_0=\varphi & \mathrm{on}~\partial{D}.
\end{cases}
\end{equation}
 By the fourth line in \eqref{lame-system-infty} and decomposition \eqref{decomposition-u}, we have a linear system of these free constants $C_i^l$,
\begin{equation}\label{equ-decompositon}
\sum_{i=1}^2\sum\limits_{l=1}^{\frac{d(d+1)}{2}} C_i^l \int_{\ptl D_j} \frac{\ptl v_i^l}{\ptl \nu} \Big| _+\cdot \psi^k +\int_{\ptl D_j}\frac{\ptl v_0}{\ptl \nu}\Big|_+\cdot \psi^k =0,
\end{equation}
where $j=1,\, 2$, $k= 1, \, \cdots, \frac{d(d+1)}{2}$. But the coefficients here are all boundary integrals involving $\frac{\ptl v_i^l}{\ptl \nu}$ and $\frac{\ptl v_0}{\ptl \nu}$. To solve \eqref{equ-decompositon}, the hard work is to establish sufficiently good estimates of $\nabla v_i^l$ and $\nabla v_0$, rather than their $L^{\infty}$ estmates. However, for $C^{1,\gamma}$ inclusions, new difficulties need to be overcome to apply the iteration technique in \cite{bll,bll2} and obtain the local  energy estimates. Here, we turn to the Campanato's approach and the $W^{1, p}$ estimates, instead of the De Giorgi-Nash estimates  used in \cite{clx} for the scalar case. Finally we combine the estimates of $|\nabla v_i^l|$ and those of $C_i^l$ to show that the blow-up rate is $\va^{-1/(1+\gamma)}$ in $\R^2$, which is a little bigger than $\va^{-1/2}$ previously obtained in \cite{bll} for $C^{2,\gamma}$ inclusions.

Before we state our main results precisely, we first fix our domain and notation. Let $D_{1}^{0}$ and $D_{2}^{0}$ be a pair of (touching at the origin) convex subdomains of $D$, far away from the boundary $\partial D$, and satisfying
$$D_{1}^{0}\subset\{(x', x_{d})\in\mathbb R^{d}~:~ x_{d}>0\},\quad D_{2}^{0}\subset\{(x', x_{d})\in\mathbb R^{d}~:~ x_{d}<0\},$$
with $\{x_d=0\}$ being their common tangent plane, after a rotation of coordinates if necessary.
We use superscripts prime to denote ($d-1$)-dimensional variables and domains, such as $x'$ and $B'$. 
Translate $D_{i}^{0}$ ($i=1,\, 2$) by $\pm\frac{\va}{2}$ along $x_{d}$-axis in the following way
\begin{equation*}
D_{1}^{\va}:=D_{1}^{0}+(0',\frac{\va}{2})\quad \text{and}\quad D_{2}^{\va}:=D_{2}^{0}+(0',-\frac{\va}{2}).
\end{equation*}
For simplicity of notation, we drop the superscript $\va$ and denote
\begin{equation*}
D_{i}:=D_{i}^{\va}\, (i=1, \, 2), \quad  \Omega:=D\setminus\overline{D_1\cup D_2}.
\end{equation*}
So the distance between $D_1$ and $D_2$ is $\va$. Denote $P_1:= (0',\frac{\va}{2})$, $P_2:=(0',-\frac{\va}{2}) $ the two nearest points between $\ptl D_1$ and $\ptl D_2$ such that
\begin{equation*}
\text{dist}(P_1, P_2)=\text{dist}(\ptl D_1, \ptl D_2)=\va.
\end{equation*}
We here assume that $\ptl D_1$ and $\ptl D_2$ are of $C^{1, \gamma}$, $0<\gamma<1$ and there exists a constant $R_1$, independent of $\va$, such that the top and bottom boundaries of the narrow region between $\ptl D_1$  and $\ptl D_2$ can be represented, respectively, 
by graphs 
\begin{equation}\label{h1h2'}
x_d=\frac{\va}{2}+h_1(x')\quad\text{and}\quad x_d=-\frac{\va}{2}+h_2(x'),\quad \text{for}~ |x'|\leq 2R_1,
\end{equation}
where $h_1$, $h_2\in C^{1,\gamma}(B'_{2R_{1}}(0'))$ and satisfy 
\begin{equation}\label{h1-h2}
-\frac{\va}{2}+h_{2}(x') <\frac{\varepsilon}{2}+h_{1}(x'),\quad\mbox{for}~~ |x'|\leq 2R_1;
\end{equation}
\begin{equation}\label{h1h1}
h_{1}(0')=h_2(0')=0,\quad \nabla' h_{1}(0')=\nabla' h_2(0')=0;
\end{equation}
\begin{equation}\label{h1h_convex2}
\kappa_{0}|x'|^{\gamma}\leq|\nabla' h_{1}(x')|,|\nabla' h_2(x')|\leq\,\kappa_{1}|x'|^{\gamma},\quad\mbox{for}~~|x'|<2R_{1},
\end{equation}
and 
\begin{equation}\label{h1h3-1}
\|h_{1}\|_{C^{1,\gamma}(B'_{2R_{1}})}+\|h_{2}\|_{C^{1,\gamma}(B'_{2R_{1}})}\leq{\kappa_2},
\end{equation}
where the constants $0<\kappa_{0}<\kappa_{1}<\kappa_2$. Set
\begin{equation*}
\Omega_r:=\left\{(x',x_{d})\in \Omega~:~ -\frac{\va}{2}+h_2(x')<x_{d}<\frac{\va}{2}+h_1(x'),~|x'|<r\right\},
\end{equation*}
and assume that for some $\delta_0>0$,
\begin{equation}\label{coefficient-bound}
\delta_0\leq \mu,\, d\lam+2\mu\leq \frac{1}{\delta_0}.
\end{equation}
Then the first main result in this paper is as follows.
\begin{theorem}\label{thm-interior}
For $d=2$, let $D_1$, $D_2$ be two convex bounded $C^{1,\gamma}$ subdomains of $D$ with  $\va$ apart. Suppose \eqref{h1h2'}--\eqref{coefficient-bound} hold for $d=2$. Let $u\in H^1(D; \R^2)\cap C^1(\overline{\Omega}; \R^2)$ be the solution of \eqref{lame-system-infty} with $\varphi \in L^\infty(\ptl D; \R^2)$.  Then for sufficiently small $\va >0$,
\begin{equation}\label{upper-bounds}
|\nabla u(x_1,x_{2})| \leq \frac{C\va^{\frac{\gamma}{1+\gamma}}}{\va+|x_1|^{1+\gamma}}\cdot\|\varphi\|_{L^\infty(\ptl D; \R^2)}, \quad \text{for} ~ (x_1,x_{2})\in \Omega_{R_{1}},
\end{equation}
and
\begin{equation}\label{nabla-u-out}
\|\nabla u\|_{L^{\infty}(\Omega\setminus\Omega_{R_{1}})}\leq\,C\|\varphi\|_{L^\infty(\ptl D; \R^2)},
\end{equation}
where $C$ is a positive constant independent of $\va$.
\end{theorem}

Here and throughout this paper, unless otherwise stated, $C$ denotes a constant, whose value may vary from line to line, depending only on $n$, $\delta_0$, $\kappa_{0}$,$\ \kappa_1$, $ \kappa_2$, $ R_{1}$, and an upper bound of the $C^{1,\gamma}$ norms of $\partial{D}_{1}$ and $\partial{D}_2$, but not on $\varepsilon$. We call a constant having such dependence a {\it universal constant}. 

\begin{remark}
	From  pointwise upper bound estimate \eqref{upper-bounds}, we see that on the segment $\overline{P_1 P_2}$
	\begin{equation*}
	|\nabla u(0, x_2)|\Big|_{\overline{P_1 P_2}}\leq \frac{C}{\va^{1/(1+\gamma)}} \|\varphi\|_{L^\infty(\ptl D)}.
	\end{equation*}
\end{remark}
To show that this blow-up rate, $\va^{-1/(1+\gamma)}$, is optimal, we shall provide a lower bound of $|\nabla u(x)|$ on the segment $\overline{P_1 P_2}$, with the same blow-up rate, for some symmetric domains in dimension $d=2$. We assume that
\begin{equation}\label{assump-symm}
D_1\cup D_2~\mbox{and}~ D~\mbox{are~ symmetric~ with~ respect~ to~} x_{1} \mbox{-axis ~and~} x_{2}\mbox{-axis}.
\end{equation}
Furthermore, we assume that the boundary data $\varphi=(\varphi^{(1)}, \varphi^{(2)})^{\text{T}}\in L^\infty(\ptl D;\, \R^2)$ satisfies that $\varphi^{(1)}(x_{1},x_{2})$ is odd with respect to $x_{2}$ and $\varphi^{(2)}(x_{1},x_{2})$ is even with respect to $x_{2}$, that is, 
\begin{equation}\label{assup-varphi}
\left\{\begin{aligned}
\varphi^{(1)}(x_1, x_2)&=-\varphi^{(1)}(x_1, -x_2),\\
\varphi^{(2)}(x_1, x_2)&=\varphi^{(2)}(x_1, -x_2),
\end{aligned}\right. \qquad\mbox{for}~(x_1, x_2)\in\ptl D.
\end{equation}
For instance, we may take $\varphi=(x_2, x_1)$. We define a family of linear functionals of $\varphi$ by
\begin{equation}\label{def-blowup-factor}
\widetilde{b}_{*j}^{l}[\varphi]:=\int_{\ptl D_j^*} \frac{\ptl u^*}{\ptl \nu}\Big|_+ \cdot \psi^l, \quad  j=1, \, 2, ~~  l=1, \, 2, \, 3,
\end{equation}
where $u^*$ verifies  
\begin{equation}\label{equ-limit}
\left\{\begin{aligned}
& \mathcal{L}_{\lam, \mu} u^*=0, &\text{in}& ~~\Omega^*,\\
& u^*=\sum_{l=1}^3C_*^l \psi^l, &\text{on}& ~~\ptl D_1^*\cup\ptl D_2^*,\\
&\int_{\ptl D_1^*}\frac{\ptl u^*}{\ptl \nu}\Big|_+\cdot \psi^k+ \int_{\ptl D_2^*}\frac{\ptl u^*}{\ptl \nu}\Big|_+\cdot \psi^k=0, ~~~&k&=1, \, 2,\, 3,\\
&u^*=\varphi, &\text{on}&~~ \ptl D,
\end{aligned}
\right.
\end{equation}
and 
\begin{equation*}\label{def-domain}
D_i^*:=\{x\in \R^2 : x+P_i \in D_i\},\quad i=1, \ 2, \quad\Omega^*:=D\setminus \overline{D_1^*\cap D_2^*}.
\end{equation*}
Then we have

\begin{theorem}\label{thm-lower-bound}
For $d=2$, under the assumptions as in Theorem \ref{thm-interior}, with additional symmetric condition \eqref{assump-symm}, let $u\in H^1(D; \R^2)\cap C^1(\overline{\Omega}; \R^2)$ be the solution of \eqref{lame-system-infty}. If there exists $\varphi \in L^\infty(\ptl D;\, \R^2)$ satisfying \eqref{assup-varphi}
	 such that $\widetilde{b}_{*1}^{\,k_0}[\varphi] \neq 0$ for some $k_0\in \{1, \, 2\}$, then, for sufficiently small $\va>0$, 
	\begin{equation*}
|\nabla u(x)|\geq \frac{|\widetilde{b}_{*1}^{\,k_0}[\varphi] |}{C}\frac{1}{\va^{1/(1+\gamma)}}, \quad\text{for} ~~ x\in \overline{P_1 P_2},
	\end{equation*}
where $C$ is a universal constant.
\end{theorem}

\begin{remark}
We remark that assumption \eqref{h1h_convex2} can be replaced by a weaker assumption in the following
\begin{equation*}
\kappa_0 |x'|^{1+\gamma}\leq h_1(x')-h_2(x')\leq \kappa_1|x'|^{1+\gamma},~~\text{and}~ |\nabla' h_{i}(x')|\leq\,\kappa_2 |x'|^{\gamma},\quad\mbox{for}~~|x'|<2R_{1},
\end{equation*}	
for $i=1, \,2$,  $0<\gamma<1$, and $\kappa_j>0$, $j=0,\,1,\, 2$. We would like to point out that for $m$-convex inclusions (the power $1+\gamma$ above replaced by any positive integer $m\geq 2$), a relationship between the blow-up rate of gradient and the order of the relative convexity of inclusions was revealed in \cite{hjl}. Thus, the result in this paper for $1<m=1+\gamma<2$  in some sense can be regarded as a supplement to those in \cite{hjl}. While,  when $m=1$, for the Lipschitz inclusions, the corner singularity will be another interesting and challenge topic. For the scalar case, we refer to  Kozlov et al's book \cite{kmr} and Kang and Yun's recent work for bow-tie structure \cite{ky2}.
\end{remark}

Actually, the lower bound in Theorem \ref{thm-lower-bound} not only shows that the blow-up rate $\va^{-1/(1+\gamma)}$ is optimal but also allows us to identify that the blow-up location is at the origin. However, as a mater of fact, under the symmetric assumptions on the domain and the given boundary value, like in \eqref{assump-symm} and \eqref{assup-varphi}, we can obtain the asymptotic expansion of $\nabla u$ near the origin. Our method is totally different with that used in \cite{ky}, where the layer potential method is applied and the domain is assumed to be of $C^{3, 1}$. Let
\begin{equation}\label{def-wide-Q}
\widetilde{Q}_\gamma:=2\int_{0}^{+\infty}\frac{1}{1+ t^{1+\gamma}} dt,
\end{equation}	
and $E_{ij}$ denote the basic matrix with only one non-zero entry 1 in the $i\mbox{-th}$ row and $j\mbox{-th}$ column.
We have

\begin{theorem}\label{thm-asymptotics}
	For $d=2$, under the assumptions in Theorem \ref{thm-lower-bound}, and instead of \eqref{h1h1} and \eqref{h1h_convex2}, we assume that 
	\begin{equation}\label{asymp-symm-h1h2}
	h_1(x_1)=\frac{\kappa}{2}|x_1|^{1+\gamma}+O(|x_1|^{2+\gamma}), \quad h_2(x_1)=-h_1(x_1), \quad \text{for}~~ |x_1|\leq 2R_1,
	\end{equation}
	where $0<\gamma<1$ and $\kappa>0$ are independent of $\va$. Let $u\in H^1(D; \R^2)\cap C^1(\overline{\Omega}; \R^2)$ be the solution of \eqref{lame-system-infty}. If there exists $\varphi \in L^\infty(\ptl D;\, \R^2)$ satisfying \eqref{assup-varphi}, then, for sufficiently small $\va>0$,
	\begin{equation}\label{asymp-nabla}
	\begin{aligned}
	\nabla u(x)&=\frac{\va^{\frac{\gamma}{1+\gamma}}}{\va+2h_1(x_1)}\cdot \frac{\kappa^{\frac{1}{1+\gamma}}}{\widetilde{Q}_\gamma}\Big( \mathbb{B}_1[\varphi] +O(\va^{\frac{\gamma}{1+2\gamma}}) \Big) +O(1)\|\varphi\|_{L^\infty(\ptl D)}, 
	\end{aligned}
	\end{equation}
	for $x\in \Omega_{\va^{1/(1+\gamma)}}$, where
	\begin{equation*}
	\mathbb{B}_1[\varphi]:=\frac{\widetilde{b}_{*1}^{\, 1}[\varphi]}{\mu} E_{12}+ \frac{\widetilde{b}_{*1}^{\, 2}[\varphi]}{\lam+2\mu} E_{22},
	\end{equation*}
is a blow-up factor matrix.
\end{theorem}

Finally, following the arguments used in the proof of Theorem \ref{thm-interior}, we also have the pointwise upper bound estimates for higher dimensions $d\geq 3$. While, new technique is needed to get the corresponding lower bound estimate for these $C^{1, \, \gamma}$ inclusions.

\begin{theorem}\label{thm-d-geq-3}
For $d\geq3$, let $D_1$, $D_2$  be two convex bounded $C^{1,\gamma}$ subdomains of $D$ with  $\va$ apart. Suppose \eqref{h1h2'}--\eqref{coefficient-bound} hold for $d\geq 3$. Let $u\in H^1(D; \R^d)\cap C^1(\overline{\Omega}; \R^d)$ be the solution of \eqref{lame-system-infty} with $\varphi \in L^\infty(\ptl D; \R^d)$.  Then for sufficiently small $\va >0$,
	\begin{equation}\label{upper-bounds-d3}
	|\nabla u(x',x_{d})| \leq \frac{C}{\va+|x'|^{1+\gamma}}\cdot\|\varphi\|_{L^\infty(\ptl D; \R^d)}, \quad \text{for} ~ (x',x_{d})\in \Omega_{R_{1}},
	\end{equation}
	and
	$$\|\nabla u\|_{L^{\infty}(\Omega\setminus\Omega_{R_{1}})}\leq\,C\|\varphi\|_{L^\infty(\ptl D; \R^2)},$$
	where $C$ is a positive constant independent of $\va$.
\end{theorem}

This paper is organized as follows. In Section \ref{sec-2}, we first introduce a  scalar auxiliary function $\bar{u}$ to generate a family of vector-valued auxiliary functions, whose gradients will be proved to be the major singular terms of $|\nabla v_{i}^{l}|$, see Proposition \ref{prop-interior} below. Since $\partial D_{1}$ and $\partial D_{2}$ are of $C^{1,\gamma}$, in order to prove Proposition \ref{prop-interior}, we need to establish the $C^{1,\gamma}$ estimates and the $W^{1, p}$ estimates for elliptic systems with right hand side in divergence form, with partially zero Dirichlet boundary data, see Theorems \ref{lem-global-C1alp-estimates} and \ref{lem-lp-esti}.  Using them to replace the $W^{2,p}$ estimates argument used in \cite{bll}, we adapt the iteration process used in \cite{bll,bll2,llby} and obtain the estimates of $|\nabla v_i^l|$. The estimates in Proposition \ref{prop-interior} for dimension two are proved in Section \ref{sec3}, for higher dimensions in Section \ref{sec4}.
In section \ref{sub-lower-bounds}, we mainly deal with the convergence of the free constants $C_i^l$, and give  the proof of Theorem \ref{thm-lower-bound} to show the blow-up rate $\va^{-1/(1+\gamma)}$ is optimal. Section \ref{sec-asymp} is devoted to the proof of Theorem \ref{thm-asymptotics}, and the asymptotic expansions of $\nabla u$ near the origin is obtained. The proofs of  Theorems \ref{lem-global-C1alp-estimates} and \ref{lem-lp-esti} are provided  in Section \ref{sec-c1-alp} for readers' convenience. 

\vspace{0.5cm}

\section{Main ingredients for the proof of Theorems \ref{thm-interior} and \ref{thm-d-geq-3} }\label{sec-2}

In this section, we shall list some main ingredients to prove Theorems \ref{thm-interior} and \ref{thm-d-geq-3}. Recall that the linear space  of rigid displacement in $\R^d$, $\Psi$, is spanned by
\begin{equation*}
\Big\{e_i, \, x_je_k-x_ke_j~:~1\leq i \leq d, 1\leq j < k \leq d \Big\},
\end{equation*}
where $e_1, \, e_2, \cdots,\, e_d$ denote the standard basis of $\R^d$. 
By decomposition \eqref{decomposition-u}, we write
\begin{align}\label{decomposition_u2}
\nabla{u}&=\sum\limits_{l=1}^{\frac{d(d+1)}{2}} \big( C_1^l \nabla v_1^l + C_2^l \nabla v_2^l\big)+\nabla v_0 \notag  \\
&=\sum\limits_{l=1}^d(C_{1}^l-C_{2}^l)\nabla{v}_{1}^l+\sum\limits_{l=1}^d C_2^l \nabla(v_1^l+v_2^l) +\sum\limits_{l=d+1}^{\frac{d(d+1)}{2}}\sum\limits_{i=1}^2C_i^l\nabla{v}_{i}^l+\nabla v_0,\qquad~\text{in}~\Omega.
\end{align}
Thus, in order to prove Theorem \ref{thm-interior}, it suffices to estimate each term in \eqref{decomposition_u2}. Without loss of generality, we assume that $\|\varphi\|_{L^\infty(\ptl D)}=1$, by considering $u/\|\varphi\|_{L^\infty(\ptl D)}$ if $\|\varphi\|_{L^\infty(\ptl D)}>0$. If $\varphi|_{\ptl D}=0$, then $u \equiv 0$.
\subsection{Auxiliary functions}
To estimate $|\nabla v_i^l|$, $i=1, 2$, $l=1, 2, \cdots, \frac{d(d+1)}{2}$, we introduce a scalar auxiliary function $\bar{u}\in{C}^{1, \gamma}(\mathbb{R}^{d})$ as in \cite{bll}, such that $\bar{u}=1$ on $\partial{D}_{1}$, $\bar{u}=0$ on $\partial{D}_{2}\cup\partial D$,
\begin{align}\label{ubar}
\bar{u}(x)
=\frac{x_{d}-h_{2}(x')+\va/2}{\varepsilon+h_{1}(x')-h_{2}(x')},\ \ \qquad~ x\in\,\Omega_{2R_{1}},
\end{align}
here $h_1$, $h_2\in C^{1,\, \gamma}(B_{2R_1}(0'))$, and
\begin{equation}\label{ubar-out}
\|\bar{u}\|_{C^{1,\gamma}(\R^d\setminus \Omega_{R_{1}})}\leq\,C.
\end{equation}

 Denoting $\ptl_j:=\ptl/\ptl x_j$ and in view of \eqref{h1h1}, \eqref{h1h_convex2},
a direct calculation yields that
\begin{equation}\label{nablau_bar-interior}
\left|\partial_{j}\bar{u}(x)\right|\leq\frac{C|x'|^{\gamma}}{\varepsilon+|x'|^{1+\gamma}},\quad j=1, \cdots, d-1, \quad \partial_{d}\bar{u}(x)=\frac{1}{\delta(x')},  ~\mbox{for}~ x\in \Omega_{R_1},
\end{equation}
where
\begin{equation}\label{delta_x}
\delta(x'):=\varepsilon+h_1(x')-h_2(x').
\end{equation}
Use $\bar{u}$ to define a family of vector-valued auxiliary functions
\begin{equation}\label{def-function1}
\bar{u}_1^l=\bar{u}\psi^l, \quad l=1, \, 2,\cdots, \frac{d(d+1)}{2}, \quad \text{in} ~~ \Omega,
\end{equation}
then $\bar{u}_1^l=v_1^l$ on $\ptl \Omega$. Similarly, we define 
\begin{equation}\label{def-function2}
\bar{u}_2^l=\underline{u}\psi^l,\quad l=1, \, 2,\cdots, \frac{d(d+1)}{2}, \quad  \text{in} ~~ \Omega,
\end{equation}
where $\underline{u}\in C^{1,\gamma}(\R^d)$ satisfies $\underline{u}=1$ on $\ptl D_2$, $\underline{u}=0$ on $\ptl D_1\cup \ptl D$, 
\begin{equation*}
\underline{u}=1-\bar{u},\quad\mbox{in}~~ \Omega_{2R_1},
\end{equation*}
and $\|\underline{u}\|_{C^{1,\gamma}(\R^d\setminus \Omega_{R_1})}\leq C$. We shall prove that $|\nabla\bar{u}_i^l|$ are the main singular terms of $|\nabla v_i^l|$ near the origin.

\subsection{Estimates of $|\nabla v_i^l|$}
Set
$$w_i^l:=v_i^l-\bar{u}_{i}^l,\quad l=1,\,2,\cdots, \frac{d(d+1)}{2}, ~~ i=1,\,2.$$
Then
\begin{prop}\label{prop-interior}
Let $v_i^l,\, v_0\in  C^1(\overline{\Omega}; \R^d)$ be the solutions of \eqref{equ_v1} and \eqref{equ_v3}, respectively. Then under the hypotheses of Theorems \ref{thm-interior} and \ref{thm-d-geq-3}, for sufficiently small $\va >0$, we have  
\begin{itemize}
	\item[(i)] for $l=1,\,2, \cdots, d$, $i=1,\,2$, 
	\begin{equation}\label{nabla-w-i0}
	|\nabla w_i^l(x)|\leq \frac{C}{(\va+|x_1|^{1+\gamma})^{\frac{1}{1+\gamma}}},\quad  x\in \Omega_{ R_1}; 
	\end{equation}
consequently,
	\begin{equation}\label{nabla-v-i0}
	\frac{1}{C(\va +|x_1|^{1+\gamma})}\leq|\nabla v_{i}^l(x)|\leq \frac{C}{\va +|x_1|^{1+\gamma}},\quad  x\in \Omega_{ R_1};
	\end{equation}
	
	\item[(ii)] for $l=d+1, \cdots, \frac{d(d+1)}{2}$, $i=1,\,2$, 
	\begin{equation}\label{nabla-w-i3}
	|\nabla w_i^l (x)|\leq C,\quad  x\in \Omega_{ R_1};
	\end{equation}
consequently,
	\begin{equation}\label{nabla-v-i3}
	|\nabla v_{i}^l(x)|\leq \frac{C(\va +|x'|)}{\va +|x'|^{1+\gamma}},\quad  x\in \Omega_{ R_1};
	\end{equation}	
	\item[(iii)] for $l=1,\, 2, \cdots,\, \frac{d(d+1)}{2}$,
	\begin{equation}\label{v1+v2_bounded1}
	\|\nabla (v_{1}^l+ v_{2}^l)\|_{L^\infty(\Omega)}\leq C,\quad \text{and}\quad 	\|\nabla{v}_0\|_{L^\infty(\Omega)}\leq C\|\varphi\|_{L^{\infty}(\partial D)}; 
	\end{equation}
	\item[(iv)] 	for $l=1,\,2,\, \cdots, \frac{d(d+1)}{2}$, $i=1, \,2$,
	\begin{equation}\label{esti-infty-out-12}
	\|\nabla v_i^l\|_{L^\infty(\Omega\setminus \Omega_{R_1})} \leq C,
	\end{equation}
\end{itemize}
	where $C$ is a {\it universal constant}. 
\end{prop}

\begin{remark}
Since $h_{1}$ and $h_{2}$ here are of $C^{1,\gamma}$, now $\bar{u}$ and $\underline{u}$ are not twice continuously differentiable any more. Thus, $w_i^l$ satisfies the following system
$$-\mathcal{L}_{\lam,\mu} w_i^l=\nabla\cdot(\mathbb{C}^0e(\bar{u}_i^l)),$$
 with right hand side in divergence form.
Hence, we are not able to directly follow the iteration approach used in \cite{bll,llby} and to apply $W^{2,\, p }$ estimates to get the estimates of $w_i^l$. To overcome this difficulty, we here turn to the $C^{1, \gamma}$ estimates and the $W^{1,p}$ estimates  for elliptic system.
\end{remark}

 \subsection{$C^{1, \gamma}$ estimates and $W^{1, p}$ estimates}
We adapt the $C^{1,\gamma}$ estimates and the $W^{1, p}$ estimates in \cite{gm}  to our setting with partially zero boundary condition, which can be regarded as an analogue of theorem 9.13 in \cite{gt}  and are of independent interest. 

 Recall some properties of tensor $\mathbb{C}^0$.  For the isotropic elastic material, the components $C_{ijkl}$ satisfy symmetric condition:
\begin{equation*}
C_{ijkl}=C_{klij}=C_{klji},\quad i,\,j,\,k,\,l=1,\,2,\cdots, d.
\end{equation*} 
Thus, $\mathbb{C}^0$ satisfies the ellipticity condition: for every $d\times d$ real symmetric matrix $A=(A_{ij})$,
\begin{equation}\label{elliptic}
\min\{2\mu, \, d\lam +2\mu\}|A|^2\leq (\mathbb{C}^0 A, A)\leq \max\{2\mu, d\lam +2\mu\}|A|^2,
\end{equation}
where $|A|^2=\sum\limits_{i,j}A_{ij}^2$.

Let $Q$ be a bounded domain in $\R^d$, $d\geq 2$, with a $C^{1,\gamma}$ boundary portion $\Gamma \subset \ptl Q$. We consider the following boundary value problem
\begin{equation}\label{equ-w-divf-q1}
\left\{ \begin{aligned}
-\ptl_j\big(C_{ijkl}\ptl_l \widetilde{w}^{(k)} \big) &=\ptl_j \widetilde{f}_{ij} \quad &\text{in}\quad Q, \\
\widetilde{w}&=0 \quad &\text{on} \quad \Gamma,
\end{aligned}\right.
\end{equation}
where $\widetilde{f}_{ij} \in C^{\gamma}(Q)$,  $0<\gamma<1$, $i,\,j, \, k,\, l=1,\, 2,\, \cdots, d$, and the Einstein summation convention in repeated indices is used.
\begin{theorem}[$C^{1,\gamma}$ estimates]\label{lem-global-C1alp-estimates}
Let $\widetilde{w} \in{H}^{1}(Q; \R^d)\cap{C}^{1}(Q \cup \Gamma; \R^d)$ be the solution of \eqref{equ-w-divf-q1}.
Then for any domain $Q^\prime \subset\subset Q \cup \Gamma$,
\begin{equation}\label{ineq-global-C1alp-estimates}
\|   \widetilde{w} \|_{C^{1, \, \gamma}(Q^\prime)} \leq C\left( \|  \widetilde{w} \|_{L^\infty( Q)}+[\widetilde{F}]_{\gamma,\, Q}\right),
\end{equation}
where $\widetilde{F}:=( \widetilde{f}_{ij} )$ and $C=C(n, \gamma, Q^\prime, Q)$.
\end{theorem}
The H\"{o}lder semi-norm of matrix-valued function $\widetilde{F}=(\widetilde{f}_{ij})$ is defined as follows:
\begin{equation}\label{def-nablaU-alp}
[\widetilde{F}]_{\gamma,\, Q}:=\max\limits_{1\leq i,\, j \leq d}[\widetilde{f}_{ij}]_{\gamma, \, Q}\quad \text{and} \quad [\widetilde{f}_{ij}]_{\gamma, \, Q}:=\sup_{\substack{x, \, y\in Q \\ x\neq y}}\frac{|\widetilde{f}_{ij}(x)-\widetilde{f}_{ij}(y)|}{|x-y|^\gamma}.
\end{equation}

For the perfect conductivity problem with $C^{1, \gamma}$ inclusions, we adapted the famous De Giorgi-Nash approach or Moser's iteration to get the $L^\infty$ estimates of $u$ in \cite{clx}. But for Lam\'{e} system, these approaches are unable to be applied. Here, we need the following $W^{1, p}$  estimates for boundary value problem \eqref{equ-w-divf-q1}.

\begin{theorem}[$W^{1,p}$ estimates]\label{lem-lp-esti}
Let $Q$ and $\Gamma$ be defined as in Theorem \ref{lem-global-C1alp-estimates} for $d \geq 2$. Let $\widetilde{w}\in H^1(Q; \R^d)$ be the weak solution of \eqref{equ-w-divf-q1} with $\widetilde{f}_{ij}\in C^\gamma(Q)$, $0<\gamma<1$ and $i, \,j=1,\, 2,\cdots, d$. Then, for any $2\leq p<\infty$ and any domain $Q^\prime \subset\subset Q \cup \Gamma$, 
\begin{equation}\label{ineq-w1p}
\|\widetilde{w}\|_{W^{1, p}(Q^\prime)}\leq C(\|\widetilde{w}\|_{H^1(Q)}+\|\widetilde{F}\|_{L^p(Q)}),
\end{equation}
where $C=C(\lam, \mu, p, Q', Q)$ and $\widetilde{F}:=( \widetilde{f}_{ij} )$. In particular, if $p> d$, it holds that
\begin{equation}\label{ineq-c-alp}
\|\widetilde{w}\|_{C^{\tau}(Q')}\leq C(\|\widetilde{w}\|_{H^1(Q)}+[\widetilde{F}]_{\gamma, Q}),
\end{equation}
where $0<\tau\leq 1-d/p$ and $C=C(\lam, \mu, \tau, p, Q', Q)$.
\end{theorem}
For readers' convenience, the proofs of Theorem \ref{lem-global-C1alp-estimates} and Theorem \ref{lem-lp-esti} are given in Section \ref{sec-c1-alp}.

\subsection{Estimates of $C_i^l$}

For the estimates of $C_i^l$, we have 
\begin{prop}\label{lem-rest-terms} 
	Let $C_i^l$ be defined in \eqref{decomposition-u}. Then, (i)
	\begin{equation}\label{esti-C12}
	|C_i^l|\leq C, \quad \text{for} ~~ i=1, \, 2,~~~~ l=1, \, 2,\cdots, \frac{d(d+1)}{2}. 
	\end{equation}
	where $C$ is independent of $\va$. (ii) For $d=2$, 
	\begin{equation}\label{esti-c1-c2}
 |C_1^l-C_2^l|\leq  C\va^{\frac{\gamma}{1+\gamma}}, \quad \text{for} ~~ l=1,\, 2.
	\end{equation}
\end{prop}

\subsection{Proof of Theorem \ref{thm-interior} and Theorem \ref{thm-d-geq-3}}
We are in position to prove Theorem \ref{thm-interior} and  Theorem \ref{thm-d-geq-3} by using Proposition \ref{prop-interior} and Proposition \ref{lem-rest-terms}.

\begin{proof}[Proof of Theorem \ref{thm-interior}] 
For $d=2$, by using \eqref{decomposition_u2}, \eqref{nabla-v-i0}, \eqref{nabla-v-i3}, \eqref{v1+v2_bounded1}, \eqref{esti-C12} and \eqref{esti-c1-c2},  one has, for $x\in \Omega_{R_1}$, 
\begin{equation*}
|\nabla u(x)|\leq \sum\limits_{l=1}^2|C_1^l-C_2^l| |\nabla v_1^l(x)|+C\sum\limits_{i=1}^2|\nabla v_i^3(x)|+C\leq  \frac{C\va^{\frac{\gamma}{1+\gamma}}}{\va+|x_1|^{1+\gamma}}.
\end{equation*}
Thus, \eqref{upper-bounds} is proved. \eqref{nabla-u-out} follows from \eqref{esti-infty-out-12} and \eqref{esti-C12}.
	\end{proof}

\begin{proof}[Proof of Theorem \ref{thm-d-geq-3}] For $d\geq 3$, by using \eqref{decomposition_u2}, \eqref{nabla-v-i0}, \eqref{nabla-v-i3}, \eqref{v1+v2_bounded1} and \eqref{esti-C12}, one has, for $x\in \Omega_{R_1}$,
	\begin{equation*}
	|\nabla u(x)|\leq \sum\limits_{i=1}^2\sum_{l=1}^{\frac{d(d+1)}{2}}|C_i^l||\nabla v_i^l(x)| +|\nabla v_0(x)| \leq \frac{C}{\va+|x'|^{1+\gamma}}.
	\end{equation*}
The proof of \eqref{upper-bounds-d3} is completed.	
\end{proof}

\vspace{0.5cm}

\section{ Estimates for Theorem \ref{thm-interior}}\label{sec3}

 In this section, we prove Proposition \ref{prop-interior} in Subsection \ref{sub-3.1}, and prove Proposition \ref{lem-rest-terms} in Subsection \ref{prop-2.4}. Finally, some remarks on the lower bound are given in Subsection \ref{sec-lower-bound}.

Since $\ptl D_1$ and $\ptl D_2$ are of $C^{1, \gamma}$, we need to adapt the iteration technique developed in \cite{bll,llby} to our setting to apply Theorem \ref{lem-global-C1alp-estimates}. For fixed $|z_1|\leq R_1<1$, we define
\begin{equation*}
\widehat{\Omega}_{s}(z_1):=\left\{x\in \mathbb{R}^{2}:-\frac{\va}{2}+h_{2}(x_1)<x_{2}<\frac{\va}{2}+h_{1}(x_1),~|x_1-z_1|<{s}~\right\},
\end{equation*}
for $0<s< R_1$. In the following, we always assume that $\varepsilon$ and $|z_{1}|$ are sufficiently small. 

The following  estimate for the H\"{o}lder semi-norm of $\nabla \bar{u}$ is needed: For a fixed constant $s\leq C\delta(z_1)$, 
\begin{equation}\label{ineq-semi-holder-norm}
[\nabla \bar{u}]_{\gamma, \, \widehat{\Omega}_{s}(z_1) }\leq C\delta(z_1)^{-1-\frac{1}{1+\gamma}}s^{1-\gamma}+C\delta(z_1)^{-1-\frac{\gamma^2}{1+\gamma}},
\end{equation}
where $\delta(z_1)=\va+h_1(z_1)-h_2(z_1)$, defined by \eqref{delta_x}.

Indeed, since
\begin{equation}\label{sim-norm}
[\nabla \bar{u}]_{\gamma, \, \widehat{\Omega}_{s}(z_1)}\leq [\ptl_1 \bar{u}]_{\gamma, \, \widehat{\Omega}_{s}(z_1)}+[\ptl_2 \bar{u}]_{\gamma, \, \widehat{\Omega}_{s}(z_1)},
\end{equation}
we will estimate the H\"{o}lder semi-norms $[\ptl_1 \bar{u}]_{\gamma, \, \widehat{\Omega}_{s}(z_1)}$ and $[\ptl_2 \bar{u}]_{\gamma, \, \widehat{\Omega}_{s}(z_1)}$ in $\widehat{\Omega}_{s}(z_1)$ in the following, one by one. First, since 
$$s\leq C\delta(z_1),\quad\mbox{and}\quad |z_1|\leq C\delta(z_1)^{1/(1+\gamma)},$$ it follows that, for any $x=(x_1, x_2)\in \widehat{\Omega}_{s}(z_1)$, 
\begin{equation}\label{ineq-x-prime}
|x_1|\leq |x_1-z_1|+|z_1|<s+|z_1|\leq C\delta(z_1)^{1/(1+\gamma)}.
\end{equation}
By using mean value theorem and \eqref{h1h_convex2}, we have, for any $x, \tilde{x}\in \widehat{\Omega}_{s}(z_1)$ with $x_1\neq \tilde{x}_1$, 
 \begin{equation}\label{ineq-hi-meanvalue}
\begin{aligned}
|h_i(x_1)-h_i(\tilde{x}_1)|&=| h_{i}'(x_{1\theta_{i}})|\cdot|x_1-\tilde{x}_1|\leq C\delta(z_1)^{\frac{\gamma}{1+\gamma}}\cdot|x_1-\tilde{x}_1|, \quad i=1,\,2.
\end{aligned}
 \end{equation}
 So that
 \begin{equation}\label{ineq-meanvalue}
 |\delta(x_1)-\delta(\tilde{x}_1)|\leq\sum_{i=1}^{2}|h_i(x_1)-h_i(\tilde{x}_1)|\leq C\delta(z_1)^{\frac{\gamma}{1+\gamma}}\cdot|x_1-\tilde{x}_1|.
 \end{equation}
In particular, taking $\tilde{x}_1=z_1$, and recalling that $|x_1-z_1|\leq s\leq C\delta(z_1)$, we have 
 \begin{equation}\label{ineq-meanvalue-leq}
\delta(x_1)\leq \delta(z_1)+|h_1(x_1)-h_1(z_1)|+|h_2(x_1)-h_2(z_1)|\leq C\delta(z_1),
\end{equation}
and for sufficiently small $\va$ and $|z_1|$, 
 \begin{align}\label{ineq-meanvalue-1}
\delta(x_1)&\geq\delta(z_1)-|h_1(x_1)-h_1(z_1)|-|h_2(x_1)-h_2(z_1)|\geq\frac{1}{2}\delta(z_1).
 \end{align}
 
Recalling
\begin{equation*}
\ptl_2\bar{u}(x)=\frac{1}{\delta(x_1)},
\end{equation*}
 it follows from \eqref{ineq-meanvalue} and \eqref{ineq-meanvalue-1} that
\begin{align}\label{delta-semi}
|\partial_{2}\bar{u}(x)-\partial_{2}\bar{u}(\tilde{x})|&=\Big|\frac{1}{\delta(x_1)}-\frac{1}{\delta(\tilde{x}_1)}\Big|=\frac{|\delta(x_1)-\delta(\tilde{x}_1)|}{\delta(x_1)\delta(\tilde{x}_1)} \leq C\delta(z_1)^{-1-\frac{1}{1+\gamma}}\cdot|x_1-\tilde{x}_1|,
\end{align}
Thus, using $|x_1-\tilde{x}_1|\leq |x_1-z_1|+|z_1-\tilde{x}_1|\leq 2s$, we have
\begin{equation}\label{ineq-xn-alp}
[\ptl_2 \bar{u}]_{\gamma,\, \widehat{\Omega}_{s}(z_1)}:=\sup\limits_{\substack{x,\,\tilde{x}\in \widehat{\Omega}_{s}(z_1) \\ x_1\neq\tilde{x}_1}}\frac{|\partial_{2}\bar{u}(x)-\partial_{2}\bar{u}(\tilde{x})|}{|x_1-\tilde{x}_1|^\gamma}\leq C\delta(z_1)^{-1-\frac{1}{1+\gamma}}s^{1-\gamma}.
\end{equation}

While, 
\begin{align*}
\ptl_1\bar{u}(x)& =\frac{-h_2'(x_1)}{\delta(x_1)}+\frac{-x_2 \delta'(x_1)}{\delta(x_1)^2}+\frac{\big(h_2(x_1)-\va/2 \big)\delta'(x_1)}{\delta(x_1)^2}\\
&=:\Phi_1(x)+\Phi_2(x)+\Phi_3(x).
\end{align*}
Since $h_1, h_2 \in C^{1, \,\gamma}$, and by virtue of \eqref{delta-semi}, we have
\begin{align*}
|\Phi_1(x)-\Phi_1(\tilde{x})|&\leq \frac{|h_2'(x_1)-h_2'(\tilde{x}_1)|}{\delta(x_1)}+|h_2'(\tilde{x}_1)|\cdot\Big|\frac{1}{\delta(x_1)}-\frac{1}{\delta(\tilde{x}_1)}\Big|\\
&\leq C\Big( \delta(z_1)^{-1}\cdot |x_1-\tilde{x}_1|^\gamma+\delta(z_1)^{-\frac{2}{1+\gamma}}\cdot|x_1-\tilde{x}_1|\Big).
\end{align*}
In view of $\va/2 -h_2(\tilde{x}_1)\leq \delta(\tilde{x}_1)$, we have
\begin{align*}
|\Phi_3(x)-\Phi_3(\tilde{x})|&\leq  \frac{|\delta'(x_1)|}{\delta(x_1)^2}\cdot|h_2(x_1)-h_2(\tilde{x}_1)|+\frac{\delta(\tilde{x}_1)}{\delta(x_1)^2}\cdot|\delta'(x_1)-\delta'(\tilde{x}_1)|\\
&\quad +\delta(\tilde{x}_1)\cdot|\delta'(\tilde{x}_1)|\cdot \Big|\frac{1}{\delta(x_1)^2}-\frac{1}{\delta(\tilde{x}_1)^2}\Big|\\
&\leq C\Big(\delta(z_1)^{-\frac{2}{1+\gamma}}\cdot|x_1-\tilde{x}_1|+ \delta(z_1)^{-1}\cdot|x_1-\tilde{x}_1|^\gamma\Big).
\end{align*}
Thus, for any $x,\,\tilde{x}\in \widehat{\Omega}_{s}(z_1) $ with $x_{1}\neq\tilde{x}_{1}$, 
\begin{equation}\label{ineq-phi-12}
\frac{|\Phi_i(x)-\Phi_i(\tilde{x})|}{|x-\tilde{x}|^\gamma} \leq C \Big( \delta(z_1)^{-1}+\delta(z_1)^{-\frac{2}{1+\gamma}} s^{1-\gamma}\Big),\quad i=1, \, 3.
\end{equation}

Recalling \eqref{ineq-meanvalue-leq}, we know
\begin{equation*}
|\tilde{x}_2|\leq \delta(\tilde{x}_1)\leq C\delta(z_1), \quad\mbox{
in}\quad \widehat{\Omega}_{s}(z_1).
\end{equation*}
So that
\begin{align*}
|\Phi_2(x)-\Phi_2(\tilde{x})|&\leq \frac{|\delta'(x_1)|}{\delta(x_1)^2}\cdot|x_2-\tilde{x}_2|+\frac{|\tilde{x}_2|}{\delta(x_1)^2}\cdot\big|\delta'(x_1)-\delta'(\tilde{x}_1)\big|\\
&\leq C \Big(\delta(z_1)^{-1-\frac{1}{1+\gamma}}\cdot|x_2-\tilde{x}_2|+ \delta(z_1)^{-1}\cdot|x_1-\tilde{x}_1|^\gamma\Big).
\end{align*}
Using
\begin{equation*}\label{ineq-x2-x2}
|x_2-\tilde{x}_2|\leq |x_2|+|\tilde{x}_2|\leq C\delta(z_1),
\end{equation*}
we have,  for any $x,\,\tilde{x}\in \widehat{\Omega}_{s}(z_1) $ with $x\neq\tilde{x}$, 
\begin{align}\label{ineq-phi-3}
\frac{|\Phi_2(x)-\Phi_2(\tilde{x})|}{|x-\tilde{x}|^\gamma} &\leq C \Big(\delta(z_1)^{-1-\frac{1}{1+\gamma}}\cdot|x_2-\tilde{x}_2|^{1-\gamma}+\delta(z_1)^{-\frac{2}{1+\gamma}} \cdot|x_1-\tilde{x}_1|^{1-\gamma}\Big)\notag \\
&\leq C\Big( \delta(z_1)^{-\gamma-\frac{1}{1+\gamma}}+\delta(z_1)^{-\frac{2}{1+\gamma}} s^{1-\gamma}\Big).
\end{align}

Combining \eqref{ineq-phi-12}, \eqref{ineq-phi-3}, and using $\gamma+\frac{1}{1+\gamma}=1+\frac{\gamma^2}{1+\gamma}>1$, we obtain
\begin{align}\label{ineq-xi-alp}
[\ptl_1\bar{u}]_{\gamma, \, \widehat{\Omega}_{s}(z_1)}&\leq \sum_{j=1}^3 \sup\limits_{\substack{x,\,\tilde{x}\in \widehat{\Omega}_{s}(z_1) \\ x\neq\tilde{x}}} \frac{|\Phi_j(x)-\Phi_j(\tilde{x})|}{|x-\tilde{x}|^\gamma} \notag \\
&\leq C\Big(\delta(z_1)^{-1-\frac{\gamma^2}{1+\gamma}}+\delta(z_1)^{-\frac{2}{1+\gamma}}s^{1-\gamma}\Big).
\end{align}
Thus, \eqref{ineq-semi-holder-norm} follows immediately  from \eqref{ineq-xn-alp} and \eqref{ineq-xi-alp}.

\subsection{Proof of Proposition \ref{prop-interior} in dimension $d=2$} \label{sub-3.1}

\begin{proof}[Proof of \eqref{nabla-w-i0} and \eqref{nabla-v-i0} when $d=2$] We only prove the estimate for the case that $i=l=1$, that is, $\nabla v_{1}^{1}$, for instance, since the proof is the same for the other cases.
For simplicity, we denote $w:=w_1^1$. Then $w$ satisfies
\begin{equation}\label{w20'}
\left\{
\begin{aligned}
-\mathcal{L}_{\lam,\mu} w&=\nabla\cdot(\mathbb{C}^0e(\bar{u}_1^1)) \quad&\mbox{in}~\Omega,\\
w&=0\quad&\mbox{on}~\partial \Omega.
\end{aligned}\right.
\end{equation}
Clearly, $w$ still satisfies 
\begin{equation}\label{equ-w-matrix}
-\mathcal{L}_{\lam,\mu} w=\nabla\cdot(\mathbb{C}^0e(\bar{u}_1^1)-\mathcal{M}) \quad\mbox{in}~\Omega,
\end{equation}
for any constant matrix $\mathcal{M}=(\mathfrak{a}_{ij})$, $i,\,j=1,\, 2$. We will mainly deal with this main difference with that in \cite{bll}. We modify the proof in \cite{bll} and divide it into three steps.

{\bf STEP 1. The boundedness of the total energy:} 
\begin{equation}\label{energy-nabla-w11}
\int_{\Omega}|\nabla {w}|^2\ dx\leq C.
\end{equation}

Indeed, we multiply \eqref{w20'} by $w$, making use of the integration by parts, and obtain
\begin{equation}\label{equal-integ-ptl}
\int_{\Omega}\Big(\mathbb{C}^0e(w), e(w)\Big)dx =\int_{\Omega}\big( \nabla\cdot(\mathbb{C}^0e(\bar{u}_1^1))\big)\cdot w \ dx.
\end{equation}
For the left hand side of \eqref{equal-integ-ptl}, it follows from \eqref{elliptic} and the first Korn's inequality that 
\begin{equation}\label{ineq-kohn-1}
\int_{\Omega}\Big(\mathbb{C}^0e(w), e(w)\Big)dx\geq C  \int_{\Omega} |e(w)|^2 \ dx \geq C \int_{\Omega}|\nabla w|^2 \ dx.
\end{equation}
For the right hand side of \eqref{equal-integ-ptl}, we rewrite it as
\begin{align}\label{ineq-div-nabla-u}
&\int_{\Omega}\big( \nabla\cdot(\mathbb{C}^0e(\bar{u}_1^1))\big)\cdot w \ dx\notag \\
&= \int_{\Omega} \mu\Div (\nabla \bar{u}) w^{(1)}  +(\lam+\mu)\big(\ptl_1(\ptl_1 \bar{u})w^{(1)}+ \ptl_2(\ptl_1 \bar{u}) w^{(2)} \big) \ dx.
\end{align}

First, by the mean value theorem, we fix a $r_0\in (\frac{R_1}{2}, \frac{2R_1}{3})$ such that
\begin{align}\label{ineq-mean-value}
\int\limits_{\substack{ |x_1|=r_0 \\ -\frac{\va}{2}+h_2(x_1)<x_2<\frac{\va}{2}+h_1(x_1)}}|w| dx_2 &=\frac{6}{R_1}\int\limits_{\substack{R_1/2<|x_1|<2R_1/3 \\ -\frac{\va}{2}+h_2(x_1)<x_2<\frac{\va}{2}+h_1(x_1)}}|w| dx\notag \\
&\leq C\int_{\Omega_{2R_1/3}\setminus \Omega_{R_1/2}}|\nabla w| \ dx\notag \\
& \leq C\Big(\int_{\Omega}|\nabla w|^2 \ dx\Big)^\frac{1}{2}.
\end{align}
For the first term on the right hand side of \eqref{ineq-div-nabla-u}, noticing that $\ptl_{22} \bar{u}=0$ in $\Omega_{R_1}$ and \eqref{ubar-out}, one has
\begin{align}\label{ineq-div-nabla-u-2}
\Big|\int_{\Omega} \Div (\nabla \bar{u}) w^{(1)}\ dx\Big|\leq \Big|\int_{\Omega_{r_0}}\ptl_1(\ptl_1\bar{u} ) w^{(1)} \ dx\Big|+\Big|\int_{\Omega\setminus \Omega_{r_0}} \Div (\nabla \bar{u}) w^{(1)} \ dx\Big|. 
\end{align}
For the first term on the right hand side of \eqref{ineq-div-nabla-u-2}, by using the integration by parts,
\begin{align*}
\int_{\Omega_{r_0}}\ptl_1(\ptl_1\bar{u} ) w^{(1)} \ dx&=-\int_{\Omega_{r_0}} \ptl_1\bar{u} \ptl_1 w^{(1)} \ dx +\int\limits_{\substack{|x_1|=r_0\\ -\frac{\va}{2}+h_2(x_1)<x_2<\frac{\va}{2}+h_1(x_1) }} w^{(1)}\ptl_1\bar{u}  \ dx_2\\
&=:\mathrm{I}_1 + \mathrm{I}_2.
\end{align*}
 In view of \eqref{nablau_bar-interior} and \eqref{ineq-mean-value}, we have 
\begin{equation*}
|\mathrm{I}_1| \leq C\Big(\int_{\Omega_{r_0}}|\ptl_1\bar{u}|^2 \ dx\Big)^{\frac{1}{2}}\Big(\int_{\Omega}|\nabla w|^2 \ dx\Big)^{\frac{1}{2}}\leq C \Big(\int_{\Omega}|\nabla w|^2 \ dx\Big)^{\frac{1}{2}},
\end{equation*}
and
\begin{equation*}
|\mathrm{I}_2|\leq C \int\limits_{\substack{ |x_1|=r_0\\ -\frac{\va}{2}+h_2(x_1)<x_2<\frac{\va}{2}+h_1(x_1) }}|w| dx_2\leq C \Big(\int_{\Omega}|\nabla w|^2 \ dx\Big)^{\frac{1}{2}}.
\end{equation*}
For the second term on the right hand side of \eqref{ineq-div-nabla-u-2}, we have
\begin{equation*}
\Big|\int_{\Omega\setminus \Omega_{r_0}} \Div (\nabla \bar{u}) w^{(1)} \ dx\Big|\leq C\Big(\int_{\Omega} |\nabla w|^2 \ dx\Big)^{\frac{1}{2}}.
\end{equation*} 
For the last term on the right hand side of \eqref{ineq-div-nabla-u}, it follows from the integration by parts with $w=0$ on $\partial{D}_{1}\cup\partial{D}_{2}$ and the H\"{o}lder inequality that
\begin{equation*}\label{ineq-integral-holder}
\Big|\int_{\Omega} \ptl_2(\ptl_1 \bar{u}) w^{(2)} \ dx\Big|  \leq C\Big(\int_{\Omega}|\ptl_1 \bar{u}|^2 \ dx \Big)^{\frac{1}{2}}\Big(\int_{\Omega} |\nabla w|^2 \ dx\Big)^{\frac{1}{2}}\leq C \Big(\int_{\Omega}|\nabla w|^2 \ dx\Big)^{\frac{1}{2}}.
\end{equation*}
These, combining with \eqref{equal-integ-ptl}--\eqref{ineq-div-nabla-u}, yield
\begin{equation*}
\int_{\Omega} |\nabla w|^2 dx \leq C\Big(\int_{\Omega}|\nabla w|^2 \ dx\Big)^{\frac{1}{2}}.
\end{equation*}
So that \eqref{energy-nabla-w11} is proved.

{\bf STEP 2. The local energy estimates:}
\begin{equation}\label{energy_w11_inomega_z1}
\int_{\widehat{\Omega}_{\delta(z_1)}(z_1)}\left|\nabla{w}\right|^{2}\ dx\leq
C\delta(z_1)^{\frac{2\gamma}{1+\gamma}},
\end{equation}
where $\delta(z_1)=\va+h_1(z_1)-h_2(z_1)$.

Indeed, for $0<t<s<R_{1}$, let $\eta$ be a cut-off function satisfying 
\begin{align*}\label{def-eta-x-prime}
\eta(x_1)=\begin{cases}
1 & \text{if} ~ |x_1-z_1|<t, \\
0 & \text{if} ~ |x_1-z_1|>s,\\
\end{cases}
\quad\text{and}\quad |\eta'(x_1)|\leq\frac{2}{s-t}.
\end{align*}
Multiplying  \eqref{equ-w-matrix} by $\eta^{2}w$ and using the integration by parts, one has
\begin{equation}\label{equal-integrate-parts}
\int_{\widehat{\Omega}_{s}(z_1)}\Big(\mathbb{C}^0e(w), e(\eta^2w)\Big) \ dx =- \int_{\widehat{\Omega}_{s}(z_1)} \Big(\mathbb{C}^0e(\bar{u}_1^1)-\mathcal{M}, \nabla (\eta^2w)\Big) \ dx.
\end{equation}
For the left hand side of \eqref{equal-integrate-parts}, using the first Korn's inequality and standard arguments, one has
\begin{equation*}
\int_{\widehat{\Omega}_{s}(z_1)} \Big(\mathbb{C}^0e(w), e(\eta^2w)\Big) \ dx\geq \frac{1}{C} \int_{\widehat{\Omega}_{s}(z_1)} |\nabla ( \eta w)|^2 dx- C\int_{\widehat{\Omega}_{s}(z_1)} |w|^2| \eta '|^2 \ dx.
\end{equation*}
For the right hand side of \eqref{equal-integrate-parts}, using the Young's inequality, we have for any $\zeta>0$,
\begin{align*}
&\Big|\int_{\widehat{\Omega}_{s}(z_1)}  \Big(\mathbb{C}^0e(\bar{u}_1^1)-\mathcal{M}, \nabla (\eta^2w)\Big) \ dx\Big| \\
&\leq \zeta\int_{\widehat{\Omega}_{s}(z_1)} \eta^2|\nabla w|^2 \ dx\quad +\frac{C}{\zeta} \int_{\widehat{\Omega}_{s}(z_1)} |\eta'|^2|w|^2 dx +\frac{C}{\zeta}\int_{\widehat{\Omega}_{s}(z_1)} |\mathbb{C}^0e(\bar{u}_1^1)-\mathcal{M}|^2 \ dx.
\end{align*}
It follows that
\begin{equation}\label{ineq-iteration-w1}
\int_{\widehat{\Omega}_{t}(z_1)}|\nabla w|^{2}\ dx \leq\,\frac{C}{(s-t)^{2}}\int_{\widehat{\Omega}_{s}(z_1)}|w|^{2}\ dx +C\int_{\widehat{\Omega}_{s}(z_1)} |\mathbb{C}^0e(\bar{u}_1^1)-\mathcal{M}|^2 \ dx.
\end{equation}
Noticing
\begin{equation*}
\mathbb{C}^0e(\bar{u}_1^1)=\begin{pmatrix}
~(\lam+2\mu)\ptl_1\bar{u} &  \mu \ptl_2\bar{u} ~\\
~~\\
~\mu \ptl_2\bar{u} & \lam \ptl_1\bar{u}~
\end{pmatrix},
\end{equation*}
then in \eqref{ineq-iteration-w1} we take the  constant matrix $\mathcal{M}=\mathcal{M}_1=(\mathfrak{a}_{ij})$, defined by
\[\mathcal{M}_1:=\pvint_{\widehat{\Omega}_{s}(z_1)} \mathbb{C}^0e(\bar{u}_1^1(y)) \ dy:=\frac{1}{|\widehat{\Omega}_{s}(z_1)|}\int_{\widehat{\Omega}_{s}(z_1)} \mathbb{C}^0e(\bar{u}_1^1(y)) \ dy .\]

{\bf Case 1.} For $|z_1|\leq \va^{\frac{1}{1+\gamma}}$ and $0<s<\va^{\frac{1}{1+\gamma}}$, we have $\va\leq\delta(z_1)\leq\,C\varepsilon$. By a direct calculation, we have
\begin{align}\label{energy_nabla_w11_square_in}
\int_{\widehat{\Omega}_{s}(z_1)}|w|^{2} \ dx &=\int_{|x_1-z_1|<s}\int_{-\va/2+h_2(x_1)}^{\va/2+h_1(x_1)}\left(\int_{-\frac{\va}{2}+h_2(x_1)}^{x_{2}}\partial_{2}w \ dx_{2}\right)^{2}\ dx_{2}dx_1\nonumber\\
&\leq C\varepsilon^{2}\int_{\widehat{\Omega}_{s}(z_1)}|\nabla{w}|^{2} \ dx.
\end{align}
On the other hand, recalling the definition of semi-norm $[\cdot]_{\gamma, \, \widehat{\Omega}_{s}(z_1)}$ in \eqref{def-nablaU-alp},
\begin{align*}
|\mathbb{C}^0e(\bar{u}_1^1)-\mathcal{M}_1|^2&\leq|(2\mu+\lam)\ptl_1 \bar{u}-\mathfrak{a}_{11}|^2+2|\mu \ptl_2 \bar{u} -\mathfrak{a}_{12}|^2+|\lam\ptl_1 \bar{u} -\mathfrak{a}_{22}|^2\\
&\leq \frac{C}{|\widehat{\Omega}_{s}(z_1)|^2}\Big(\int_{\widehat{\Omega}_{s}(z_1)} \big( |\ptl_1  \bar{u}(x)-\ptl_1 \bar{u}(y)| +|\ptl_2\bar{u}(x)-\ptl_2\bar{u}(y)| \big)\ dy\Big)^2  \\
&\leq \frac{C[\nabla \bar{u}]_{\gamma, \, \widehat{\Omega}_{s}(z_1)}^2}{|\widehat{\Omega}_{s}(z_1)|^2}\Big(\int_{\widehat{\Omega}_{s}(z_1)} |x-y|^{\gamma} \ dy\Big)^2 \\
& \leq C [\nabla \bar{u}]_{\gamma, \, \widehat{\Omega}_{s}(z_1)}^2 \Big(s^{2\gamma}+\delta(z_1)^{2\gamma}\Big).
\end{align*}
Using \eqref{ineq-semi-holder-norm} and by a direct calculation,
\begin{align}\label{ineq-s-va}
&\int_{\widehat{\Omega}_{s}(z_1)} |\mathbb{C}^0e(\bar{u}_1^1)-\mathcal{M}_1|^2 \ dx \notag \\
&\leq  \int_{\widehat{\Omega}_{s}(z_1)} \big(|(2\mu+\lam)\ptl_1 \bar{u}-\mathfrak{a}_{11}|^2+2|\mu \ptl_2 \bar{u} -\mathfrak{a}_{12}|^2+|\lam\ptl_1 \bar{u} -\mathfrak{a}_{22}|^2 \big) \ dx\notag  \\
& \leq C [\nabla \bar{u}]_{\gamma, \, \widehat{\Omega}_{s}(z_1)}^2  \int_{\widehat{\Omega}_{s}(z_1)}  (s^{2\gamma}+\delta(z_1)^{2\gamma}) \ dx \notag \\
 &\leq C\Big(\frac{s^3}{\va^{1+\frac{2}{1+\gamma}}}+\frac{s}{\va^{\frac{2}{1+\gamma}-1}}
+\frac{s^{3-2\gamma}}{\va^{1+\frac{2}{1+\gamma}-2\gamma}}+\frac{s^{1+2\gamma}}{\va^{2\gamma+\frac{2}{1+\gamma}-1}}\Big)=:G_1(s).
\end{align}
It follows from \eqref{ineq-iteration-w1}--\eqref{ineq-s-va}  that 
\begin{equation}\label{ineq-F111_in}
F(t)\leq\,\left(\frac{c_{1}\varepsilon}{s-t}\right)^{2}F(s) +CG_1(s),\quad \forall~ 0<t<s<\va^{\frac{1}{1+\gamma}},
\end{equation}
here we fix the constant $c_1$, and denote
\begin{equation*}
F(t):=\int_{\widehat{\Omega}_{t}(z_1)}|\nabla{w}|^{2} \ dx.
\end{equation*}

Similarly as in \cite{bll}, let $k=\left[\frac{1}{4c_{1}\varepsilon^{\frac{\gamma}{1+\gamma}}}\right]$ and $t_{i}=\delta(z_1)+2c_{1}i\varepsilon$, $i=0, 1, 2, \cdots, k$. It is easy to see from the definition of $G_{1}(s)$ in \eqref{ineq-s-va} that
\begin{align*}
G_1(t_{i+1})\leq C(i+1)^{3}\va^{\frac{2\gamma}{1+\gamma}}.
\end{align*}
Taking $s=t_{i+1}$ and $t=t_{i}$ in \eqref{ineq-F111_in}, we have the following iteration formula
$$F(t_{i})\leq\,\frac{1}{4}F(t_{i+1}) +C(i+1)^{3}\va^{\frac{2\gamma}{1+\gamma}}.$$
After $k$ iterations, and by virtue of \eqref{energy-nabla-w11}, we have
\[F(t_{0})\leq (\frac{1}{4})^{k}F(t_{k})+C\va^{\frac{2\gamma}{1+\gamma}}\sum_{i=0}^{k-1}(\frac{1}{4})^i(i+1)^{3} \leq C\va^{\frac{2\gamma}{1+\gamma}}.\]
This implies \eqref{energy_w11_inomega_z1} with $\delta(z_1)\geq\,\va$. 

{\bf Case 2.} For $\va^{\frac{1}{1+\gamma}}\leq|z_1|\leq\,R_{1}$ and $0<s<|z_1|$, we have $\frac{1}{C}|z_1|^{1+\gamma}\leq\delta(z_1)\leq\,C|z_1|^{1+\gamma}$. Estimates \eqref{energy_nabla_w11_square_in} and \eqref{ineq-s-va}  become, respectively,
\begin{align}\label{energy_nabla_w11_square}
\int_{\widehat{\Omega}_{s}(z_1)}|w|^{2}\ dx
\leq&\,C|z_1|^{2(1+\gamma)}\int_{\widehat{\Omega}_{s}(z_1)}|\nabla{w}|^{2}\ dx, \quad\mbox{if}~\,0<s<\frac{2}{3}|z_1|,
\end{align}
and
\begin{align}\label{ineq-H-z}
&\int_{\widehat{\Omega}_{s}(z_1)} |\mathbb{C}^0e(\bar{u}_1^1)-\mathcal{M}_1|^2 \ dx\notag \\
& \leq C\Big(\frac{s^3}{|z_1|^{3+\gamma}}+ \frac{s}{|z_1|^{1-\gamma}}
 +\frac{s^{3-2\gamma}}{|z_1|^{3-\gamma-2\gamma^2}}+\frac{s^{1+2\gamma}}{|z_1|^{1+\gamma+2\gamma^2}}\Big)=:G_2(s).
\end{align}
In view of \eqref{ineq-iteration-w1}, and \eqref{energy_nabla_w11_square}, estimate \eqref{ineq-F111_in} becomes 
\begin{equation}\label{ineq-ft-fs}
F(t)\leq\,\left(\frac{c_{2}|z_1|^{1+\gamma}}{s-t}\right)^{2}F(s)+CG_2(s),
\quad\forall~0<t<s<\frac{2}{3}|z_1|,
\end{equation}
where $c_2$ is another fixed constant.
Let $k=\left[\frac{1}{4c_{2}|z_1|^{\gamma}}\right]$ and $t_{i}=\delta(z_1)+2c_{2}i\,|z_1|^{1+\gamma}$, $i=0,1,2,\cdots,k$. Recalling the definition of $H(s)$ in \eqref{ineq-H-z}, one has
\begin{align*}
G_2(t_{i+1}) \leq C(i+1)^{3}|z_1|^{2\gamma}.
\end{align*}
Then, taking $s=t_{i+1}$ and $t=t_{i}$ in \eqref{ineq-ft-fs}, the iteration formula is
$$F(t_{i})\leq\,\frac{1}{4}F(t_{i+1}) +C(i+1)^{3}|z_1|^{2\gamma}.$$
After $k$ iterations, and using \eqref{energy-nabla-w11} again, 
\begin{equation*}
F(t_{0}) \leq (\frac{1}{4})^{k}F(t_{k})+C|z_1|^{2\gamma}\sum_{i=0}^{k-1}(\frac{1}{4})^i(i+1)^{3}\leq C|z_1|^{2\gamma}.
\end{equation*}
Thus, \eqref{energy_w11_inomega_z1} is proved with $\delta(z')\geq \frac{1}{C}|z_1|^{1+\gamma}$.

{\bf STEP 3. Rescaling and $L^{\infty}$ estimates of $|\nabla w|$.}

Making the following change of variables on $\widehat{\Omega}_{\delta(z_1)}(z_1)$, as in \cite{bll,bll2}, 
\begin{equation*}
\left\{
\begin{aligned}
&x_1-z_1=\delta(z_1) y_1,\\
&x_2=\delta(z_1) y_2,
\end{aligned}
\right.
\end{equation*}
then $\widehat{\Omega}_{\delta(z_1)}(z_1)$ becomes $Q_{1}$ of nearly unit size, where
\begin{align*}
Q_r=\Big\{y\in \R^2 :& -\frac{\va}{2\dt(z_1)} +\frac{1}{\dt(z_1)}h_2(\dt(z_1) y_1 +z_1) \\
&  < y_2 < \frac{\va}{2\dt(z_1)} +\frac{1}{\dt(z_1)}h_1(\dt(z_1) y_1 +z_1), \, |y_1 |< r  \Big\},
\end{align*}
for $r\leq1$, and the top and
bottom boundaries
become
\[
\Gamma^+_r=\left\{ y\in \R^2 \, : \,
y_2=\frac{\varepsilon}{2\delta(z_1)}+\frac{1}{\delta(z_1)}h_{1}(\delta(z_1) y_1+z_1),\quad|y_1|<r\right\},\]
and
\[\Gamma^-_r=\left\{y\in \R^2 \, : \,y_2=-\frac{\va}{2\delta(z_1)}+\frac{1}{\delta(z_1)}h_{2}(\delta(z_1) y_1+z_1), \quad |y_1|<r \right\}.\]
We denote
\[ \widetilde{w}(y_1, y_2):= w(\dt(z_1) y_1 + z_1 , \dt(z_1) y_2), \quad \widetilde{u}(y_1 , y_2):=\bar{u}_1^1(\dt(z_1) y_1 +z_1, \dt(z_1) y_2), \]
 for $(y_1, y_2)\in Q_1$. From \eqref{w20'}, we see that $\widetilde{w}$ satisfies
\begin{equation}\label{equ-w-divf-Q1}
\left\{ \begin{aligned}
-\ptl_j\big(C_{ijkl}\ptl_l \widetilde{w}^{(k)} \big) &=\ptl_j\big(C_{ijkl}\ptl_l\widetilde{u}^{(k)}\big)\quad &\text{in}\quad Q_1, \\
\widetilde{w}&=0 \quad &\text{on} \quad \Gamma_1^{\pm}.
\end{aligned}\right.
\end{equation}

Now we apply Theorem \ref{lem-lp-esti} for \eqref{equ-w-divf-Q1} with $\tilde{f}_{ij}= C_{ijkl} \ptl_l \widetilde{u}^{(k)}$. Noticing that
\begin{equation*}
[C_{ijkl} \ptl_l \widetilde{u}^{(k)}]_{\gamma, \, Q_1} \leq C [\nabla \widetilde{u}]_{\gamma,\, Q_1},
\end{equation*}
 we obtain 
\begin{equation}\label{esti-l-infty-poin}
\|\widetilde{w}\|_{L^\infty (Q_{1/2})} \leq C\left(\|\widetilde{w}\|_{H^1( Q_1)}+ [\nabla \widetilde{u}]_{\gamma,\, Q_1} \right).
\end{equation}
Applying Theorem \ref{lem-global-C1alp-estimates} for \eqref{equ-w-divf-Q1} with $\tilde{f}_{ij}= C_{ijkl} \ptl_l \widetilde{u}^{(k)}$ on $Q_{1/2}$,  we have
\begin{equation*}\label{esti-C1alp}
 \|\widetilde{w}\|_{C^{1, \, \gamma}(Q_{1/4} )} \leq C\left( \|\widetilde{w}\|_{L^\infty( Q_{1/2})} + [\nabla \widetilde{u}]_{\gamma,\, Q_{1/2}} \right).
\end{equation*}
This, combining with \eqref{esti-l-infty-poin} and using the Poincar\'{e} inequality, yields  
\[\|\nabla \widetilde{w} \|_{L^\infty(Q_{1/4})} \leq  C \left(\|\nabla \widetilde{w}\|_{L^2( Q_1)}+ [\nabla \widetilde{u}]_{\gamma, \, Q_1 } \right).\]

Rescaling back to the original region $\widehat{\om}_{\dt(z_1)}(z_1)$,
\begin{equation}\label{ineq-scale-original}
\|\nabla w\|_{L^\infty( \widehat{\Omega}_{\delta(z_1)/4}(z_1))}\leq \frac{C}{\delta(z_1)}\left( \|\nabla w\|_{L^2( \widehat{\Omega}_{\delta(z_1)}(z_1))}+\delta(z_1)^{1+\gamma}[\nabla \bar{u}_1^1]_{\gamma, \, \widehat{\Omega}_{\delta(z_1)}(z_1) }\right).
\end{equation}
Here, combining  \eqref{ineq-semi-holder-norm} and  \eqref{def-function1}, one has
\begin{equation}\label{semi-norm-v12}
[\nabla \bar{u}_1^1]_{\gamma, \, \widehat{\Omega}_{\delta(z_1)}(z_1) }\leq [\nabla \bar{u}]_{\gamma, \, \widehat{\Omega}_{\delta(z_1)}(z_1)}\leq C\delta(z_1)^{-\gamma-\frac{1}{1+\gamma}}.
\end{equation}
By virtue of \eqref{semi-norm-v12} and \eqref{energy_w11_inomega_z1}, we have for $(z_1, x_2)\in \widehat{\Omega}_{\delta(z_1)/4}(z_1)$ and $|z_1| \leq R_1$,
\begin{align*}
|\nabla w(z_1, x_2)|&\leq \|\nabla w\|_{L^\infty( \widehat{\Omega}_{\delta(z_1)/4}(z_1))}\\
&\leq C\left(\delta(z_1)^{-1}\cdot\delta(z_1)^{\frac{\gamma}{1+\gamma}}+\delta(z_1)^{\gamma}\cdot\delta(z_1)^{-\gamma-\frac{1}{1+\gamma}}\right)\\
&\leq C\delta(z_1)^{-\frac{1}{1+\gamma}}.
\end{align*}
\eqref{nabla-w-i0} is proved, recalling that $\delta(z_1)\geq \frac{1}{C}(\va+|z_1|^{1+\gamma})$. 

This, together with \eqref{nablau_bar-interior}, yields for $x\in \Omega_{R_1}$,
\begin{equation*}
|\nabla v_1^1(x)|\leq |\nabla \bar{u}_1^1(x)|+|\nabla (v_1^1- \bar{u}_1^1)(x)|\leq\frac{C}{\va+|x_1|^{1+\gamma}},
\end{equation*}
and
\begin{equation*}
|\nabla v_1^1(x)|\geq |\nabla \bar{u}_1^1(x)|-|\nabla (v_1^1-\bar{u}_1^1)(x)|\geq \frac{1}{C(1+|x_1|^{1+\gamma})}.
\end{equation*}
Thus, the proof of \eqref{nabla-v-i0} is completed.
\end{proof}

\begin{proof}[Proof of \eqref{nabla-w-i3} and \eqref{nabla-v-i3} when $d=2$]  Recalling that
\[\bar{u}_1^3=(x_2\bar{u},-x_1\bar{u})^{T},\quad \text{and}\quad \bar{u}_2^3=(x_2\underline{u}, -x_1\underline{u})^{T}, \]
from \eqref{ubar}--\eqref{nablau_bar-interior}, we obtain for $i=1,\,2$,
\begin{equation}\label{nabla-baru-3}
|\nabla \bar{u}_i^3(x)|\leq \frac{C(\va+|x_1|)}{\va+|x_1|^{1+\gamma}}\quad x\in \Omega_{R_1} \quad \text{and} \quad |\nabla \bar{u}_i^3(x)|\leq C\quad x\in \Omega\setminus\Omega_{R_1}.
\end{equation}
Since the proof is similar to that of \eqref{nabla-w-i0}, we only list some key differences and take $i=1$ for instance. For simplicity, denote $w:=w_1^3$, then $w$ satisfies 
	 \begin{equation*}
	 \left\{
	 \begin{aligned}
	 -\mathcal{L}_{\lam,\mu} w&=\nabla\cdot(\mathbb{C}^0e(\bar{u}_1^3))=\nabla\cdot(\mathbb{C}^0e(\bar{u}_1^3)-\mathcal{M}) \quad&\mbox{in}~\Omega,\quad\\
	 w&=0\quad&\mbox{on}~\partial \Omega,
	 \end{aligned}\right.
	 \end{equation*}
for any constant matrix $\mathcal{M}=(\mathfrak{a}_{ij})$, $ i, \, j=1,\, 2$.

First, the total energy is bounded, that is,
\begin{equation}\label{ineq-nabla-w13-bound}
\int_{\Omega}|\nabla w|^2 dx \leq C.
\end{equation}
Indeed, by using \eqref{nabla-baru-3} and the H\"{o}lder inequality, one has
\begin{align*}
\int_{\Omega}\big( \nabla\cdot(\mathbb{C}^0e(\bar{u}_1^3))\big)\cdot w \ dx&=-\int_{\Omega}  \Big(\mathbb{C}^0e(\bar{u}_1^3),  \nabla w \Big) \ dx\\
&\leq C\Big(\int_{\Omega}|\nabla w|^2 dx\Big)^{\frac{1}{2}} \Big(\int_{\Omega_{R_1}}|\nabla \bar{u}_1^3|^2 dx +\int_{\Omega\setminus\Omega_{R_1}} |\nabla \bar{u}_1^3|^2 dx\Big)^{\frac{1}{2}}\\
&\leq C\Big(\int_{\Omega}|\nabla w|^2 dx\Big)^{\frac{1}{2}}.
\end{align*}
Combining with \eqref{ineq-kohn-1}, we have
\begin{equation*}
\int_{\Omega}|\nabla w|^2 dx \leq C\Big(\int_{\Omega}|\nabla w|^2 dx\Big)^{\frac{1}{2}}.
\end{equation*}

Next, we estimate the local energy estimates:
\begin{equation}\label{engy-w13}
\int_{\widehat{\Omega}_{\delta(z_1)}(z_1)}\left|\nabla{w}\right|^{2}\ dx\leq
C\delta(z_1)^{2}.
\end{equation}
As in the proof of \eqref{energy_w11_inomega_z1}, we have, instead of \eqref{ineq-iteration-w1}, 
\begin{equation*}
\int_{\widehat{\Omega}_{t}(z_1)}|\nabla w|^{2}\ dx \leq\,\frac{C}{(s-t)^{2}}\int_{\widehat{\Omega}_{s}(z_1)}|w|^{2}\ dx +C\int_{\widehat{\Omega}_{s}(z_1)} |\mathbb{C}^0e(\bar{u}_1^3)-\mathcal{M}|^2 \ dx,
\end{equation*}
where 
\begin{equation*}\label{def-Ce13}
\mathbb{C}^0e(\bar{u}_1^3)=\begin{pmatrix}
~2\mu x_2\ptl_1 \bar{u} + \lam(x_2\ptl_1\bar{u} -x_1\ptl_2\bar{u} ) &  \mu(x_2 \ptl_2\bar{u} -x_1\ptl_1\bar{u} )~\\
~~\\
~\mu( x_2 \ptl_2\bar{u} -x_1\ptl_1\bar{u} ) & -2\mu x_1\ptl_2 \bar{u} + \lam ( x_2\ptl_1\bar{u} -x_1\ptl_2\bar{u})~
\end{pmatrix},
\end{equation*}
and take $\mathcal{M}=\mathcal{M}_2$, its average on $\widehat{\Omega}_{s}(z_1)$,
\begin{equation*}
\mathcal{M}_2:=\pvint_{\widehat{\Omega}_{s}(z_1)} \mathbb{C}^0e(\bar{u}_1^3(y)) \ dy.
\end{equation*}
By using \eqref{ineq-semi-holder-norm}, \eqref{ineq-x-prime}, \eqref{nablau_bar-interior} and the following inequality
\begin{equation}\label{ineq-holder}
[x_i\ptl_j\bar{u}]_{\gamma,\, \widehat{\Omega}_{s}(z_1)}\leq \|x_i\|_{L^\infty(\widehat{\Omega}_{s}(z_1))}[\ptl_j \bar{u}]_{\gamma,\,\widehat{\Omega}_{s}(z_1)}+\|\ptl_j \bar{u}\|_{L^\infty(\widehat{\Omega}_{s}(z_1))} [x_i]_{\gamma,\, \widehat{\Omega}_{s}(z_1)},
\end{equation}
where $i, j=1, 2$, we have 
\begin{align}\label{ineq-semi-ui3}
&|\mathbb{C}^0e(\bar{u}_1^3)-\mathcal{M}_2|\notag\\
&\leq C\sum\limits_{i,j=1}^2[x_i\ptl_j\bar{u}]_{\gamma,\, \widehat{\Omega}_{s}(z_1)}\Big(s^\gamma+\delta(z_1)^\gamma\Big)\notag\\
&\leq C\Big(\delta(z_1)^{-\frac{1}{1+\gamma}}s^{1-\gamma}+\delta(z_1)^{-1}s^{1-\gamma}+\delta(z_1)^{-\gamma}+\delta(z_1)^{-\frac{\gamma^2}{1+\gamma}}\Big)\Big(s^\gamma+\delta(z_1)^\gamma\Big)\notag\\
&\leq C\Big(\delta(z_1)^{-1}s^{1-\gamma}+\delta(z_1)^{-\gamma}\Big)\Big(s^\gamma+\delta(z_1)^\gamma\Big).
\end{align}

{\bf Case 1.} For $|z_1|\leq \va^{\frac{1}{1+\gamma}}$, we still have \eqref{energy_nabla_w11_square_in} for $0<s<\va^{\frac{1}{1+\gamma}}$. By using \eqref{ineq-semi-ui3}, a direct calculation leads to 
\begin{equation}\label{ineq-ce-u3-1}
\int_{\widehat{\Omega}_{s}(z_1)} |\mathbb{C}^0e(\bar{u}_1^3)-\mathcal{M}_2|^2 \ dx\leq C(\va^{-1}s^3+\va^{2\gamma-1}s^{3-2\gamma}+\va^{1-2\gamma}s^{1+2\gamma}+\va s)=:\overline{G}_1(s).
\end{equation}
Instead of \eqref{ineq-F111_in}, we have 
\begin{equation*}
F(t)\leq \Big(\frac{c_1\va}{s-t}\Big)^2F(s) + C\overline{G}_1(s), \quad \forall \ 0<t<s< \va^{\frac{1}{1+\gamma}}.
\end{equation*}
By applying the iteration process as in the proof of \eqref{energy_w11_inomega_z1}, we obtain 
\begin{equation*}
\int_{\widehat{\Omega}_{\delta(z_1)}(z_1)}\left|\nabla{w}\right|^{2}\ dx \leq C\va^{2}.
\end{equation*}	 
	 
{\bf Case 2.}  For $\va^{\frac{1}{1+\gamma}}\leq |z_1|\leq R_1$, estimate \eqref{energy_nabla_w11_square} remains the same. Estimate \eqref{ineq-ce-u3-1} becomes 
 \begin{align}\label{ineq-bar-h-s}
 &\int_{\widehat{\Omega}_{s}(z_1)} |\mathbb{C}^0e(\bar{u}_1^3)-\mathcal{M}_2|^2 \ dx\notag \\
 &\leq C(|z_1|^{-(1+\gamma)}s^{3} + |z_1|^{(1+\gamma)(1-2\gamma )}s^{1+2\gamma}+|z_1|^{(1+\gamma)(2\gamma-1)}s^{3-2\gamma}+|z_1|^{1+\gamma}s)=:\overline{G}_2(s).
 \end{align}
Estimate \eqref{ineq-ft-fs} becomes 	 
\begin{equation}\label{intera-z-F}
F(t)\leq \Big(\frac{c_2|z_1|^{1+\gamma}}{s-t}\Big)^2F(s)+C\overline{G}_2(s), \quad \forall \ 0<t<s<\frac{2}{3}|z_1|.
\end{equation}
By iteration, we have
\begin{equation*}
F(t_0) \leq C|z_1|^{2(1+\gamma)}.
\end{equation*}
Thus, \eqref{engy-w13} is proved.

In view of \eqref{ineq-semi-ui3},  we have
\begin{equation*}
[\mathbb{C}^0e(\bar{u}_1^3)]_{\gamma, \widehat{\Omega}_{\delta(z_1)}(z_1)} \leq C[\nabla \bar{u}_1^3]_{\gamma, \widehat{\Omega}_{\delta(z_1)}(z_1)}\leq C \delta(z_1)^{-\gamma}.
\end{equation*}
Similar to the proof of \eqref{nabla-w-i0}, by using \eqref{engy-w13}, Theorem \ref{lem-global-C1alp-estimates} and Lemma \ref{lem-lp-esti},
\begin{align*}
\|\nabla w\|_{L^{\infty}(\widehat{\Omega}_{\delta(z_1)/4}(z_1)) }& \leq \frac{C}{\delta(z_1)}\big( \|\nabla w\|_{L^2(\widehat{\Omega}_{\delta(z_1)}(z_1)} +\delta(z_1)^{1+\gamma}[\mathbb{C}^0e(\bar{u}_1^3)]_{\gamma, \widehat{\Omega}_{\delta(z_1)}(z_1)} \big)\\
&\leq C\delta(z_1)+C\leq C.
\end{align*}
Thus, \eqref{nabla-w-i3} is proved. Consequently, \eqref{nabla-v-i3} easily follows from the definition of $w_i^3$, $i=1, 2$, \eqref{def-function1}, \eqref{nabla-w-i3} and \eqref{nabla-baru-3}.
\end{proof}

\begin{proof}[ Proof of \eqref{v1+v2_bounded1} and \eqref{esti-infty-out-12}]
The proof of \eqref{v1+v2_bounded1} follows from theorem 1.1 in \cite{llby} that the gradient is bounded, because the displacement takes the same constant value on $\ptl D_1$ and $\ptl D_2$.
	
For \eqref{esti-infty-out-12}, thanks to \eqref{ubar-out}, \eqref{energy-nabla-w11} for $w=w_i^l$ and \eqref{ineq-nabla-w13-bound} for $w=w_i^3$, $i=1,\, 2$, one has
\begin{equation*}
\int_{\Omega\setminus \Omega_{\frac{R_1}{2}}} |\nabla v_i^l|^2 \ dx\leq 2\int_{\Omega\setminus \Omega_{\frac{R_1}{2}}}\Big(|\nabla \bar{u}_i^l|^2+|\nabla w_i^l|^2\Big) dx \leq C.
\end{equation*}
Then \eqref{esti-infty-out-12} follows from the classical elliptic estimates (see \cite{adn1} or \cite{bll}). 
\end{proof}

\subsection{Proof of Proposition \ref{lem-rest-terms}}\label{prop-2.4}

\begin{proof}[Proof of Proposition \ref{lem-rest-terms}]
By the trace theorem and a minor modification of the proof of lemma 4.1 in \cite{bll}, it is easy to obtain \eqref{esti-C12}.

To estimate $|C_1^k- C_2^k|$, $k=1, \, 2$, we denote 
\begin{equation}\label{def-u-b} 
u_b:=\sum\limits_{l=1}^3 C_2^l (v_1^l+v_2^l) + v_0,
\end{equation}
and
\begin{align}
& a_{ij}^{k l}:=-\int_{\ptl D_j}\frac{\ptl v_i^k}{\ptl \nu}\Big|_+\cdot \psi^l, \quad  i, \, j =1, \, 2,~~k, \, l= 1, \, 2, \, 3,\label{def-aij-bt}\\
& \widetilde{b}_j^l:=\widetilde{b}_j^l[\varphi]:=\int_{\ptl D_j} \frac{\ptl u_b}{\ptl \nu}\Big|_+\cdot \psi^l,\quad  j =1, \, 2,~~ l= 1, \, 2, \, 3. \label{def-b-j}
\end{align}
Then, we rewrite \eqref{equ-decompositon} for $j=1$ as the following equation: 
\begin{equation*}
\sum\limits_{k=1}^3 (C_1^k-C_2^k) a_{11}^{kl}-\widetilde{b}_1^l =0,  \quad l=1, \, 2, \, 3. 
\end{equation*}
Then, in order to solve $C_1^1-C_2^1$ and $C_1^2-C_2^2$, it can be rewritten in matrix form: 
\begin{equation}\label{equ-two-lower}
\mathcal{A}X:= \begin{aligned}
\begin{pmatrix}
~a_{11}^{11}  & a_{11}^{21} & a_{11}^{31} ~\\
~~\\
~a_{11}^{12} & a_{11}^{22} & a_{11}^{32} ~\\
~~\\
~a_{11}^{13} & a_{11}^{23} & a_{11}^{33}~ 
\end{pmatrix}
\begin{pmatrix}
~C_1^1-C_2^1~\\
~~\\
~C_1^2-C_2^2~\\
~~\\
~C_1^3-C_2^3~
\end{pmatrix}
=\begin{pmatrix}
~\widetilde{b}_1^1~\\
~~\\
~\widetilde{b}_1^2~\\
~~\\
~\widetilde{b}_1^3~
\end{pmatrix}.
\end{aligned}
\end{equation}
To solve it, we need the following lemma. 
\begin{lemma}\label{lemma-a11kk} 
$\mathcal{A}$ is positive definite, and
   \begin{align}
	\frac{1}{C}\va^{-\frac{\gamma}{1+\gamma}}  \leq a_{11}^{kk}  \leq &C\va^{-\frac{\gamma}{1+\gamma}},  \, k= 1,2; \qquad \frac{1}{C} \leq a_{11}^{33} \leq C;\label{esti-a11bb}\\
	|a_{11}^{12}|=&|a_{11}^{21}|\leq C|\ln \va|;\label{esti-a11-1221}\\
	 |a_{12}^{33}|=|a_{21}^{33}|\leq C, &\quad |a_{ij}^{k3}|=|a_{ji}^{3k}|\leq C, \quad i,\,j,\, k=1, \, 2;\label{esti-a-ij-3b}\\
	|\widetilde{b}_j^l|\leq C, &\quad j=1,\,2, \ l=1,\, 2, \, 3.\label{esti-bjk}	
	\end{align}
\end{lemma}
The proof of Lemma \ref{lemma-a11kk} will be given later. It follows from the Cramer's rule that
\begin{equation}\label{equal-c1-c2-k}
C_1^k-C_2^k=\frac{\mathrm{det}\mathcal{A}_k}{\mathrm{det}\mathcal{A}},  \quad k=1, \,2, 
\end{equation}
where
\begin{align*}
\mathcal{A}_1:=\begin{pmatrix}
~\widetilde{b}_1^{1}  & a_{11}^{21} & a_{11}^{31} ~\\
~~\\
~\widetilde{b}_1^{2} & a_{11}^{22} & a_{11}^{32} ~\\
~~\\
~\widetilde{b}_1^{3} & a_{11}^{23} & a_{11}^{33} ~
\end{pmatrix},
\quad
\mathcal{A}_2:=\begin{pmatrix}
~a_{11}^{11}  & \widetilde{b}_1^{1} & a_{11}^{31} ~\\
~~\\
~a_{11}^{12} & \widetilde{b}_1^{2} & a_{11}^{32} ~\\
~~\\
~a_{11}^{13} & \widetilde{b}_1^{3} & a_{11}^{33} ~
\end{pmatrix}.
\end{align*}
By using Lemma \ref{lemma-a11kk}, we have
\begin{equation}\label{det-A1}
\mathrm{det} \mathcal{A}_1= a_{11}^{22}\big( \widetilde{b}_1^{1}\,a_{11}^{33}-\widetilde{b}_1^3\,a_{11}^{31}\big) + O(|\ln \va|),
\end{equation}
and
\begin{equation}\label{det-A}
\mathrm{det}\mathcal{A}=a_{11}^{11}a_{11}^{22}a_{11}^{33}+O(\va^{-\frac{\gamma}{1+\gamma}}).
\end{equation}
Then, substituting \eqref{det-A1} and \eqref{det-A} into \eqref{equal-c1-c2-k} and using Lemma \ref{lemma-a11kk}, we have
\begin{equation*}
C_1^1-C_2^1=\frac{\mathrm{det}\mathcal{A}_1}{\mathrm{det}\mathcal{A}}=\frac{1}{a_{11}^{11}+O(1)}\Big( \widetilde{b}_1^1-\widetilde{b}_1^3\frac{a_{11}^{31}}{a_{11}^{33}}\Big) +O(\va^{\frac{2\gamma}{1+\gamma}}|\ln\va|).
\end{equation*}
Similarly,
\begin{equation*}
\mathrm{det} \mathcal{A}_2=a_{11}^{11} \big( \widetilde{b}_1^2\,a_{11}^{33}-\widetilde{b}_1^3\, a_{11}^{32} \big) +O(|\ln \va|),
\end{equation*}
and
\begin{equation*}
C_1^2-C_2^2=\frac{1}{a_{11}^{22}+O(1)}\Big(\widetilde{b}_1^2-\widetilde{b}_1^3\, \frac{a_{11}^{32}}{a_{11}^{33}} \Big)+O(\va^{\frac{2\gamma}{1+\gamma}}|\ln\va|).
\end{equation*}
Thus, in view of \eqref{esti-a11bb} and \eqref{esti-a-ij-3b}-\eqref{esti-bjk}, we have
\begin{equation*}
|C_1^k-C_2^k|\leq \frac{C}{|a_{11}^{kk}|-C} +C\va^{\frac{2\gamma}{1+\gamma}}|\ln\va| \leq C\va^{\frac{\gamma}{1+\gamma}},\quad k=1,\, 2.
\end{equation*}
The proof of Proposition \ref{lem-rest-terms} is completed.
\end{proof}

\begin{proof}[The proof of Lemma \ref{lemma-a11kk}]
There are some differences with lemma 4.4 in \cite{bll}. For readers' convenience, we sketch the proof in the following. From \eqref{def-aij-bt}, integrating by parts and using \eqref{equ_v1}, we notice that
\begin{equation*}
a_{ij}^{k l}=\int_{\Omega} \Big( \mathbb{C}^0 e(v_i^k), e(v_j^l) \Big)dx. 
\end{equation*}	 
For \eqref{esti-a11bb}, by using \eqref{elliptic}, \eqref{nabla-v-i0} and \eqref{esti-infty-out-12}, one has 
\begin{equation*}
a_{11}^{kk}=\int_{\Omega}\Big( \mathbb{C}^0e(v_1^k), e(v_1^k) \Big) dx\leq C \int_{\Omega} |\nabla v_1^k|^2 \ dx \leq C \va^{-\frac{\gamma}{1+\gamma}}, \quad \text{for} ~~k=1, \, 2.
\end{equation*} 
Moreover, by using \eqref{nablau_bar-interior}, \eqref{def-function1} and \eqref{energy-nabla-w11}, 
\begin{align*}
a_{11}^{11}&=\int_{\Omega} \Big( \mathbb{C}^0 e(v_1^1), e(v_1^1)\Big) dx \geq \frac{1}{C} \int_{\Omega}|e(v_1^1)|^2 dx \\
&\geq \frac{1}{C}\int_{\Omega} |e(\bar{u}_1^1)|^2 dx- C\int_{\Omega}|e(w_1^1)|^2 dx\\
&\geq \frac{1}{C}\int_{\Omega_{R_1}}\frac{1}{(\va+|x_1|^{1+\gamma})^2} \ dx -C \geq \frac{1}{C}\va^{-\frac{\gamma}{1+\gamma}}.
\end{align*}
By the same reason, one has 
\begin{equation*}
a_{11}^{22}\geq\frac{1}{C}\va^{-\frac{\gamma}{1+\gamma}}.
\end{equation*}
Since the estimate of $a_{11}^{33}$ is the same as in \cite{bll}, we omit it. Thus, \eqref{esti-a11bb} is proved.

For $a_{11}^{12},\, a_{11}^{21}$, we have
\begin{equation}\label{equi-12=21}
a_{11}^{12}=\int_{\Omega}\Big( \mathbb{C}^0 e(v_1^1), e(v_1^2) \Big) dx= \int_{\Omega}\Big( \mathbb{C}^0 e(v_1^2) ,  e(v_1^1)\Big) dx=a_{11}^{21}.
\end{equation}
Denote
\begin{equation*}
a_{11}^{12}=-\int_{\ptl D_1}\frac{\ptl v_1^1}{\ptl \nu}\Big|_+\cdot \psi^2=-\int_{\ptl D_1}\frac{\ptl \bar{u}_1^1}{\ptl \nu}\Big|_+\cdot \psi^2 -\int_{\ptl D_1}\frac{\ptl (v_1^1-\bar{u}_1^1)}{\ptl \nu}\Big|_+\cdot \psi^2=:-\mathrm{I}_1-\mathrm{I}_2.
\end{equation*}
We divide $\mathrm{I}_1$ further into two parts as follows
\begin{equation*}
\mathrm{I}_1=\int_{\ptl D_1\cap \mathcal{C}_{R_1}} \frac{\ptl \bar{u}_1^1}{\ptl \nu}\Big|_+\cdot \psi^2 + \int_{\ptl D_1\setminus \mathcal{C}_{R_1}} \frac{\ptl \bar{u}_1^1}{\ptl \nu}\Big|_+\cdot \psi^2, 
\end{equation*}
where the cylinder $\mathcal{C}_r$ is defined by 
\begin{equation}\label{def-cylinder}
\mathcal{C}_r:=\left\{x\in \R^2 : -\frac{\va}{2}+2\min\limits_{|x_1|=r}h_2(x_1)\leq 
x_2\leq \frac{\va}{2}+2\max\limits_{|x_1|=r}h_1(x_1), \quad |x_1|< r\right\}.
\end{equation}
Notice that the components of the normal vector on the portion $\ptl D_1\cap \mathcal{C}_{R_1}$ are
\begin{equation*}
n_1=\frac{\ptl_1 h_1(x_1)}{\sqrt{1+|\ptl_1h_1(x_1)|^2}},\quad n_2=\frac{-1}{\sqrt{1+|\ptl_1 h_1(x_1)|^2}}.
\end{equation*}
Then, 
\begin{align*}
\frac{\ptl \bar{u}_1^1}{\ptl \nu}\Big|_+\cdot \psi^2 \Big|_{\ptl D_1 \cap \mathcal{C}_{R_1}}&=\Big(\lam (\nabla \cdot \bar{u}_1^1) {\bf n} +\mu \big(\nabla \bar{u}_1^1+(\nabla \bar{u}_1^1)^{T}\big){\bf n}\Big)\cdot \psi^2\\
&=\lam \ptl_1 \bar{u} n_2 +\mu \ptl_2 \bar{u} n_1 .
\end{align*}
In view of \eqref{h1h_convex2}, \eqref{ubar-out} and \eqref{nablau_bar-interior}, we have
\begin{align*}
|\mathrm{I}_1|& \leq \Big|\int_{\ptl D_1\cap \mathcal{C}_{R_1}} \frac{\ptl \bar{u}_1^1}{\ptl \nu}\Big|_+\cdot \psi^2 \Big|+\Big|\int_{\ptl D_1\setminus \mathcal{C}_{R_1}} \frac{\ptl \bar{u}_1^1}{\ptl \nu}\Big|_+\cdot \psi^2\Big|\\
&\leq \int_{\ptl D_1\cap \mathcal{C}_{R_1}}| \lam \ptl_1 \bar{u} n_2 +\mu \ptl_2 \bar{u} n_1| \ dS +C\\
&\leq \int_{|x_1|\leq R_1}\frac{C|x_1|^\gamma}{\va+|x_1|^{1+\gamma}} \ dx_1 +C \leq C|\ln \va|.
\end{align*}
For $\mathrm{I}_2$, by dividing $\mathrm{I}_2$ into two parts like $\mathrm{I}_1$ and using \eqref{nabla-w-i0}, \eqref{esti-infty-out-12} and \eqref{ubar-out}, we obtain 
\begin{align*}
|\mathrm{I}_2|&\leq \Big|\int_{\ptl D_1\cap \mathcal{C}_{R_1}} \Big(\lam (\nabla \cdot (v_1^1-\bar{u}_1^1)) {\bf n} +\mu \big(\nabla (v_1^1-\bar{u}_1^1)+(\nabla (v_1^1-\bar{u}_1^1))^{T}\big){\bf n}\Big)\cdot \psi^2\Big|+C\\
&\leq \int_{|x_1|\leq R_1}\frac{C}{(\va+|x_1|^{1+\gamma})^{\frac{1}{1+\gamma}}} \ dx_1 +C\leq C|\ln \va|.
\end{align*}
Thus $|a_{11}^{12}|\leq C|\ln \va|$. Thus, \eqref{esti-a11-1221} is proved.

For \eqref{esti-a-ij-3b}, it follows from the same reason as in \eqref{equi-12=21} that $a_{ij}^{k 3}=a_{ji}^{3k}$ and $a_{12}^{33}=a_{21}^{33}$. Notice that
\begin{equation*}
a_{ij}^{k 3}=\int_{\Omega} \Big(\mathbb{C}^0 e(v_i^k), e(v_j^3)\Big) \ dx =\int_{\Omega}\Big( \mathbb{C}^0 \nabla v_i^k, \nabla v_j^3\Big) \ dx.
\end{equation*}
 We only prove the case that $i=j=k=1$, since the proof is the same for the other cases. By using \eqref{esti-infty-out-12}, \eqref{energy-nabla-w11} and \eqref{ineq-nabla-w13-bound}, we have
\begin{align*}
a_{11}^{13}&=\int_{\Omega_{R_1}}\Big(\mathbb{C}^0 \nabla v_1^1, \nabla v_1^3\Big) \ dx +O(1)\\
&=\int_{\Omega_{R_1}}\Big( \mathbb{C}^0 \nabla \bar{u}_1^1, \nabla \bar{u}_1^3\Big) \ dx + \int_{\Omega_{R_1}}\Big(\mathbb{C}^0 \nabla \bar{u}_1^1, \nabla w_1^3\Big) dx\\ &\quad+\int_{\Omega_{R_1}}\Big(\mathbb{C}^0\nabla\bar{u}_1^3, \nabla w_1^1 \Big) \ dx\
 +\int_{\Omega_{R_1}}\Big(\mathbb{C}^0 \nabla w_1^1, \nabla w_1^3\Big) \ dx +O(1)\\
& = \int_{\Omega_{R_1}} \Big(\mathbb{C}^0\nabla \bar{u}_1^1, \nabla \bar{u}_1^3\Big) dx + \int_{\Omega_{R_1}}\Big(\mathbb{C}^0 \nabla \bar{u}_1^1, \nabla w_1^3\Big) dx \\ &\quad +\int_{\Omega_{R_1}}\Big(\mathbb{C}^0\nabla\bar{u}_1^3, \nabla w_1^1 \Big) \ dx +O(1)=:\mathrm{I}_1+\mathrm{I}_2+\mathrm{I}_3+O(1).
\end{align*}
Since
\begin{align*}
\Big(\mathbb{C}^0\nabla \bar{u}_1^1, \nabla \bar{u}_1^3\Big) &=\begin{pmatrix}
~(\lam+2\mu)\ptl_1\bar{u} & \mu\ptl_2\bar{u}~\\
~~\\
\mu\ptl_2\bar{u} & \lam \ptl_1 \bar{u}
\end{pmatrix} : \begin{pmatrix}
~x_2\ptl_1\bar{u} & \bar{u}+x_2\ptl_2\bar{u}~\\
~~\\
~-\bar{u}-x_1\ptl_1\bar{u} & -x_1\ptl_2\bar{u}~
\end{pmatrix}\\
&=(\lam+2\mu)x_2(\ptl_1\bar{u})^2+\mu x_2 (\ptl_2\bar{u})^2-(\lam+\mu)x_1\ptl_1\bar{u}\ptl_2\bar{u},
\end{align*}
thus, by using \eqref{nablau_bar-interior}, one has
\begin{align*}
|\mathrm{I}_1|\leq C\Big(&\int_{\Omega_{R_1}} \frac{|x_2||x_1|^{2\gamma}}{(\va+|x_1|^{1+\gamma})^2} dx +\int_{\Omega_{R_1}}\frac{|x_2|}{(\va+|x_1|^{1+\gamma})^2} dx\\
&+\int_{\Omega_{R_1}} \frac{|x_1|^{1+\gamma}}{(\va+|x_1|^{1+\gamma})^2} dx \Big)
\leq C.
\end{align*}
Using \eqref{nabla-w-i3} and \eqref{nablau_bar-interior}, one has 
\begin{equation*}
|\mathrm{I}_2|\leq \Big|\int_{\Omega_{R_1}} \Big(\mathbb{C}^0 \nabla \bar{u}_1^1, \nabla w_1^3\Big) \ dx \Big|\leq C\int_{\Omega_{R_1}}|\nabla \bar{u}_1^1| \ dx \leq C.
\end{equation*}
Using \eqref{nabla-baru-3} and \eqref{nabla-w-i0}, 
\begin{align*}
|\mathrm{I}_3|\leq \Big| \int_{\Omega_{R_1}} \Big(\mathbb{C}^0 \nabla \bar{u}_1^3, \nabla w_1^1 \Big) \ dx \Big|\leq C \Big(\int_{\Omega_{R_1}} |\nabla \bar{u}_1^3|^2 dx \Big)^{\frac{1}{2}} \Big(\int_{\Omega_{R_1}} |\nabla w_1^1|^2 dx\Big)^{\frac{1}{2}}\leq C.
 \end{align*}
Therefore, we have
\begin{equation*}
|a_{11}^{13}|\leq C.
\end{equation*}

Similarly, for $a_{21}^{33}$, in view of \eqref{esti-infty-out-12}, we have
\begin{equation*}
a_{21}^{33}=\int_{\Omega}\Big(\mathbb{C}^0 e(v_2^3), e(v_1^3) \Big) dx=\int_{\Omega_{R_1}}\Big(\mathbb{C}^0 \nabla v_2^3, \nabla v_1^3 \Big) dx + O(1).
\end{equation*}
Then, by using \eqref{nabla-v-i3}, one has 
\begin{equation*}
\Big|\int_{\Omega_{R_1}}\Big(\mathbb{C}^0 \nabla v_2^3, \nabla v_1^3 \Big) dx\Big|\leq C\int_{\Omega_{R_1}} |\nabla v_2^3|\cdot |\nabla v_1^3| dx\leq C.
\end{equation*}
Thus, $|a_{21}^{33}|\leq C$  and \eqref{esti-a-ij-3b} is proved.

For \eqref{esti-bjk},  recalling from \eqref{def-u-b} and \eqref{def-b-j}, 
\begin{equation*}
\widetilde{b}_1^l=\int_{\ptl D_1} \frac{\ptl u_b}{\ptl \nu}\Big|_+\cdot \psi^l=\sum\limits_{k=1}^3 C_2^k \int_{\ptl D_1} \frac{\ptl(v_1^k+v_2^k)}{\ptl \nu}\Big|_+\cdot \psi^l + \int_{\ptl D_1} \frac{\ptl v_0}{\ptl \nu}\Big|_+ \cdot \psi^l.
\end{equation*}
Integrating by parts and using \eqref{v1+v2_bounded1}, \eqref{nabla-v-i3} and \eqref{nabla-v-i0}, one has
\begin{equation*}
\Big| \int_{\ptl D_1} \frac{\ptl v_0}{\ptl \nu}\Big|_+\cdot \psi^l\Big|= \Big| \int_{\Omega} \Big(\mathbb{C}^0 e(v_0), e(v_1^l) \Big) dx \Big| \leq C\|\nabla v_0\|_{L^\infty (\Omega)}\int_{\Omega}|\nabla v_1^l| dx\leq C.
\end{equation*}
Similarly, by using \eqref{v1+v2_bounded1}, we have  
\begin{equation*}
\Big|\int_{\ptl D_1} \frac{\ptl (v_1^k+v_2^k)}{\ptl \nu}\Big|_+ \cdot \psi^l \Big|\leq C.
\end{equation*}
Then, combining with \eqref{esti-C12}, we have \eqref{esti-bjk} for $j=1$ and the case of $j=2$ follows from the same reason. By the same proof of lemma 4.2 in \cite{lhg}, we obtain that $\mathcal{A}$  is positive definite. The proof of Lemma \ref{lemma-a11kk} is completed.
\end{proof}

\vspace{0.5cm}

\section{Lower bound of $|\nabla u|$ }\label{sub-lower-bounds}\label{sec-lower-bound} 

This section is devoted to the proof of Theorem 1.3. We follow the idea in \cite{lhg}, with some modifications provided. Recalling the definition of $u_b$ in \eqref{def-u-b}, it follows from \eqref{v1+v2_bounded1} and \eqref{esti-C12} that
\begin{equation}\label{ineq-bdd-ub}
\|\nabla u_b\|_{L^\infty(\Omega)}\leq C.
\end{equation}
From decomposition \eqref{decomposition_u2}, we see that
\begin{equation*}
|\nabla u(x)|\geq \Big|\sum\limits_{k=1}^3(C_{1}^k-C_{2}^k)\nabla{v}_{1}^k(x)\Big|-
|\nabla u_b(x)|, \quad x\in\Omega_{R_1}.
\end{equation*}
Then, to estimate the lower bound of $|\nabla u|$, in view of \eqref{nabla-v-i0}, \eqref{nabla-v-i3} and \eqref{esti-C12}, we only need to establish the lower bound of $|C_1^k-C_2^k|$, $k=1,\, 2, \,3$.

To this end, we need the following two lemmas. Let $v_i^{*l}$, $i=1, \,2$, $l=1,\,2,\,3$, satisfy, respectively,
\begin{equation}\label{equ-vil*}
\left\{
\begin{aligned}
\mathcal{L}_{\lam,\mu} v_i^{*l} &=0, ~~~&\mbox{in}&~\Omega^*,\\
v_i^{*l}&=\psi^l, ~~~&\mbox{on}&~\partial D_i^*\setminus\{0\},\\
v_i^{*l}&=0, ~~~ &\mbox{on}&~ \ptl D\cup\ptl D_j^*,~~j\neq i.
\end{aligned}\right.
\end{equation}
By using the approach from proposition 2.1 in \cite{lhg}, we have the following lemma about the gradient estimates on the boundary.
\begin{lemma}\label{lem-v-v}
	Let $v_i^l$ and $v_i^{*l}$ satisfy \eqref{equ_v1} and \eqref{equ-vil*}, respectively. Then, for sufficiently small $\va>0$,
	\begin{equation}\label{ineq-vil-v*}
	|\nabla (v_i^l-v_i^{*l})(x)|\leq C\va^{\frac{\gamma}{1+2\gamma}},  \quad \text{on} ~~\ptl D,\quad i, \,l=1,\, 2, 
	\end{equation}
	and
	\begin{equation}\label{ineq-vi3-v*}
	|\nabla (v_i^3-v_i^{*3})(x)|\leq C \va^{\frac{1+\gamma}{1+2\gamma}}, \quad \text{on} ~~\ptl D, \quad i=1, \,2,
	\end{equation}
	where $C$ is a universal constant.
\end{lemma}

\begin{proof}
	Since the proof is same for the other cases, we only prove \eqref{ineq-vil-v*} for $i=1$ and $l=1$, following the idea used in Step 1.1 and Step 1.2 in proof of proposition 2.1 in \cite{lhg}.  We define the auxiliary functions, $\bar{u}^*$ and $\bar{u}_1^{*l}$, as limits of $\bar{u}$ and $\bar{u}_1^l$, where $\bar{u}^*\in C^{1,\, \gamma}(\R^2)$ satisfies $\bar{u}^*=1$ on $\ptl D_1^*\setminus \{0\}$, $\bar{u}^*=0$ on $\ptl D_2^*\cup \ptl D$ and 
	\begin{equation*}
	\bar{u}^*=\frac{x_2-h_2(x_1)}{h_1(x_1)-h_2(x_1)},\quad \text{in} ~~\Omega_{2R_1}^*,\quad \|\bar{u}^*\|_{C^{1,\,\gamma} (\Omega^*\setminus \Omega_{R_1}^*)}\leq C,
	\end{equation*}	
	and
	\begin{equation}\label{def-bar-u*}
	\bar{u}_1^{*l}:=\bar{u}^*\psi^l,\quad l=1,\, 2, \, 3.
	\end{equation}
	Here $\Omega_r^*:=\{(x_1, x_2)\in \R^2 : |x_1|< r,~ h_2(x_1)<x_2<h_1(x_1)\}$. By a direct calculation, we have
	\begin{equation}\label{nabla-bar-u*}
	|\ptl_1 \bar{u}^*(x)|\leq \frac{C}{|x_1|}, \quad \frac{1}{C|x_1|^{1+\gamma}}\leq \ptl_2 \bar{u}^*(x)\leq \frac{C}{|x_1|^{1+\gamma}}, \quad \text{for}~~x\in \Omega_{R_1}^*.
	\end{equation}
	Then, by using \eqref{nablau_bar-interior}, \eqref{nabla-bar-u*} and the definitions of $\bar{u}_1^l$ and $\bar{u}_1^{*l}$, we have for  $x\in \Omega_{R_1}^*$,
	\begin{equation}\label{bar-u-ptl2-12}
	|\ptl_2(\bar{u}_1^1-\bar{u}_1^{*1})|\leq |\ptl_2(\bar{u}-\bar{u}^*)|\leq \frac{C\va}{(\va+|x_1|^{1+\gamma})|x_1|^{1+\gamma}}.
	\end{equation} 
	Applying Proposition \ref{prop-interior} to   \eqref{equ-vil*} yields 
	\begin{equation}\label{nabla-v*-123}
	|\nabla (v_1^{*1}-\bar{u}_1^{*1})(x)|\leq \frac{C}{|x_1|}, \quad x\in \Omega_{R_1}^*,\quad
	\text{and}\quad
	|\nabla v_1^{*1}(x)|\leq C,\quad x\in \Omega^*\setminus\Omega_{R_1}^*.
	\end{equation}
	Moreover, we notice that $v_1^1-v_1^{*1}$ satisfies, respectively,
	\begin{equation}\label{equ-v11-v11*}
	\left\{
	\begin{aligned}
	\mathcal{L}_{\lam,\mu} (v_1^1-v_1^{*1}) &=0, \quad&\mbox{in}&~V:=D\setminus \overline{D_1\cup D_2 \cup D_1^*\cup D_2^*},\\
	v_1^1-v_1^{*1}&=\psi^1-v_1^{*1}, \quad&\mbox{on}&~\partial D_1\setminus D_1^*,\\
	v_1^1-v_1^{*1}&=-v_1^{*1}, \quad&\mbox{on}&~\partial D_2\setminus D_2^*,\\
	v_1^1-v_1^{*1}&=v_1^1-\psi^1, \quad&\mbox{on}&~\partial D_1^*\setminus( D_1\cup \{0\}),\\
	v_1^1-v_1^{*1}&=v_1^1, \quad&\mbox{on}&~\partial D_2^*\setminus D_2,\\
	v_1^1-v_1^{*1}&=0, \quad&\mbox{on}&~\partial D.
	\end{aligned}\right.
	\end{equation}
	For $x\in \big(\ptl D_1\setminus D_1^*\big)\cup \big(\ptl D_2\setminus D_2^*\big)\subset \Omega^*\setminus \Omega_{R_1}^*$, it follows from mean value theorem and \eqref{nabla-v*-123} that
	\begin{equation}\label{v1l-v*-ptl-D1}
	|(v_1^1-v_1^{*1})(x_1, x_2)|=|(\psi^1-v_1^{*1})(x_1, x_2)|=|v_1^{*1}(x_1, x_2-\frac{\va}{2})-v_1^{*1}(x_1, x_2)|\leq C\va.
	\end{equation}
	For $x\in \ptl D_1^*\setminus (D_1\cup \mathcal{C}_{\va^{\theta}})$, by virtue of mean value theorem and \eqref{nabla-v-i0}, we have
	\begin{align}\label{v1l-v*-ptl-D1*}
	|(v_1^1-v_1^{*1})(x_1, x_2)|&=|(v_1^1 -\psi^1)(x_1, x_2)|=|v_1^1(x_1, x_2)-v_1^1(x_1, x_2+\frac{\va}{2})|\notag \\
	&\leq \frac{C\va}{\va+|x_1|^{1+\gamma}}\leq C\va^{1-\theta(1+\gamma)},
	\end{align}
	where $0<\theta< 1/(1+\gamma)$ is some constant to be determined later. Similarly, for $x\in \ptl D_2^*\setminus (D_2\cup \mathcal{C}_{\va^\theta})$,  one has
	\begin{equation}\label{v1l-v*-ptl-D2*}
	|(v_1^1-v_1^{*1})(x)|\leq C\va^{1-\theta(1+\gamma)}. 
	\end{equation}
	For $x\in \Omega_{R_1}^*$ with $|x_1|=\va^\theta$, by using \eqref{nabla-w-i0}, \eqref{bar-u-ptl2-12} and  \eqref{nabla-v*-123}, we have
	\begin{align}\label{ptl-2-v1l-v*}
	|\ptl_2(v_1^1-v_1^{*1})(x)|
	&=|\ptl_2 (v_1^1-\bar{u}_1^1)(x)+\ptl_2(\bar{u}_1^1-\bar{u}_1^{*1})(x)+ \ptl_2 (\bar{u}_1^{*1}-v_1^{*1})(x)| \notag \\
	&\leq \frac{C}{(\va+|x_1|^{1+\gamma})^{1/(1+\gamma)}}+\frac{C\va}{(\va+|x_1|^{1+\gamma})|x_1|^{1+\gamma}}+\frac{C}{|x_1|} \notag \\
	&\leq \frac{C}{\va^\theta}+\frac{C}{\va^{2\theta(1+\gamma)-1}}.
	\end{align}
	Thus, for $x\in \Omega_{R_1}^*$ with $|x_1|=\va^\theta$, by virtue of mean value theorem, \eqref{v1l-v*-ptl-D2*} and \eqref{ptl-2-v1l-v*}, we have
	\begin{align}\label{v1l-v*-va-theta}
	|(v_1^1-v_1^{*1})(x_1, x_2)|
	&\leq |(v_1^1-v_1^{*1})(x_1, x_2)-(v_1^1-v_1^{*1})(x_1, h_2(x_1))|+ C\va^{1-\theta(1+\gamma)} \notag \\
	&\leq |\ptl_2(v_1^1-v_1^{*1})|\cdot \big(h_1(x_1)-h_2(x_1)\big)\Big|_{|x_1|=\va^\theta}+C\va^{1-\theta(1+\gamma)} \notag \\
	&\leq C\Big(\va^{\theta\gamma}+ \va^{1-\theta(1+\gamma)} \Big).
	\end{align}
	By taking $\theta \gamma=1-\theta(1+\gamma)$, one has $\theta=1/(1+2\gamma)$. Substituting it into \eqref{v1l-v*-ptl-D1*}, \eqref{v1l-v*-ptl-D2*} and \eqref{v1l-v*-va-theta}, together with \eqref{v1l-v*-ptl-D1} and $v_1^1-v_1^{*1}=0$ on $\ptl D$, we obtain
	\begin{equation*}
	|v_1^1-v_1^{*1}|\leq C\va^{\frac{\gamma}{1+2\gamma}},\quad \text{on}~~\ptl (V\setminus \mathcal{C}_{\va^{1/(1+2\gamma)}}).
	\end{equation*}
	Then, by virtue of the maximum principle for Lam\'{e} systems in $V\setminus \mathcal{C}_{\va^{1/(1+2\gamma)}}$, see \cite{mmn}, we have 
	\begin{equation}\label{ineq-v11-v11*}
	|(v_1^1-v_1^{*1})(x)|\leq C\va^{\frac{\gamma}{1+2\gamma}},\quad \text{in}~~V\setminus \mathcal{C}_{\va^{1/(1+2\gamma)}}.
	\end{equation}
It follows from the standard boundary gradient estimates for Lam\'{e} system (see \cite{adn1}) that 
	\begin{equation*}
	|\nabla (v_1^1-v_1^{*1})(x)|\leq C\va^{\frac{\gamma}{1+2\gamma}}, \quad \text{on} ~~\ptl D.
	\end{equation*}
The proof of Lemma \ref{lem-v-v} is completed.	
\end{proof}

\begin{lemma}\label{lem-ck-c*}
	Under the hypotheses of Theorem \ref{thm-lower-bound}, let $C_i^l$ and $C_*^l$ be defined in \eqref{decomposition-u} and \eqref{equ-limit}, respectively. If $\varphi$ satisfies \eqref{assup-varphi}, then
	\begin{equation}\label{equal-c13=c23}
	C_1^3=C_2^3,
	\end{equation} 
	and, for sufficiently small $\va>0$, 
	\begin{equation}\label{ineq-c1l+c2l-c*l}
	\Big|\frac{C_1^l+C_2^l}{2} -C_*^l\Big|\leq C\va^{\frac{\gamma}{1+2\gamma}}, \quad l=1, \, 2,\,  3,
	\end{equation}
	 where $C$ is a universal constant. Consequently, 
	\begin{equation}\label{ineq-c2-c*}
	|C_i^l-C_{*}^l|\leq \frac{1}{2}|C_1^l-C_2^l| + \Big|\frac{C_1^l+C_2^l}{2}- C_*^l\Big| \leq C\va^{\frac{\gamma}{1+2\gamma}}, \quad i=1, \,2, ~~ l=1,\, 2, \,3.
	\end{equation}
	\end{lemma}

\begin{proof}
	Recalling decomposition \eqref{decomposition-u} and the forth line of \eqref{lame-system-infty}, we have 
	\begin{equation}\label{equ-6*6}
	\left\{\begin{aligned}
	&\sum_{k=1}^3 C_1^k \, a_{11}^{kl} +\sum_{k=1}^3 C_2^k\, a_{21}^{kl}- b_1^l=0,\\
	&\sum_{k=1}^3 C_1^k \,a_{12}^{kl} +\sum_{k=1}^3C_2^k \, a_{22}^{kl}- b_2^l=0,
	\end{aligned}\qquad l =1, 
	\, 2,\,  3,\right.
	\end{equation}
	where $a_{ij}^{kl}$ is defined by \eqref{def-aij-bt}, and 
	\begin{equation*}
	b_j^l:=b_j^l[\varphi]:=\int_{\ptl D_j} \frac{\ptl v_0}{\ptl\nu}\Big|_+\cdot \psi^l.
	\end{equation*}

	Now, we use the symmetric conditions \eqref{assump-symm} and \eqref{assup-varphi}, to estimate $a_{11}^{kl}-a_{22}^{kl}$, $a_{12}^{kl}-a_{21}^{kl}$ and $C_1^3-C_2^3$. Since $\Omega$ is symmetric with respect to $x_1$-axis, it follows that
	\begin{equation}\label{rel-v1-2}
	\left\{\begin{aligned}
	(v_2^1)^{(1)}(x_1, x_2)&= (v_1^1)^{(1)}(x_1, -x_2),\\
	(v_2^1)^{(2)}(x_1, x_2)&= -(v_1^1)^{(2)}(x_1, -x_2),
	\end{aligned}\right.\quad \text{for}~~ x\in \Omega,
	\end{equation}
	and
	\begin{equation}\label{rel-v2-3}
	\left\{\begin{aligned}
	(v_2^k)^{(1)}(x_1, x_2)&= -(v_1^k)^{(1)}(x_1, -x_2),\\
	(v_2^k)^{(2)}(x_1, x_2)&= (v_1^k)^{(2)}(x_1, -x_2),
	\end{aligned}\right. \quad \text{for}~~ x\in \Omega,\quad k =2, \, 3.
	\end{equation}
	By virtue of \eqref{rel-v1-2} and \eqref{rel-v2-3},  a direct calculation leads to 
	\begin{equation*}
	\begin{aligned}
	\ptl_1(v_2^1)^{(1)}(x_1, x_2)&=\ptl_1 (v_1^1)^{(1)}(x_1, -x_2),  \\
	\ptl_1(v_2^1)^{(2)}(x_1, x_2)&= -\ptl_1(v_1^1)^{(2)}(x_1, -x_2), 
	\end{aligned}\quad
	\begin{aligned}
	\ptl_2(v_2^1)^{(1)}(x_1, x_2)&=-\ptl_2 (v_1^1)^{(1)}(x_1, -x_2),\\
\ptl_2(v_2^1)^{(2)}(x_1, x_2)&= \ptl_2(v_1^1)^{(2)}(x_1, -x_2),
	\end{aligned}
	\end{equation*}
and for $k=2,\, 3$,
	\begin{equation*}
\begin{aligned}
\ptl_1(v_2^k)^{(1)}(x_1, x_2)&=-\ptl_1 (v_1^k)^{(1)}(x_1, -x_2),  \\
\ptl_1(v_2^k)^{(2)}(x_1, x_2)&= \ptl_1(v_1^k)^{(2)}(x_1, -x_2), 
\end{aligned}\quad
\begin{aligned}
\ptl_2(v_2^k)^{(1)}(x_1, x_2)&=\ptl_2 (v_1^k)^{(1)}(x_1, -x_2),\\
\ptl_2(v_2^k)^{(2)}(x_1, x_2)&= -\ptl_2(v_1^k)^{(2)}(x_1, -x_2).
\end{aligned}
\end{equation*}
Thus
	\begin{equation*}
	\begin{aligned}
	\Big(\mathbb{C}^0e(v_2^1), e(v_2^2)\Big)&=\lam\Big( \ptl_1 (v_2^1)^{(1)}+ \ptl_2(v_2^1)^{(2)}\Big)\cdot \Big( \ptl_1(v_2^2)^{(1)} +\ptl_2(v_2^2)^{(2)}   \Big)\\
	&\qquad+2\mu \Big(\ptl_1 (v_2^1)^{(1)}\cdot\ptl_1(v_2^2)^{(1)}  + \ptl_2(v_2^1)^{(2)}\cdot\ptl_2(v_2^2)^{(2)}  \Big)\\
	&\qquad +\mu \Big( \ptl_1(v_2^1)^{(2)}+\ptl_2(v_2^1)^{(1)}\Big) \cdot \Big(  \ptl_1(v_2^2)^{(2)} +\ptl_2(v_2^2)^{(1)} \Big) \\
	&= \lam\Big( \ptl_1 (v_1^1)^{(1)}+ \ptl_2(v_1^1)^{(2)}\Big)\cdot \Big(- \ptl_1(v_1^2)^{(1)} -\ptl_2(v_1^2)^{(2)}   \Big)\\
	&\qquad+2\mu \Big(-\ptl_1 (v_1^1)^{(1)}\cdot\ptl_1(v_1^2)^{(1)}  - \ptl_2(v_1^1)^{(2)}\cdot\ptl_2(v_1^2)^{(2)}  \Big)\\
	&\qquad +\mu \Big( -\ptl_1(v_1^1)^{(2)}-\ptl_2(v_1^1)^{(1)}\Big) \cdot \Big(  \ptl_1(v_1^2)^{(2)} +\ptl_2(v_1^2)^{(1)} \Big) \\
&=-\Big(\mathbb{C}^0e( v_1^1), e( v_1^2)\Big).
	\end{aligned}
	\end{equation*}
	Therefore 
	\begin{equation}\label{symm-axis-1}
	a_{22}^{12}=-a_{11}^{12}.
	\end{equation}
	By the same way, 
	\begin{equation*}
	\begin{aligned}
	\Big(\mathbb{C}^0e(v_2^2), e(v_2^3)\Big)&=\lam\Big( \ptl_1 (v_2^2)^{(1)}+ \ptl_2(v_2^2)^{(2)}\Big)\cdot \Big( \ptl_1(v_2^3)^{(1)} +\ptl_2(v_2^3)^{(2)}   \Big)\\
	&\qquad+2\mu \Big(\ptl_1 (v_2^2)^{(1)}\cdot\ptl_1(v_2^3)^{(1)}  + \ptl_2(v_2^2)^{(2)}\cdot\ptl_2(v_2^3)^{(2)}  \Big)\\
	&\qquad +\mu \Big( \ptl_1(v_2^2)^{(2)}+\ptl_2(v_2^2)^{(1)}\Big) \cdot \Big(  \ptl_1(v_2^3)^{(2)} +\ptl_2(v_2^3)^{(1)} \Big) \\
	&= \lam\Big( -\ptl_1 (v_1^2)^{(1)}- \ptl_2(v_1^2)^{(2)}\Big)\cdot \Big(- \ptl_1(v_1^3)^{(1)} -\ptl_2(v_1^3)^{(2)} \Big)\\
	&\qquad+2\mu \Big(\ptl_1 (v_1^2)^{(1)}\cdot\ptl_1(v_1^3)^{(1)}  + \ptl_2(v_1^2)^{(2)}\cdot\ptl_2(v_1^3)^{(2)}  \Big)\\
	&\qquad +\mu \Big( -\ptl_1(v_1^2)^{(2)}-\ptl_2(v_1^2)^{(1)}\Big) \cdot \Big( -\ptl_1(v_1^3)^{(2)} -\ptl_2(v_1^3)^{(1)} \Big) \\
	&=\Big(\mathbb{C}^0e( v_1^2), e(v_1^3)\Big).
	\end{aligned}
	\end{equation*}
		Hence,
	\begin{equation}\label{symm-axis-2}
	a_{22}^{23}=a_{11}^{23}.
	\end{equation}
	Similarly, we have
	\begin{equation}\label{symm-axis-rest}
	\begin{aligned}
	&a_{22}^{12}=-a_{11}^{12},~~ a_{12}^{12}=-a_{21}^{12},~~ a_{12}^{21}=-a_{21}^{21},\\
	&a_{12}^{13}=-a_{21}^{13}, ~~a_{12}^{31}=-a_{21}^{31},~~a_{12}^{23}=a_{21}^{23},~~a_{12}^{32}=a_{21}^{32}.
	\end{aligned}
	\end{equation}
	
	On the other hand, \eqref{assump-symm} implies that $\Omega$ is symmetric about the origin. Then, for $x\in \Omega$,
	\begin{equation*}
	v_2^k(x_1, x_2)= v_1^k(-x_1, -x_2), \quad k =1, \,2, \quad \text{and} \quad v_2^3(x_1, x_2)=-v_1^3(-x_1, -x_2). 
	\end{equation*}
	By a direct calculation, 
	\begin{equation}\label{symm-origin-1}
	a_{11}^{kl}=a_{22}^{kl}, ~~ a_{12}^{kl}=a_{21}^{kl}, ~~ a_{11}^{33}= a_{22}^{33},  \quad k, l=1, 2,
	\end{equation}
	and
	\begin{equation}\label{symm-origin-2}
	a_{11}^{k3}=-a_{22}^{k3}, ~~ a_{11}^{3k}=-a_{22}^{3k},  ~~ a_{12}^{k3}= -a_{21}^{k3}, ~~a_{12}^{3k} =- a_{21}^{3k}, \quad k =1,\, 2.
	\end{equation}
	Combining with \eqref{symm-axis-1}-\eqref{symm-origin-2}, we obtain that	
	\begin{equation}\label{equal-aij=0}
	a_{ij}^{12}=a_{ij}^{21}=a_{ij}^{23}=a_{ij}^{32}=0, \quad i,\, j =1, \, 2.
	\end{equation}
	To see clearly, we put them into the matrix:
	\begin{equation}\label{matrix-66}
	\begin{aligned}
		\widetilde{\mathcal{A}} :=\begin{pmatrix}
	~ a_{11}^{11} & a_{11}^{21} & a_{11}^{31}  & a_{21}^{11} & a_{21}^{21} & a_{21}^{31}~\\
	~~\\
	~ a_{11}^{12} & a_{11}^{22} & a_{11}^{32}  & a_{21}^{12} & a_{21}^{22} & a_{21}^{32} ~\\
	~~\\
	a_{11}^{13} & a_{11}^{23} & a_{11}^{33}  & a_{21}^{13} & a_{21}^{23} & a_{21}^{33}~\\
	~~\\
	~ a_{12}^{11} & a_{12}^{21} & a_{12}^{31}  & a_{22}^{11} & a_{22}^{21} & a_{22}^{31}~\\
	~~\\
	~ a_{12}^{12} & a_{12}^{22} & a_{12}^{32}  & a_{22}^{12} & a_{22}^{22} & a_{22}^{32} ~\\
	~~\\
	a_{12}^{13} & a_{12}^{23} & a_{12}^{33}  & a_{22}^{13} & a_{22}^{23} & a_{22}^{33}~
	\end{pmatrix}=\begin{pmatrix}
	~ a_{11}^{11} & 0 & a_{11}^{13}  & a_{12}^{11} & 0& a_{12}^{13}~\\
	~~\\
	~ 0 & a_{11}^{22} & 0 & 0 & a_{12}^{22} & 0~\\
	~~\\
	a_{11}^{13} & 0 & a_{11}^{33}  & -a_{12}^{13} & 0 & a_{12}^{33}~\\
	~~\\
	~ a_{12}^{11} & 0 & -a_{12}^{13}  & a_{11}^{11} & 0 & -a_{11}^{13} ~\\
	~~\\
	~ 0 & a_{12}^{22} & 0  & 0 & a_{11}^{22} & 0~\\
	~~\\
	a_{12}^{13} & 0 & a_{12}^{33}  & -a_{11}^{13} & 0 & a_{11}^{33}~
	\end{pmatrix}.
	\end{aligned}
	\end{equation}
	
	To prove \eqref{equal-c13=c23}, using  \eqref{assup-varphi} and the symmetry of domains with respect to $x_2$-axis, we have
	\begin{equation*}
	\left\{ \begin{aligned}
	& v_0^{(1)}(x_1, x_2)=-v_0^{(1)}(x_1, -x_2),\\
	& v_0^{(2)}(x_1, x_2)=v_0^{(2)}(x_1, -x_2), 
	\end{aligned}\right.\quad \text{for}~~ x\in \Omega.
	\end{equation*}
	Then, a direct calculation leads to
	\begin{equation}\label{rel-b123}
	b_2^1=-b_1^1,\quad b_2^2=b_1^2,\quad b_2^3=b_1^3.
	\end{equation} 
	We rewrite \eqref{equ-6*6} into matrix equation:
	\begin{equation*}
	\widetilde{\mathcal{A}} \widetilde{X}= B,
	\end{equation*}
	where $\widetilde{X}= (C_1^1, \, C_1^2, \, C_1^3,\, C_2^1, \, C_2^2, \, C_2^3)^{\text{T}}$ and $B=(b_1^1, \, b_1^2, \, b_1^3,\, b_2^1, \, b_2^2, \, b_2^3 )^{\text{T}}$. By using the Cramer's rule and a direct calculation, we have
	\begin{equation*}
	C_1^3=\frac{( a_{11}^{22})^2-(a_{12}^{22})^2 }{\det \widetilde{\mathcal{A}} }\cdot \det{\mathcal{A}_3'},
	\end{equation*}
where 
\begin{equation*}
\begin{aligned}
\det{\mathcal{A}_3'}=
\begin{vmatrix}
~a_{11}^{11} & b_1^1 &  a_{12}^{11}  &  a_{12}^{13}~\\
~~\\
~a_{11}^{13} & b_1^3 &  -a_{12}^{13}  & a_{12}^{33}~\\
~~\\
~a_{12}^{11} & b_2^1 &  a_{11}^{11}  & -a_{11}^{13}~\\
~~\\
~a_{12}^{13} & b_2^3 &  -a_{11}^{13}  & a_{11}^{33}~
\end{vmatrix}.
\end{aligned}
\end{equation*}
 Adding the third row by the first row, subtracting the fourth row by the second row and using \eqref{rel-b123},  we have 
\begin{equation*}
\begin{aligned}
\det{\mathcal{A}_3'}=
\begin{vmatrix}
~a_{11}^{11} & b_1^1 &  a_{12}^{11}  &  a_{12}^{13}~\\
~~\\
~a_{11}^{13} & b_1^3 &  -a_{12}^{13}  & a_{12}^{33}~\\
~~\\
~a_{11}^{11}+a_{12}^{11} & 0 &  a_{11}^{11} +a_{12}^{11} &  a_{12}^{13}-a_{11}^{13}~\\
~~\\
~a_{12}^{13}-a_{11}^{13} & 0 &  a_{12}^{13}-a_{11}^{13}  & a_{11}^{33}-a_{12}^{33}~
\end{vmatrix}.
\end{aligned}
\end{equation*}
Subtracting the first column by the third column, we obtain
\begin{equation*}
\begin{aligned}
\det{\mathcal{A}_3'}&=
\begin{vmatrix}~~\\
~a_{11}^{11}-a_{12}^{11}  & b_1^1 &  a_{12}^{11}  &  a_{12}^{13}~\\
~~\\
~a_{11}^{13}+a_{12}^{13}  & b_1^3 &  -a_{12}^{13}  & a_{12}^{33}~\\
~~\\
~0 & 0 &  a_{11}^{11} +a_{12}^{11} &  a_{12}^{13}-a_{11}^{13}~\\
~~\\
~0 & 0 &  a_{12}^{13}-a_{11}^{13}  & a_{11}^{33}-a_{12}^{33}~\\
~~
\end{vmatrix}\\
&=
\begin{vmatrix}
~a_{11}^{11}-a_{12}^{11} & b_1^1~\\
~~\\
~a_{11}^{13}+a_{12}^{13} & b_1^3~ 
\end{vmatrix} \begin{vmatrix}
~a_{11}^{11} +a_{12}^{11} &  a_{12}^{13}-a_{11}^{13}~~\\
~~\\
~ a_{12}^{13}-a_{11}^{13}  & a_{11}^{33}-a_{12}^{33}~
\end{vmatrix}.
\end{aligned}
\end{equation*}
Therefore, we have
\begin{equation*}
	C_1^3=\frac{( a_{11}^{22})^2-(a_{12}^{22})^2 }{\det \widetilde{\mathcal{A}} }\cdot \begin{vmatrix}
	~a_{11}^{11}-a_{12}^{11} & b_1^1~\\
	~~\\
	~a_{11}^{13}+a_{12}^{13} & b_1^3~ 
	\end{vmatrix} \begin{vmatrix}
	~a_{11}^{11} +a_{12}^{11} &  a_{12}^{13}-a_{11}^{13}~~\\
	~~\\
	~ a_{12}^{13}-a_{11}^{13}  & a_{11}^{33}-a_{12}^{33}~
	\end{vmatrix}.
\end{equation*}
By the same way, we also have 
\begin{equation*}
\begin{aligned}
	C_2^3&=\frac{( a_{11}^{22})^2-(a_{12}^{22})^2 }{\det \widetilde{\mathcal{A}} } \cdot\begin{vmatrix}
	~  a_{11}^{11}-a_{12}^{11}   &   -b_2^1~\\
	~~\\
	~   a_{11}^{13} +a_{12}^{13}  & b_2^3~
	\end{vmatrix}\begin{vmatrix}
~a_{11}^{11}+a_{12}^{11} & a_{12}^{13}-a_{11}^{13} ~\\
~~\\
~a_{12}^{13}-a_{11}^{13}  & a_{11}^{33}-a_{12}^{33}  ~
\end{vmatrix}.
\end{aligned}
\end{equation*}
Using \eqref{rel-b123} again, we see that
\begin{equation*}
C_1^3=C_2^3.
\end{equation*}

In order to prove \eqref{ineq-c1l+c2l-c*l}, we define
\begin{equation*}
v^l:=v_1^l+v_2^l,\quad v^{*l}:=v_1^{*l}+v_2^{*l},\quad l=1,\,2, \,3,
\end{equation*}
then $v^l$, $v^{*l}$
satisfy, respectively,
\begin{equation}\label{equ-v1+v2+limit}
\left\{
\begin{aligned}
\mathcal{L}_{\lam,\mu} v^{l} &=0, ~~~&\mbox{in}&~\Omega,\\
v^{l}&=\psi^l, ~~~&\mbox{on}&~\partial D_1\cup \ptl D_2,\\
v^{l}&=0,~~~&\mbox{on}&~ \ptl D,
\end{aligned}\right.\quad\left\{
\begin{aligned}
\mathcal{L}_{\lam,\mu} v^{*l} &=0, ~~~&\mbox{in}&~\Omega^*,\\
v^{*l}&=\psi^l, ~~~&\mbox{on}&~\partial D_1^*\cup \ptl D_2^*,\\
v^{*l}&=0,~~~&\mbox{on}&~ \ptl D.
\end{aligned}\right.
\end{equation}
Recalling \eqref{equ-limit}, we decompose  $u^*$ as follows:
\begin{equation}\label{def-ub*}
u^*=\sum_{k=1}^3 C_{*}^k v^{*k}+v_0^{*},
\end{equation}
where  $v_0^*$ satisfies
\begin{equation*}
\left\{
\begin{aligned}
\mathcal{L}_{\lam,\mu} v_0^* &=0, ~~~&\mbox{in}&~\Omega^*,\\
v_0^*&=0, ~~~&\mbox{on}&~\partial D_1^*\cup \ptl D_2^*,\\
v_0^*&=\varphi, ~~~ &\mbox{on}&~ \ptl D.
\end{aligned}\right.
\end{equation*}
 Then, it follows from the third line of \eqref{equ-limit} that
	\begin{equation}\label{equ-limit-all}
	\sum_{k=1}^3 C_*^k a_*^{kl} -(b_{*1}^l+b_{*2}^l)=0, \quad l=1, \,2, \, 3,
	\end{equation}
	where
	\begin{equation*}
	a_*^{kl}:=-\Big(\int_{\ptl D_1^*}\frac{\ptl v^{*k}}{\ptl \nu}\Big|_+\cdot\psi^l+ \int_{\ptl D_2^*}\frac{\ptl v^{*k}}{\ptl \nu}\Big|_+\cdot\psi^l\Big),
	\end{equation*}
	and
	\begin{equation*}
	b_{*j}^l:=b_{*j}^l[\varphi]:=\int_{\ptl D_j^*}\frac{\ptl v_0^*}{\ptl \nu}\Big|_+\cdot \psi^l,\quad j =1, \, 2.
	\end{equation*}
	We rewrite the first group of equations in \eqref{equ-6*6} into the following two equations
	\begin{equation*}
	\sum_{k=1}^3 (C_1^k +C_2^k)\, a_{11}^{kl} +\sum_{k=1}^3 C_2^k\, (a_{21}^{kl}-a_{11}^{kl})= b_1^l,
	\end{equation*}
	and 
	\begin{equation*}
	\sum_{k=1}^3 (C_1^k +C_2^k)\, a_{21}^{kl} -\sum_{k=1}^3 C_1^k\, (a_{21}^{kl}-a_{11}^{kl})= b_1^l.
	\end{equation*}
	Adding these two equations, we have
	\begin{equation}\label{equ-11}
	\sum_{k=1}^3 (C_1^k +C_2^k)\, (a_{11}^{kl}+a_{21}^{kl}) +\sum_{k=1}^3 (C_1^k-C_2^k)\, (a_{11}^{kl}-a_{21}^{kl})= 2b_1^l.
	\end{equation}
	Similarly, for the second group of equations in \eqref{equ-6*6}, we have
	\begin{equation}\label{equ-22}
	\sum_{k=1}^3 (C_1^k +C_2^k)\, (a_{12}^{kl}+a_{22}^{kl}) +\sum_{k=1}^3 (C_1^k-C_2^k)\, (a_{12}^{kl}-a_{22}^{kl})= 2b_2^l.
	\end{equation}
	Combining \eqref{equ-11} and \eqref{equ-22},  one has
	\begin{equation}\label{equ-all}
	\sum_{k=1}^3\frac{C_1^k+C_2^k}{2} a^{kl} +\sum_{k=1}^3\frac{C_1^k-C_2^k}{2}(a_{11}^{kl}-a_{22}^{kl}+a_{12}^{kl}-a_{21}^{kl})=b_1^l+b_2^l,
	\end{equation}
	where $a^{kl}:=\sum\limits_{i,\,j=1}^2 a_{ij}^{kl}$. We rewrite \eqref{equ-limit-all} as the follows
	\begin{equation}\label{equ-limit-varia} 
	\sum_{k=1}^3 C_*^k a^{kl}++\sum_{k=1}^3 C_*^k (a_*^{kl}-a^{kl}) =b_{*1}^l+b_{*2}^l, \quad l=1, \,2, \, 3.
	\end{equation}
	 Noticing from \eqref{equal-aij=0} and \eqref{matrix-66}, we see that
	 \begin{equation*}
	 a^{12}=a^{21}=a^{13}=a^{31}=a^{23}=a^{32}=0.
	 \end{equation*}
	Then, combining with \eqref{equ-all} and \eqref{equ-limit-varia},  we have
	\begin{equation}\label{equ-matrix-3}
	\begin{aligned}
	\begin{pmatrix}
	~ a^{11} & 0  & 0~\\
	~~\\
	~0 & a^{22}  & 0~\\
	~~\\
	~0 & 0  & a^{33}~
	\end{pmatrix} \begin{pmatrix}
	~ \frac{C_1^1+C_2^1}{2}-C_*^1~\\
	~~\\
	~\frac{C_1^2+C_2^2}{2}-C_*^2 ~\\
	~~\\
	~\frac{C_1^3+C_2^3}{2}-C_*^3 ~
	\end{pmatrix}=\begin{pmatrix}
	~ B^1~\\
	~~\\
	~ B^2 ~\\
	~~\\
	~ B^3 ~
	\end{pmatrix},
	\end{aligned}
	\end{equation}
	where
	\begin{align*}
	B^1&=\sum_{j=1}^2(b_j^1-b_{*j}^1)+\sum_{k=1}^3C_*^k(a_*^{k1}-a^{k1})-2(C_1^3-C_2^3)(a_{11}^{13}+a_{12}^{13}),\\
	B^2&=\sum_{j=1}^2(b_j^2-b_{*j}^2)+\sum_{k=1}^3C_*^k(a_*^{k2}-a^{k2}),\\
	B^3&=\sum_{j=1}^2(b_j^3-b_{*j}^3)+\sum_{k=1}^3C_*^k(a_*^{k3}-a^{k3})-2(C_1^1-C_2^1)(a_{11}^{13}+a_{12}^{13}).
	\end{align*}
 To estimate $b_j^k-b_{*j}^k$,  integrating by parts and using Lemma \ref{ineq-vil-v*}, we have
\begin{equation*}
\left|b_j^k-b_{*j}^k\right|=\Big|\int_{\ptl D}\frac{\ptl(v_j^k-v_j^{*k})}{\ptl \nu}\Big|_+\cdot\varphi \Big|\leq C\va^{\frac{\gamma}{1+2\gamma}},\quad k,\, j=1, \, 2.
\end{equation*}
Similarly, we also have  
\begin{equation*}
|b_j^3-b_{*j}^3|\leq C\va^{\frac{1+\gamma}{1+2\gamma}}, \quad j=1, \, 2.
\end{equation*}
 From the proof of lemma 3.2 in \cite{lhg}, we see that  
 \begin{equation}\label{ineq-akl-akl*}
|a^{kl}-a_{*}^{kl}|\leq C\va,\quad k,\, l=1, \,2, \,3.
 \end{equation}
Thus, by virtue of \eqref{esti-c1-c2} and \eqref{equal-c13=c23},  we have
	\begin{equation*}
	|B^l|\leq C\va^\frac{\gamma}{1+2\gamma}, \quad l=1, \, 2, \,3.
	\end{equation*}
Moreover, we notice from the first equation in \eqref{equ-v1+v2+limit} that
\begin{equation*}
a^{kk}=\sum_{i,\,j=1}^2 a_{ij}^{kk}=\int_{\Omega}\big(\mathbb{C}^0 e(v^k), e(v^k)\big) dx \leq C, \quad k=1, \, 2,\, 3.
\end{equation*}
Meanwhile, from \eqref{equ-v1+v2+limit}, we see that $v^{*k}\notin \Psi$ and is not a constant vector. By using \eqref{elliptic},  we have for $k= 1, \,2,\, 3$,
\begin{equation*}
a^{kk}_*=\int_{\Omega^*}\big(\mathbb{C}^0 e(v^{*k}\big), e(v^{*k})) dx\geq C\int_{\Omega^*} \big|e(v^{*k})\big|^2 dx >0.
\end{equation*}
Then, by virtue of \eqref{ineq-akl-akl*}, we have for small $\va>0$, 
\begin{equation*}
a^{kk}\geq C>0,\quad k=1, \, 2,\, 3.
\end{equation*}
  Applying the Cramer's rule to \eqref{equ-matrix-3}, we have for small $\va>0$,
	\begin{equation*}
	\Big|\frac{C_1^l+C_2^l}{2}-C_*^l\Big|=\frac{|B^l|}{a^{ll}}\leq C \va^{\frac{\gamma}{1+2\gamma}},\quad l=1, \,2, 3.
	\end{equation*}
	The proof of Lemma \ref{lem-ck-c*} is completed.
\end{proof}

Now, we use Lemma \ref{lem-v-v} and Lemma \ref{lem-ck-c*} to prove the following convergence proposition.

\begin{prop}\label{prop-blowup-factor-limit}
	Let $\widetilde{b}_{j}^{l}[\varphi]$ and $\widetilde{b}_{*j}^{l}[\varphi]$ be defined by \eqref{def-b-j} and \eqref{def-blowup-factor}, respectively. Then, for sufficiently small $\va>0$, we have
	\begin{equation}\label{convergence-blow-f}
	|\widetilde{b}_{j}^{k}[\varphi] - \widetilde{b}_{*j}^{k}[\varphi]|\leq C \va^{\frac{\gamma}{1+2\gamma}}, \quad  k=1, \, 2,\,3, ~~j=1, \, 2,
	\end{equation}
where $C$ is a universal constant.
\end{prop}

\begin{proof}
	We only prove \eqref{convergence-blow-f} for $j=1$ since the other cases are same. Recalling the definitions of $u_b$ and $u^*$ in \eqref{def-u-b} and \eqref{def-ub*}, we have
	\begin{align}\label{equal-bjk-b*}
	&\widetilde{b}_1^k[\varphi]-\widetilde{b}_{*1}^k[\varphi] \notag\\
	&= \Big(\int_{\ptl D_1}\frac{\ptl v_0}{\ptl\nu}\Big|_+\cdot\psi^k
	-\int_{\ptl D_1^*}\frac{\ptl v_0^*}{\ptl \nu}\Big|_+\cdot\psi^k \Big)+ \sum_{l=1}^3 (C_2^l-C_*^l)\int_{\ptl D_1^*}\frac{\ptl v^{*l}}{\ptl \nu}\Big|_+\cdot \psi^k\notag \\
	&\qquad +\sum_{l=1}^3C_2^l\Big(\int_{\ptl D_1}\frac{\ptl v^l}{\ptl \nu}\Big|_+\cdot\psi^k-\int_{\ptl D_1^*}\frac{\ptl v^{*l}}{\ptl \nu}\Big|_+\cdot \psi^k\Big). 
	\end{align}
Since $v^{*l}$ takes the same value on $\ptl D_1^*\cup \ptl D_2^*$, $|\nabla v^{*l}|\leq C$ in $\Omega^*$ (see \eqref{v1+v2_bounded1}). Thus
	\begin{equation}\label{ineq-bound}
	\Big|\int_{\ptl D_j^*}\frac{\ptl v^{*l}}{\ptl \nu}\Big|_+\cdot \psi^k\Big|\leq C.
	\end{equation}
By virtue of \eqref{ineq-vil-v*} and \eqref{ineq-vi3-v*},  by the same approach of proposition 2.1 in \cite{lhg}, we obtain that
\begin{equation}\label{ineq-b12-b*}
\Big| \int_{\ptl D_1}\frac{\ptl v_0}{\ptl\nu}\Big|_+\cdot\psi^k
-\int_{\ptl D_1^*}\frac{\ptl v_0^*}{\ptl \nu}\Big|_+\cdot\psi^k\Big|\leq C\va^{\frac{\gamma}{1+2\gamma}}\|\varphi\|_{L^1(\ptl D)},\quad k=1,\, 2,\, 3
\end{equation}
and 
\begin{equation}\label{ineq-vlk}
\Big|\int_{\ptl D_1}\frac{\ptl v^l}{\ptl \nu}\Big|_+\cdot\psi^k-\int_{\ptl D_1^*}\frac{\ptl v^{*l}}{\ptl \nu}\Big|_+\cdot \psi^k\Big|\leq C\va^{\frac{\gamma}{1+2\gamma}}|\ptl D|,\quad k,\,l=1, \,2,\, 3.
\end{equation}
Substituting  \eqref{ineq-bound}-\eqref{ineq-vlk} and \eqref{ineq-c2-c*} into \eqref{equal-bjk-b*}, we obtain \eqref{convergence-blow-f}.  The proof Proposition \ref{prop-blowup-factor-limit} is completed.
\end{proof}

\begin{proof}[The proof of Theorem \ref{thm-lower-bound}]  Thanks to \eqref{equal-c13=c23}, 
we only need to estimate the lower bound of  $|C_1^k-C_2^k|$, $k=1,\, 2$. Recalling from \eqref{equal-aij=0} and \eqref{equal-c13=c23}, $C_1^3=C_2^3$ and $a_{11}^{12}=a_{11}^{21}=0$, then  \eqref{equ-two-lower} becomes
\begin{equation*}
\begin{aligned}
\begin{pmatrix}
~a_{11}^{11}  & 0~\\
~~\\
0& a_{11}^{22}  ~
\end{pmatrix}
\begin{pmatrix}
~C_1^1-C_2^1~\\
~~\\
~C_1^2-C_2^2~
\end{pmatrix}
=\begin{pmatrix}
~\widetilde{b}_1^1[\varphi]~\\
~~\\
~\widetilde{b}_1^2[\varphi]~
\end{pmatrix}.
\end{aligned}
\end{equation*}
Thus,
 \begin{equation}\label{equal-C11-C12}
 C_1^1-C_2^1=\frac{\widetilde{b}_1^1[\varphi]}{a_{11}^{11}}, \quad \text{and}\quad C_1^2-C_2^2=\frac{\widetilde{b}_1^2[\varphi]}{a_{11}^{22}}.
 \end{equation}
Noticing from \eqref{esti-a11bb} and Proposition \ref{prop-blowup-factor-limit}, if $\widetilde{b}_{*1}^{\,k}[\varphi] \neq 0$, we have for small $\va>0$,
	\begin{equation*}
	|C_1^k-C_2^k|\geq\frac{|\widetilde{b}_{*1}^{\,k}[\varphi] |}{C}\va^{\frac{\gamma}{1+\gamma}},\quad k=1,\, 2.
	\end{equation*}
	Thus, if there exists $\varphi\in L^\infty(\ptl D; \R^2)$ satisfying \eqref{assup-varphi} such that $\widetilde{b}_{*1}^{\,k_0}[\varphi] \neq 0$ for some $k_0\in \{1, \, 2\}$, in view of \eqref{nabla-w-i0} and the definition of $\bar{u}_1^k$, \eqref{def-function1}, one has for sufficiently small $\va>0$,
	\begin{align*}
	|\nabla u(x)|&\geq \Big|\sum\limits_{k=1}^2(C_{1}^k-C_{2}^k) \nabla \bar{u}_1^k(x) \Big| -\sum\limits_{k=1}^2|C_{1}^k-C_{2}^k|\cdot|\nabla ({v}_{1}^k-\bar{u}_1^k)(x)|-C\\
	&\geq \Big|\sum\limits_{k=1}^2(C_{1}^k-C_{2}^k) \nabla \bar{u}_1^k(x) \Big| - C\va^{-\frac{1-\gamma}{1+\gamma}} \geq \frac{|\widetilde{b}_{*1}^{\,k_0}[\varphi] |}{C\va}  \va^{\frac{\gamma}{1+\gamma}}- C\va^{\frac{\gamma-1}{1+\gamma}}\\
	&\geq \frac{|\widetilde{b}_{*1}^{\,k_0}[\varphi] |}{C}\va^{-\frac{1}{1+\gamma}}, \quad\text{for} ~~ x\in \overline{P_1 P_2}.
	\end{align*}
	The proof of Theorem \ref{thm-lower-bound} is finished.
\end{proof}

\vspace{0.5cm}

\section{Asymptotic expansion of $\nabla u$}\label{sec-asymp}

In this section, we prove Theorem \ref{thm-asymptotics}. Recalling the decomposition in \eqref{decomposition_u2} and \eqref{equal-c13=c23}, we can rewrite
\begin{equation}\label{asymp-nabla-u}
\nabla u= \sum_{k=1}^2(C_1^k-C_2^k)\nabla v_1^k +\nabla u_b,\quad \text{in}~~\Omega_{R_1}.
\end{equation}
By virtue of \eqref{equal-C11-C12} and \eqref{convergence-blow-f}, we have
\begin{equation}\label{asymp-c1k-c2k}
 C_1^k-C_2^k=\frac{\widetilde{b}_1^k[\varphi]}{a_{11}^{kk}}=\frac{1}{a_{11}^{kk}}\big( \widetilde{b}_{*1}^{\,k}[\varphi] +O(\va^{\frac{\gamma}{1+2\gamma}})\big), \quad k=1, \, 2.
\end{equation}
In view of \eqref{nabla-w-i0} and \eqref{ineq-bdd-ub}, it suffices to prove the  asymptotic expansion of $a_{11}^{kk}$, $k=1, \, 2$. To this end, we need the follow proposition. We follow the idea used in \cite{LX1} and provide the main differences for reader's convenience.

\begin{prop}\label{proposition-a11-22}
	Under assumptions \eqref{asymp-symm-h1h2} and \eqref{h1h3-1}, we have, for sufficient small $\va>0$,
	\begin{equation}\label{asym-a1111}
	\begin{aligned}
	a_{11}^{11}&=\frac{\mu}{\va^{\gamma/(1+\gamma)}}\cdot\frac{\widetilde{Q}_\gamma}{\kappa^{1/(1+\gamma)}}
\Big(1+O(\va^{\frac{\gamma}{1+\gamma}}|\ln \va|)\Big),\\
a_{11}^{22}&=\frac{\lam+2\mu}{\va^{\gamma/(1+\gamma)}}\cdot\frac{\widetilde{Q}_\gamma}{\kappa^{1/(1+\gamma)}} \Big(1+O(\va^{\frac{\gamma}{1+\gamma}}|\ln \va|)\Big),
	\end{aligned}
	\end{equation}
	where $\widetilde{Q}_\gamma$ is defined by \eqref{def-wide-Q}.
\end{prop}

\begin{proof}
	For $a_{11}^{11}$, we notice that
	\begin{align*}
	a_{11}^{11}&=\int_{\Omega\setminus\Omega_{R_1}}\Big(\mathbb{C}^0e(v_1^1), e(v_1^1)\Big) dx+ \int_{\Omega_{R_1}\setminus\Omega_{\va^\theta}} \Big(\mathbb{C}^0e(v_1^1), e(v_1^1)\Big) dx\\
	&\qquad + \int_{\Omega_{\va^\theta}} \Big(\mathbb{C}^0e(v_1^1), e(v_1^1)\Big) dx=:\mathrm{I}_1+\mathrm{I}_2+\mathrm{I}_3,
	\end{align*}
	where $\theta= \frac{\gamma}{(1+2\gamma)^2}$. We will give the  asymptotic expansion of $\mathrm{I}_j$, $j=1, \, 2, \, 3$, one by one. 
	
	{\bf Step 1.} For $\mathrm{I}_1$, we claim that there exists a constant $\mathrm{I}^*_1>0$, independent of $\va>0$, such that
	\begin{equation}\label{asymp-I1}
	\mathrm{I}_1=\mathrm{I}^*_1+O(\va^{\frac{\gamma^2}{(1+\gamma)(1+2\gamma)}}).
	\end{equation}
	Indeed, in view of \eqref{equ-v11-v11*}, it follows  from the classical elliptic estimates that
	\begin{equation}\label{ineq-bdd-C1alpha}
 [\nabla (v_1^1-v_1^{*1}) ]_{\gamma, \, V^*}\leq [\nabla v_1^1 ]_{\gamma, \, V^*} + [\nabla v_1^{*1} ]_{\gamma, \, V^*}\leq C, 
\end{equation}
where $V^*:=D\setminus(D_1\cup D_2 \cup D_1^*\cup D_2^*\cup\Omega_{R_1})$. This, together with \eqref{ineq-v11-v11*} and interpolation inequality, yields 
	\begin{equation}\label{ineq-nabla-v11-v11*}
	|\nabla (v_1^1- v_1^{*1})|\leq C\va^{\frac{\gamma^2}{(1+\gamma)(1+2\gamma)}},\quad \text{in} ~~D\setminus(D_1\cup D_2 \cup D_1^*\cup D_2^*\cup\Omega_{R_1}).
	\end{equation}
Then, taking
\begin{equation*}
\mathrm{I}_1^*:=\int_{\Omega^*_{R_1}\setminus\Omega^*_{\va^\theta}}\Big(\mathbb{C}^0e(v_1^{*1}), e(v_1^{*1})\Big) dx,
\end{equation*}
we have
\begin{align*}
\mathrm{I}_1-\mathrm{I}_1^*&=\int_{\Omega\setminus(D_1^*\cup D_2^*\cup\Omega_{R_1})}\Big[\big(\mathbb{C}^0e(v_1^{1}), e(v_1^{1})\big)-\big(\mathbb{C}^0e(v_1^{*1}), e(v_1^{*1})\big) \Big]dx\\
&\quad +\int_{(D_1^*\cup D_2^*)\setminus (D_1\cup D_2 \cup \Omega_{R_1})} \big(\mathbb{C}^0e(v_1^1,), e(v_1^1)\big) dx\\
&\quad-\int_{(D_1\cup D_2)\setminus(D_1^*\cup D_2^*)}\big(\mathbb{C}^0e(v_1^{*1}), e(v_1^{*1})\big) dx.
\end{align*}
Noticing that 
	\begin{equation*}
|(D_1^*\cup D_2^*)\setminus (D_1\cup D_2 \cup \Omega_{R_1})|\leq C\va, \quad |(D_1\cup D_2)\setminus(D_1^*\cup D_2^*)|\leq C\va,
	\end{equation*}
	 $|\nabla v_1^1|$ and $|\nabla v_1^{*1}|$ are bounded in $(D_1^*\cup D_2^*)\setminus (D_1\cup D_2 \cup \Omega_{R_1})$ and $(D_1\cup D_2)\setminus(D_1^*\cup D_2^*)$, respectively, we have
	\begin{equation*}
	\int_{(D_1^*\cup D_2^*)\setminus (D_1\cup D_2 \cup \Omega_{R_1})} \big(\mathbb{C}^0e(v_1^1,), e(v_1^1)\big) dx\leq C\va,
		\end{equation*}
		and
	\begin{equation*}
	\int_{(D_1\cup D_2)\setminus(D_1^*\cup D_2^*)}\big(\mathbb{C}^0e(v_1^{*1}), e(v_1^{*1})\big) dx\leq C\va.
	\end{equation*}
	Thus, using \eqref{ineq-nabla-v11-v11*}, we have
	\begin{align*}
\mathrm{I}_1-\mathrm{I}_1^*&=\int_{\Omega\setminus(D_1^*\cup D_2^*\cup\Omega_{R_1})}\big(\mathbb{C}^0e(v_1^{1}-v_1^{*1}), e(v_1^{1}-v_1^{*1})\big) dx \\
	& \quad +2\int_{\Omega\setminus(D_1^*\cup D_2^*\cup\Omega_{R_1})}\big(\mathbb{C}^0e(v_1^{*1}), e(v_1^{1}-v_1^{*1})\big) dx+O(\va)\\
	&=O(\va^{\frac{\gamma^2}{(1+\gamma)(1+2\gamma)}}).
	\end{align*}
Hence \eqref{asymp-I1} is proved.
	
	{\bf Step 2.}  Proof of 
	\begin{equation}\label{asym-I2*}
	\mathrm{I}_2= \mathrm{I}_2^* +O(|\ln \va|),
	\end{equation}
where
\begin{equation*}
\mathrm{I}_2^*:=\int_{\Omega_{R_1}^*\setminus \Omega_{\va^\theta}^*} \Big(\mathbb{C}^0e(\bar{u}_1^{*1}), e(\bar{u}_1^{*1})\Big) dx.
	\end{equation*}
Indeed, we notice that
\begin{align*}
	\mathrm{I}_2-\int_{\Omega_{R_1}^*\setminus \Omega_{\va^\theta}^*} \Big(\mathbb{C}^0e(v_1^{*1}), e(v_1^{*1})\Big) dx&=\int_{\Omega_{R_1}^*\setminus \Omega_{\va^\theta}^*} \Big(\mathbb{C}^0e(v_1^1-v_1^{*1}) , e(v_1^1-v_1^{*1})\Big) dx\\
	&\quad +\int_{(\Omega_{R_1}\setminus \Omega_{\va^\theta})\setminus(\Omega_{R_1}^*\setminus \Omega_{\va^\theta}^*)} \Big(\mathbb{C}^0e(v_1^1) , e(v_1^1)\Big) dx\\
	&\quad +2\int_{\Omega_{R_1}^*\setminus \Omega_{\va^\theta}^*} \Big(\mathbb{C}^0e(v_1^{*1}), e(v_1^1-v_1^{*1})\Big) dx\\
	&:=\mathrm{I}_{2,1}+\mathrm{I}_{2,2}+\mathrm{I}_{2,3}.
\end{align*}	
To estimate them,  we use the change of variables:
\begin{equation*}
\left\{\begin{aligned}
&x_1-z_1=|z_1|^{1+\gamma}y_1,\\
&x_2=|z_1|^{1+\gamma}y_2,
\end{aligned}\right.\quad  \text{for}~~ \va^\theta\leq |z_1|\leq R_1.
\end{equation*}
Then $\Omega_{|z_1|+|z_1|^{1+\gamma}} \setminus \Omega_{|z_1|}$ and $\Omega_{|z_1|+|z_1|^{1+\gamma}}^* \setminus \Omega_{|z_1|}^*$ are rescaled into the nearly cubes $Q_1$ and $Q_1^*$ in size, respectively. We denote
\begin{equation*}
V_1^1:=v_1^1(z_1+|z_1|^{1+\gamma}y_1, \, |z_1|^{1+\gamma}y_2),\quad  V_1^{*1}:=v_1^{*1}(z_1+|z_1|^{1+\gamma}y_1, \, |z_1|^{1+\gamma}y_2).
\end{equation*}
	By using \eqref{ineq-bdd-C1alpha} again, we have
	\begin{equation*}
 [\nabla V_1^1 ]_{\gamma, \, Q_1}\leq C,\quad\text{and} \quad [\nabla V_1^{*1} ]_{\gamma, \, Q_1^*}\leq C.
	\end{equation*}
	Interpolating it with \eqref{ineq-v11-v11*} yields 
\begin{equation*}
|\nabla (V_1^1-V_1^{*1})|\leq C\va^{\frac{\gamma^2}{(1+\gamma)(1+2\gamma)}}.
\end{equation*}
Then, rescaling back to $v_1^1-v_1^{*1}$, we have
\begin{equation}\label{ineq-nabla-v11}
|\nabla (v_1^1-v_1^{*1})|\leq C\va^{\frac{\gamma^2}{(1+\gamma)(1+2\gamma)}}|x_1|^{-(1+\gamma)},\quad \text{in} ~~\Omega_{R_1}^*\setminus \Omega_{\va^\theta}^*.
\end{equation}	
Similar, we have
\begin{equation}\label{ineq-nabla-v11*}
|\nabla v_1^1|\leq C|x_1|^{-(1+\gamma)},\quad \text{in}~~\Omega_{R_1}\setminus \Omega_{\va^\theta}, \quad \text{and} \quad |\nabla v_1^{*1}|\leq C|x_1|^{-(1+\gamma)},\quad \text{in}~~\Omega_{R_1}^*\setminus \Omega_{\va^\theta}^*.
\end{equation}
Thus , noticing that $|(\Omega_{R_1}\setminus \Omega_{\va^\theta})\setminus(\Omega_{R_1}^*\setminus \Omega_{\va^\theta}^*)|\leq C\va$ and using \eqref{ineq-nabla-v11} and \eqref{ineq-nabla-v11*}, a direct calculation leads to
\begin{equation*}
|\mathrm{I}_{2,\,1}|\leq C\va^{\frac{2\gamma^2}{(1+\gamma)(1+2\gamma)}}\int_{\va^\theta\leq |x_1|\leq R_1}\frac{1}{|x_1|^{1+\gamma}} dx_1 \leq C\va^{\frac{r^2(1+3\gamma)}{(1+\gamma)(1+2\gamma)^2}},
\end{equation*}
\begin{equation*}
|\mathrm{I}_{2, \, 2}|\leq C\va\int_{\va^\theta\leq |x_1|\leq R_1} \frac{1}{|x_1|^{2(1+\gamma)}} dx_1\leq C\va^{\frac{1+\gamma}{1+2\gamma}},
\end{equation*}
and
\begin{equation*}
|\mathrm{I}_{2, \, 3}|\leq C\va^{\frac{\gamma^2}{(1+\gamma)(1+2\gamma)}}\int_{\va^\theta\leq |x_1|\leq R_1} \frac{1}{|x_1|^{1+\gamma}} dx_1\leq C\va^{\frac{\gamma^3}{(1+\gamma)(1+2\gamma)^2}}.
\end{equation*}	
Hence, we have
\begin{equation}\label{asymp-I2-I2*}
\mathrm{I}_2=\int_{\Omega_{R_1}^*\setminus \Omega_{\va^\theta}^*} \Big(\mathbb{C}^0e(v_1^{*1}), e(v_1^{*1})\Big) dx+O(\va^{\frac{\gamma^3}{(1+\gamma)(1+2\gamma)^2}}).
\end{equation}
	
 Next, we use the auxiliary function $\bar{u}_1^{*1}$ defined in \eqref{def-bar-u*} to get \eqref{asym-I2*}. We notice that
	\begin{align}\label{equal-I-2*}
\int_{\Omega_{R_1}^*\setminus \Omega_{\va^\theta}^*} \Big(\mathbb{C}^0e(v_1^{*1}), e(v_1^{*1})\Big) dx-\mathrm{I}_2^*&=  2 \int_{\Omega_{R_1}^*\setminus \Omega_{\va^\theta}^*} \Big(\mathbb{C}^0e(\bar{u}_1^{*1}), e(v_1^{*1}-\bar{u}_1^{*1})\Big) dx \notag \\
&\quad +\int_{\Omega_{R_1}^*\setminus \Omega_{\va^\theta}^*} \Big(\mathbb{C}^0e(v_1^{*1}-\bar{u}_1^{*1}), e(v_1^{*1}-\bar{u}_1^{*1})\Big) dx.
	\end{align}
By virtue of \eqref{nabla-bar-u*} and \eqref{nabla-v*-123},  
\begin{equation*}
\Big|\int_{\Omega_{R_1}^*\setminus \Omega_{\va^\theta}^*} \Big(\mathbb{C}^0e(\bar{u}_1^{*1}), e(v_1^{*1}-\bar{u}_1^{*1})\Big) dx\Big| \leq \int_{\va^\theta\leq |x_1|\leq R_1} \frac{C}{|x_1|} dx_1 \leq C|\ln \va|,
\end{equation*}		
and
\begin{equation*}
\int_{\Omega_{R_1}^*\setminus \Omega_{\va^\theta}^*} \Big(\mathbb{C}^0e(v_1^{*1}-\bar{u}_1^{*1}), e(v_1^{*1}-\bar{u}_1^{*1})\Big) dx \leq \int_{\va^\theta\leq |x_1|\leq R_1} \frac{C}{|x_1|^{1-\gamma}} dx_1 \leq C.
\end{equation*}
These, together with \eqref{asymp-I2-I2*} and \eqref{equal-I-2*}, lead to \eqref{asym-I2*}.	
			
{\bf Step 3.} We complete the proof of  asymptotic expansion of $a_{11}^{11}$. Recalling that
\begin{align*}
\mathrm{I}_3&= \int_{\Omega_{\va^\theta}} \Big(\mathbb{C}^0e(v_1^1), e(v_1^1)\Big) dx\\
&= \int_{\Omega_{\va^\theta}} \Big( \mathbb{C}^0e(\bar{u}_1^1), e(\bar{u}_1^1)\Big) dx +2 \int_{\Omega_{\va^\theta}} \Big( \mathbb{C}^0e(\bar{u}_1^1), e(v_1^1-\bar{u}_1^1)\Big) dx\\
&\quad + \int_{\Omega_{\va^\theta}} \Big( \mathbb{C}^0e(v_1^1-\bar{u}_1^1), e(v_1^1-\bar{u}_1^1)\Big) dx, 
\end{align*}
and using \eqref{nablau_bar-interior} and \eqref{nabla-w-i0}, we obtain that
\begin{equation*}
\Big|\int_{\Omega_{\va^\theta}} \Big( \mathbb{C}^0e(\bar{u}_1^1), e(v_1^1-\bar{u}_1^1)\Big) dx \Big| \leq \int_{|x_1|\leq \va^\theta} \frac{C}{(\va+|x_1|^{1+\gamma})^{\frac{1}{1+\gamma}}} dx_1\leq C|\ln \va|,
\end{equation*}
and 
\begin{equation*}
\int_{\Omega_{\va^\theta}} \Big( \mathbb{C}^0e(v_1^1-\bar{u}_1^1), e(v_1^1-\bar{u}_1^1)\Big) dx\leq \int_{|x_1|\leq \va^\theta} \frac{C}{(\va+|x_1|^{1+\gamma})^{\frac{1-\gamma}{1+\gamma}}} dx\leq C.
\end{equation*}
Thus, we have
\begin{equation*}
\mathrm{I}_3=\int_{\Omega_{\va^\theta}} \Big( \mathbb{C}^0e(\bar{u}_1^1), e(\bar{u}_1^1)\Big) dx  +O(|\ln \va|).
\end{equation*}

Combining with \eqref{asymp-I1} and \eqref{asym-I2*}, we have
\begin{equation}\label{asymp-a11-half}
\begin{aligned}
a_{11}^{11}&=\int_{\Omega_{\va^\theta}} \Big( \mathbb{C}^0e(\bar{u}_1^1), e(\bar{u}_1^1)\Big) dx +\mathrm{I}_2^*+O(|\ln \va|)\\
&=\int_{|x_1|\leq \va^\theta} \frac{\mu \ dx_1}{\va+ \kappa|x_1|^{1+\gamma}+O(|x_1|^{2+\gamma})} \\
&\quad +\int_{\va^\theta\leq |x_1|\leq R_1} \frac{\mu \ dx_1}{ \kappa|x_1|^{1+\gamma}+O(|x_1|^{2+\gamma})} +O(|\ln \va|),
\end{aligned}
\end{equation} 
where the second equality is obtained by using \eqref{nablau_bar-interior},  \eqref{nabla-bar-u*} and \eqref{asymp-symm-h1h2}. Moreover, a direct calculation leads to
	\begin{equation*}
\int_{|x_1|\leq \va^\theta} \frac{1 }{\va+ \kappa|x_1|^{1+\gamma}+O(|x_1|^{2+\gamma})} \ dx_1=\int_{|x_1|\leq \va^\theta} \frac{1 }{\va+ \kappa|x_1|^{1+\gamma}} \ dx_1 +O(\va^{\frac{\gamma(1-\gamma)}{(1+2\gamma)^2}}),
	\end{equation*}
and 
\begin{equation*}
\int_{\va^\theta\leq |x_1|\leq R_1} \frac{1}{ \kappa|x_1|^{1+\gamma}+O(|x_1|^{2+\gamma})} \ dx_1=\int_{\va^\theta\leq |x_1|\leq R_1} \frac{1}{ \kappa|x_1|^{1+\gamma}} \ dx_1+O(1),
\end{equation*}
and
\begin{equation*}
\int_{\va^\theta\leq |x_1|\leq R_1} \Big( \frac{1}{ \kappa|x_1|^{1+\gamma}}- \frac{1}{\va+ \kappa|x_1|^{1+\gamma}}\Big)\ dx_1\leq  \int_{\va^\theta\leq |x_1|\leq R_1} \frac{C\va}{|x_1|^{2(1+\gamma)}} dx_1\leq C\va^{\frac{1+\gamma}{1+2\gamma}}.
\end{equation*}
These, together with \eqref{asymp-a11-half},  lead to
\begin{align*}
a_{11}^{11}&=\mu \int_{|x_1|\leq R_1} \frac{1}{\va+\kappa|x_1|^{1+\gamma}} \ dx_1 + O(|\ln \va|)\\
&=\frac{2\mu}{\va^{\frac{\gamma}{1+\gamma}}\kappa^{\frac{1}{1+\gamma}}} \int_{0}^{+\infty}\frac{1}{1+ t^{1+\gamma}} dt+ O(|\ln \va|)\\
&=\frac{\mu\widetilde{Q}_\gamma}{\va^{\frac{\gamma}{1+\gamma}}\kappa^{\frac{1}{1+\gamma}}}\Big(1+O(\va^{\frac{\gamma}{1+\gamma}}|\ln \va|)\Big).
\end{align*}
 The proof of $a_{11}^{22}$ in \eqref{asym-a1111} is the same. We omit it here.
\end{proof}

Now, we are in a position to prove Theorem \ref{thm-asymptotics}. 
\begin{proof}[Proof of Theorem \ref{thm-asymptotics}]
	Recalling from \eqref{asymp-nabla-u}, \eqref{ineq-bdd-ub} and \eqref{equal-C11-C12},  it follows from Propositions \ref{prop-blowup-factor-limit} and \ref{proposition-a11-22} that for $x\in \Omega_{R_1}$,
\begin{equation}\label{asym-nablau-half}
\begin{aligned}
\nabla u(x)&=\frac{\va^{\frac{\gamma}{1+\gamma}}\kappa^{\frac{1}{1+\gamma}}}{\widetilde{Q}_\gamma}
\Big( \frac{\widetilde{b}_{*1}^{1}[\varphi]}{\mu}\nabla v_1^1(x) +\frac{\widetilde{b}_{*1}^{2}[\varphi]}{\lam+2\mu}\nabla v_1^2(x) \Big)\Big(1+O(\va^{\frac{\gamma}{1+2\gamma}})\Big)\\
&\quad+O(1)\|\varphi\|_{L^\infty(\ptl D)}.
\end{aligned}
\end{equation}
By virtue of the definition of $\bar{u}_1^k$,  \eqref{nablau_bar-interior} and \eqref{nabla-w-i0}, we have, for $x\in \Omega_{\va^{1/(1+\gamma)}}$,
\begin{equation*}
\nabla v_1^k(x) = \frac{1}{\delta(x_1)}E_{k2}+O(\va^{-\frac{1}{1+\gamma}})=\frac{1}{\va+2h_1(x_1)}\Big(E_{k2}+O(\va^{\frac{\gamma}{1+\gamma}})\Big),\quad k=1, \, 2.
\end{equation*}
Substituting them into \eqref{asym-nablau-half},  we have \eqref{asymp-nabla}. The proof is completed.
\end{proof}

\vspace{0.5cm}

\section{ Estimates in Higher dimensions $d\geq 3$ }\label{sec4}

This section is devoted to proving Proposition \ref{prop-interior} in higher dimensions $d\geq 3$. We only list some main differences with that in two dimensions.

Recall that $v_i^l$ and $v_0$ is defined in \eqref{equ_v1} and \eqref{equ_v3}. From the first line of the decomposition in \eqref{decomposition_u2},  
\begin{equation*}
|\nabla u|\leq \sum\limits_{i=1}^2\sum_{l=1}^{\frac{d(d+1)}{2}}|C_i^l||\nabla v_i^l| +|\nabla v_0| \quad \text{in}~~\Omega.
\end{equation*}
Define $\bar{u}_1^l=\bar{u}\psi^l$ and $\bar{u}_2^l=\underline{u}\psi^l$ as \eqref{def-function1} and \eqref{def-function2} with $x_1$, $x_2$ replaced by $x'$, $x_d$, $l=1, 2, \cdots, \frac{d(d+1)}{2}$. 

\begin{proof}[Proof of Propostion \ref{prop-interior} for $d\geq 3$]
 	
To prove \eqref{nabla-v-i0}, by virtue of \eqref{nablau_bar-interior}, one has for $i=1,\,2$, $l=1,\,2,\cdots, \, d$,
\begin{equation*}
|\nabla \bar{u}_i^l(x)|\leq \frac{C}{\va+|x'|^{1+\gamma}}\quad x\in \Omega_{R_1}.
\end{equation*}
Noticing that the calculation of \eqref{ineq-semi-holder-norm} still holds for $d\geq 3$. Instead of \eqref{ineq-s-va},  taking $$\mathcal{M}_3=\pvint_{\widehat{\Omega}_s(z')} \mathbb{C}^0e(\bar{u}_i^l(y)) \ dy, $$ 
we have, for $0\leq |z'|\leq \va^{\frac{1}{1+\gamma}}$ and $0<s<\va^{\frac{1}{1+\gamma}}$,
\begin{align*}
&\int_{\widehat{\Omega}_s(z')}|\mathbb{C}^0e(\bar{u}_i^l)-\mathcal{M}_3|^2 \ dx \\
&\leq C\Big(\frac{s^{d+1}}{\va^{1+\frac{2}{1+\gamma}}}
+\frac{s^{d-1}}{\va^{\frac{2}{1+\gamma}-1}}
+\frac{s^{d+1-2\gamma}}{\va^{1+\frac{2}{1+\gamma}-2\gamma}}
+\frac{s^{d-1+2\gamma}}{\va^{2\gamma+\frac{2}{1+\gamma}-1}}\Big)=:\widetilde{G}_1(s).
\end{align*}
Denote $F(t):=\int_{\widehat{\Omega}_t(z')} |\nabla (v_i^l-\bar{u}_i^l)|^2 \ dx$,
\begin{equation}\label{ineq-itera-G}
F(t)\leq \Big(\frac{C_1\va}{s-t}\Big)^2F(s)+C\widetilde{G}_1(s),\quad \forall~~ 0<t<s< \va^{\frac{1}{1+\gamma}}.
\end{equation}
Similar to {\bf Case 1} in {\bf STEP 2} in Subsection \ref{sub-3.1}, set $t_j=\delta(z')+2C_1j\va$, $j=0,\, 1,\, 2,\, \cdots$, and let $k=\Big[\frac{1}{4C_1\va^{\frac{\gamma}{1+\gamma}}}\Big]$. By using \eqref{ineq-itera-G} with $s=t_{j+1}$ and $t=t_j$, we have 
\begin{equation*}
F(t_j)\leq \frac{1}{4}F(t_{j+1})+C(j+1)^{d+1}\va^{d-\frac{2}{1+\gamma}}, \quad j=1, \, 2,\, \cdots.
\end{equation*}
After $k$ iterations, we obtain
\begin{equation*}
\int_{\widehat{\Omega}_{\delta(z')}(z')}|\nabla (v_i^l-\bar{u}_i^l)|^2 \ dx \leq C \va^{d-\frac{2}{1+\gamma}}.
\end{equation*}
Instead of \eqref{ineq-H-z}, for $\va^{\frac{1}{1+\gamma}}<|z'|<R_1$, $0<s<\frac{2}{3}|z'|$, 
\begin{align*}
&\int_{\widehat{\Omega}_s(z')}|\mathbb{C}^0e(\bar{u}_i^l)-\mathcal{M}_3|^2 \ dx \\
&\leq C\Big(\frac{s^{d+1}}{|z'|^{3+\gamma}}
+\frac{s^{d-1}}{|z'|^{1-\gamma}}
+\frac{s^{d+1-2\gamma}}{|z'|^{3-\gamma-2\gamma^2}}
+\frac{s^{d-1+2\gamma}}{|z'|^{1+\gamma+2\gamma^2}}\Big)=:\widetilde{G}_2(s).
\end{align*}
and
\begin{equation}\label{ineq-itera-H}
F(t)\leq\Big(\frac{C_2|z'|^{1+\gamma}}{s-t}\Big)^2F(s)+C\widetilde{G}_2(s),\quad \forall ~~0<t<s<\frac{2}{3}|z'|.
\end{equation}
Let $t_j=\delta(z')+2C_2j|z'|^{1+\gamma}$, $j=0, \, 1,\, 2,\cdots$ and $k=\Big[\frac{1}{4C_2|z'|^\gamma}\Big]$. By \eqref{ineq-itera-H} with $t=t_j$ and $s=t_{j+1}$, we have
\begin{equation*}
F(t_j)\leq \frac{1}{4}F(t_{j+1})+C(j+1)^{d+1}|z'|^{(1+\gamma)(d-\frac{2}{1+\gamma})}, \quad j=1, \, 2, \, \cdots.
\end{equation*}
After $k$ iteration, we have 
\begin{equation*}
\int_{\widehat{\Omega}_{\delta(z')}(z')}|\nabla (v_i^l-\bar{u}_i^l)|^2  \ dx \leq C|z'|^{(1+\gamma)(d-\frac{2}{1+\gamma})}.
\end{equation*}
Thus, instead of \eqref{energy_w11_inomega_z1}, we have
\begin{equation}\label{local-energy-esti}
\int_{\widehat{\Omega}_{\delta(z')}(z')}|\nabla (v_i^l-\bar{u}_i^l)|^2 \ dx \leq C \delta(z')^{d-\frac{2}{1+\gamma}}.
\end{equation}
As in {\bf STEP 3}  in Subsection \ref{sub-3.1}, by using Theorems \ref{lem-global-C1alp-estimates} and \ref{lem-lp-esti}, we have, instead of \eqref{ineq-scale-original}
\begin{align*}
&\|\nabla (v_i^l-\bar{u}_i^l)\|_{L^\infty(\widehat{\Omega}_{\delta(z')/4}(z'))}\\
&\leq \frac{C}{\delta(z')} \big(\delta(z')^{1-\frac{d}{2}}\|\nabla(v_i^l-\bar{u}_i^l)\|_{L^2(\widehat{\Omega}_{\delta(z')}(z'))}+\delta(z')^{1+\gamma}[\nabla \bar{u}]_{\gamma, \widehat{\Omega}_{\delta(z')}(z')} \big).
\end{align*}
By using \eqref{local-energy-esti} and \eqref{ineq-semi-holder-norm}, one has \eqref{nabla-v-i0} holds for $d\geq 3$.

In order to prove \eqref{nabla-v-i3}, from the definition of $\bar{u}_i^l$ and \eqref{nablau_bar-interior},  one has for $i=1, \, 2$, $l=d+1, \, \cdots,\, \frac{d(d+1)}{2}$,
\begin{equation*}
|\nabla \bar{u}_i^l| \leq \frac{C(\va+|x'|)}{\va+|x'|^{1+\gamma}},\quad x\in\Omega_{R_1},\quad\text{and}\quad |\nabla \bar{u}_i^l(x)|\leq C,\quad x\in \Omega\setminus \Omega_{R_1}.
\end{equation*}
Instead of \eqref{ineq-ce-u3-1}, we have for $0<|z'|<\va^{\frac{1}{1+\gamma}}$, $0<s<\va^{\frac{1}{1+\gamma}}$,
\begin{align*}
&\int_{\widehat{\Omega}_s(z')}|\mathbb{C}^0e(\bar{u}_i^l)-\mathcal{M}_3|^2 \ dx\\ 
&\leq C\big(\va^{-1}s^{d+1}+\va^{1-2\gamma}s^{d+2\gamma-1}+\va^{2\gamma-1}s^{d+1-2\gamma}+\va s^{d-1} \big)=:\widehat{G}_1(s).
\end{align*}
Denoting $F(t)=\int_{\widehat{\Omega}_t(z')}|\nabla (v_i^l-\bar{u}_i^l)|^2 \ dx$, 
\begin{equation}\label{itera-va-s-t}
F(t)\leq \Big(\frac{C_1\va}{s-t}\Big)^2F(s) + C\widehat{G}_1(s), \quad \forall~~0<t<s<\va^{\frac{1}{1+\gamma}}.
\end{equation}
Similar to {\bf Case 1} in {\bf STEP 2} in Subsection \ref{sub-3.1}, set $t_j=\delta(z') +2C_1j \va$, $j=1, \, 2, \, \cdots$ and let $k=\Big[\frac{1}{4C_1\va^{\frac{\gamma}{1+\gamma}}}\Big]$. By taking $s=t_{j+1}$ and $t=t_j$ in \eqref{itera-va-s-t}, one has 
\begin{equation*}
F(t_j)\leq \frac{1}{4} F(t_{j+1}) + C(j+1)^{d+1}\va^d.
\end{equation*}
After $k$ iteration, one has
\begin{equation*}
\int_{\widehat{\Omega}_{\delta(z')}(z')} |\nabla (v_i^l -\bar{u}_i^l)|^2 \ dx \leq C\va^d.
\end{equation*}
For $\va^{\frac{1}{1+\gamma}}<|z'|<R_1$, $0<s<\frac{2}{3}|z'|$, instead of \eqref{ineq-bar-h-s} and \eqref{intera-z-F}, one has
\begin{align*}
&\int_{\widehat{\Omega}_s(z')}|\mathbb{C}^0e(\bar{u}_i^l)-\mathcal{M}_3|^2 \ dx\\
&\leq C\left(\frac{s^{d+1}}{|z'|^{1+\gamma}}
+\frac{s^{d-1+2\gamma}}{|z'|^{(1+\gamma)(2\gamma-1)}}
+\frac{s^{d+1-2\gamma}}{|z'|^{(1+\gamma)(1-2\gamma)}}
+\frac{s^{d-1}}{|z'|^{-1-\gamma} }\right)=:\widehat{G}_2(s),
\end{align*}
and
\begin{equation}\label{itera-bar-H}
F(t)\leq \Big(\frac{C_2|z'|^{1+\gamma}}{s-t}\Big)^2 F(s) +C\widehat{G}_2(s),\quad \forall ~~ 0< t< s< \frac{2}{3}|z'|.
\end{equation}
Let $t_j=\delta(z')+2C_2j|z'|^{1+\gamma}$, $j=0,\, 1, \, 2, \cdots$, and $k=\Big[\frac{1}{4C_2|z'|^\gamma}\Big]$, by taking $s=t_{j+1}$ and $t=t_j$ in \eqref{itera-bar-H}. After $k$ iteration, we have
\begin{equation*}
\int_{\widehat{\Omega}_{\delta(z')}(z')}|\nabla (v_i^l- \bar{u}_i^l)|^2  \ dx \leq C|z'|^{(1+\gamma)d}.
\end{equation*}
Thus, we have 
\begin{equation}\label{local-engy-d-3}
\int_{\widehat{\Omega}_{\delta(z')}(z')}|\nabla (v_i^l- \bar{u}_i^l)|^2  \ dx \leq C\delta(z')^d, \quad i=1, \, 2, ~~l=d+1,\, \cdots, \, \frac{d(d+1)}{2}.
\end{equation}
As in {\bf STEP 3} in Subsection \ref{sub-3.1}, by virtue of \eqref{ineq-holder} and \eqref{ineq-semi-holder-norm}, a calculation yields
\begin{equation*}
[\mathbb{C}^0e(\bar{u}_i^l)]_{\gamma, \widehat{\Omega}_{\delta(z')}(z')} 
\leq C[\nabla \bar{u}_i^l]_{\gamma, \widehat{\Omega}_{\delta(z')}(z')}
\leq C\sum\limits_{i,j=1}^3[x_i\ptl_j\bar{u}]_{\gamma,\, \widehat{\Omega}_{\delta(z')}(z')}
\leq C \delta(z')^{-\gamma}.
\end{equation*}
By using \eqref{local-engy-d-3}, one has
\begin{align*}
&\|\nabla (v_i^l -\bar{u}_i^l)\|_{L^\infty(\widehat{\Omega}_{\delta(z')/4}(z'))}\\
&\leq \frac{C}{\delta(z')} \Big(\delta(z')^{1-\frac{d}{2}}\|\nabla (v_i^l-\bar{u}_i^l)\|_{L^2(\widehat{\Omega}_{\delta(z')}(z'))}+\delta(z')^{1+\gamma}[\nabla \bar{u}_i^l]_{\gamma, \widehat{\Omega}_{\delta(z')}(z')}\Big)\leq C.
\end{align*}
Thus, \eqref{nabla-v-i3} holds for $d\geq 3$. The proof is completed.
\end{proof}

\vspace{0.5cm}

\section{Appendix: $C^{1,\gamma}$ estimates and $W^{1, p}$ estimates}\label{sec-c1-alp}

For readers' convenience, we adapt the $C^{1,\alp}$ estimates and the $W^{1, p}$ estimates for a constant coefficient elliptic system to our setting, especially with the boundary data being partially zero. When applying these estimates to prove Proposition \ref{prop-interior}, we observe that the H\"{o}lder semi-norm of inhomogeneous term appeared in \eqref{w20'} allows us to obtain the upper and lower bound estimates of $|\nabla v_i^l(x)|$ in $\Omega_{R_1}$, see the proof of \eqref{nabla-v-i0} and \eqref{nabla-v-i3}. Our proof is based on the standard methods in \cite{gm} with using the H\"{o}lder semi-norm to replace the H\"{o}lder norm there. 

\subsection{$C^{1,\gamma}$ estimates } In this subsection, we shall use the Campanato's approach, see e.g. \cite{gm}, to prove Theorem \ref{lem-global-C1alp-estimates}. 

Let $Q$ be a  Lipschitz domain in $\R^d$, the  Campanato space $\mathcal{L}^{2,\lam}(Q)$, $\lam \geq 0$, is defined as follows
\begin{equation*}
\mathcal{L}^{2,\lam}(Q):=\Big\{u\in L^2(Q)~:~\sup_{\substack{x_0\in Q \\ \rho>0}}\frac{1}{\rho^{\lam}}\int_{B_\rho(x_0)\cap Q}|u-u_{x_0, \rho}|^2 dx< +\infty \Big\},
\end{equation*}
where $u_{x_0, \rho}:=\frac{1}{|Q\cap B_\rho(x_0)|}\int_{Q\cap B_\rho(x_0)}u(x)\,dx$. It is endowed with the norm 
\begin{equation*}
\|u\|_{\mathcal{L}^{2,\lam}(Q)}:=\|u\|_{L^2(Q)}+[u]_{\mathcal{L}^{2, \lam}(Q)},
\end{equation*}
where the semi-norm $[\cdot]_{\mathcal{L}^{2,\lam}(Q)}$ is defined by
\begin{equation*}
[u]^2_{\mathcal{L}^{2, \lam}(Q)}:=\sup_{\substack{ x_0\in Q \\ \rho>0}}\frac{1}{\rho^{\lam}}\int_{B_\rho(x_0)\cap Q}|u-u_{x_0, \rho}|^2 dx.
\end{equation*}
It is clear that if $d< \lam \leq d+2$ and $\gamma=\frac{\lam-d}{2}$, the Campanato space $\mathcal{L}^{2,\lam}(Q)$ is equivalent to the H\"{o}lder space $C^{0, \gamma}(Q)$.

We first recall a classical result in \cite{gm}.

\begin{theorem}(Theorem 5.14 in \cite{gm})\label{thm-514}
Let $Q$ be a bounded Lipschitz domain in $\R^d$. Let $\widetilde{w}\in H^1(Q; \R^d)$ be a solution of 
\begin{equation*}
-\ptl_j\big(C_{ijkl}\ptl_l \widetilde{w}^{(k)} \big) =\ptl_j \widetilde{f}_{ij} \quad \text{in}\quad Q 
\end{equation*}
with $\widetilde{f}_{ij}\in C^{\gamma}(Q)$, $0<\gamma<1$, and constant coefficients $C_{ijkl}$ satisfying \eqref{def-C-0} and \eqref{elliptic}. Then $\nabla \widetilde{w} \in \mathcal{L}_{loc}^{2, \, d+2\gamma}(Q)$ and for $B_R:=B_R(x_0)\subset Q$, 
\begin{equation*}
\|\nabla \widetilde{w}\|_{\mathcal{L}^{2, d+2\gamma}(B_{R/2})}\leq C\left(\|\nabla \widetilde{w}\|_{L^2(B_R)}+[\widetilde{F}]_{\mathcal{L}^{2, d+2\gamma}( B_R)}\right),
\end{equation*}
where $\widetilde{F}:= (\widetilde{f}_{ij})$ and $C=C(d, \gamma, R)$.
\end{theorem}

In view of the proof of Theorem \ref{thm-514} in \cite{gm} and using the equivalence between the H\"{o}lder space and the Campanato space, we have the following interior estimates.

\begin{corollary}\label{lem-interior-estimate}
Let $\widetilde{w}$ be the solution of \eqref{equ-w-divf-q1}. Then under the hypotheses of Theorem \ref{lem-global-C1alp-estimates},  for $B_R:=B_R(x_0)\subset Q$,
\begin{equation}\label{ineq-interior-estimate}
[\nabla \widetilde{w}]_{\gamma, B_{R/2}}\leq C\left(\frac{1}{R^{1+\gamma}}\|\widetilde{w}\|_{L^\infty( B_R)}+[\widetilde{F}]_{\gamma, B_R}\right),
\end{equation}
where $C=C(d, \gamma)$.
\end{corollary}

\begin{proof}[Proof of Theorem \ref{lem-global-C1alp-estimates}]
Since $\Gamma \in C^{1,\gamma}$, then for any $x_0\in \Gamma$, there exists a neighbourhood $U$ of $x_0$ and a homeomorphism $\Psi\in C^{1, \, \gamma} (U)$ such that 
\begin{align*}
&\Psi (U \cap Q )=\mathcal{B}_1^+ = \{y\in \mathcal{B}_1(0) \, : \, y_d>0 \},\\
&\Psi (U\cap\Gamma ) = \partial \mathcal{B}_1^+ \cap \{y\in \R^d \, : \, y_d=0 \},
\end{align*}
where $\mathcal{B}_1(0):=\{y\in \R^d: |y|< 1\}$. Under transformation $y=\Psi (x)=(\Psi^1(x), \cdots,\Psi^d(x))$, we denote
$$\mathcal{W}(y):=\widetilde{w}(\Psi ^{-1}(y)),\quad  \mathcal{J}(y):=\frac{\partial((\Psi^{-1})^1, \, \cdots, \, (\Psi^{-1})^d)}{\partial (y^1, \, \cdots ,\, y^d)},\quad |\mathcal{J}(y)|:=\det \mathcal{J}(y),$$ and
\begin{align*}
\mathcal{A}_{ijkl}(y)&:=C_{i\hat{j}k\hat{l}}|\mathcal{J}(y)|\big(\ptl_{\hat{l}} (\Psi^{-1})^l(y)\big)^{-1}\,\ptl_{\hat{j}}\Psi^j\big(\Psi^{-1}(y)\big),\\
\mathcal{F}_{ij}(y)&:=|\mathcal{J}(y)|\ptl_{\hat{l}}\Psi^j\big(\Psi^{-1}(y)\big)\widetilde{f}_{i\hat{l}}\big(\Psi^{-1}(y)\big).
\end{align*}
Then \eqref{equ-w-divf-q1} becomes
\begin{equation}\label{equ-halfball-y}
-\ptl_j\big(A_{ijkl}(y)\ptl_l\mathcal{W}^{(k)}\big)=\ptl_j\mathcal{F}_{ij}\quad \text{in}\quad B_R^+,
\end{equation}
and $\mathcal{W}=0$ on $\ptl B_R^+\cap\ptl \R^d_+$. Let $y_0= \Psi(x_0)$. Freeze the coefficients and rewrite \eqref{equ-halfball-y} in the form 
\begin{equation*}
-\ptl_j\big(\mathcal{A}_{ijkl}(y_0)\ptl_l\mathcal{W}^{(k)}\big)=\ptl_j\big((\mathcal{A}_{ijkl}(y)-\mathcal{A}_{ijkl}(y_0))\ptl_l\mathcal{W}^{(k)}\big)+\ptl_j\mathcal{F}_{ij}.
\end{equation*}
Then, from the proof of Theorem \ref{thm-514} and the equivalence between the H\"{o}lder space and the Campanato space, we have that for $0<R\leq 1$,
\begin{align*}
\big[\nabla \mathcal{W}\big]_{\gamma,\, \mathcal{B}^+_{R/2}}\leq &C \left(\frac{1}{R^{1+\gamma}}\| \mathcal{W}\|_{L^\infty( \mathcal{B}^+_{R})}+[ \mathcal{F}]_{\gamma,\, \mathcal{B}^+_{R}} \right)\\
&\quad+C\big[(\mathcal{A}_{ijkl}(y)-\mathcal{A}_{ijkl}(y_0))\ptl_l\mathcal{W}^{(k)} \big]_{\gamma, \, \mathcal{B}^+_{R} },
\end{align*}
where $\mathcal{F} :=(\mathcal{F}_{ij})$. Since $\mathcal{A}_{ijkl}(y)\in C^{\gamma}$,
\begin{align*}
\big[(\mathcal{A}_{ijkl}(y)-\mathcal{A}_{ijkl}(y_0))\ptl_l\mathcal{W}^{(k)} \big]_{\gamma, \, \mathcal{B}^+_{R} }
\leq C\left(  R^\gamma  [\nabla \mathcal{W}]_{\gamma,\, \mathcal{B}^+_{R}}+\|\nabla  \mathcal{W}\|_{L^\infty(\mathcal{B}_R^+)}\right).
\end{align*}
By the  interpolation inequality, one has 
\[\|\nabla  \mathcal{W}\|_{L^\infty(\mathcal{B}_R^+)} \leq R^{\gamma }[\nabla \mathcal{W}]_{\gamma, \, \mathcal{B}^+_{R} } + \frac{C}{R}\| \mathcal{W}\|_{L^\infty( \mathcal{B}^+_{R})},\]
where $C=C(n)$. Hence, 
\begin{align}\label{ineq-tilde-W}
\big[\nabla \mathcal{W}\big]_{\gamma,\, \mathcal{B}^+_{R/2}}
\leq C \left(\frac{1}{R^{1+\gamma}}\|\mathcal{W}\|_{L^\infty( \mathcal{B}^+_{R})}+R^\gamma [\nabla \mathcal{W}]_{\gamma, \, \mathcal{B}^+_{R} }+ [\mathcal{F}]_{\gamma,\, \mathcal{B}^+_{R}} \right). 
\end{align}
Since $\Psi$ is a homeomorphism, it follows that the norms in \eqref{ineq-tilde-W} defined on  $\mathcal{B}^+_{R}$ are equivalent to those on $\mathcal{N}=\Psi ^{-1}(\mathcal{B}^+_{R})$, respectively. Thus, changing back to the variable $x$, we obtain
\[\big[\nabla \widetilde{w}\big]_{\gamma,\, \mathcal{N}^\prime} \leq C \left(\frac{1}{R^{1+\gamma}}\|\widetilde{w}\|_{L^\infty(\mathcal{N})}+R^\gamma [\nabla \widetilde{w}]_{\gamma, \, \mathcal{N} }+ [\widetilde{F}]_{\gamma,\, \mathcal{N}} \right),\]
where  $\mathcal{N}^\prime=\Psi^{-1}(\mathcal{B}^+_{R/2})$ and $C=C(n, \gamma, \Psi)$. Furthermore, there exists a constant $0<\sigma<1$, independent on $R$, such that $ B_{\sigma R}(x_0)\cap Q \subset \mathcal{N}^\prime$. 

Therefore, recalling that $\Gamma \subset \ptl Q$ is a boundary portion, for any domain $Q'\subset\subset Q\cup \Gamma$ and for each $x_0\in Q^\prime \cap\Gamma$, there exist  $\mathcal{R}_{0}:=\mathcal{R}_{0}(x_0)$ and $C_0=C_0(n, \gamma, x_0)$ such that
\begin{align}\label{ineq-boundary-estimates}
\big[\nabla \widetilde{w}\big]_{\gamma,\, B_{\mathcal{R}_{0}}(x_0)\cap Q^\prime} 
\leq C_0 \left(\mathcal{R}_{0}^\gamma [\nabla \widetilde{w}]_{\gamma, \, Q^\prime}+\frac{1}{\mathcal{R}_{0}^{1+\gamma}}\|\widetilde{w}\|_{L^\infty( Q)}+ [\widetilde{F}]_{\gamma,\, Q} \right).
\end{align}

Applying the finite covering theorem to the collection of $B_{\mathcal{R}_{0}/2}(x_0)$ for all $x_0\in \Gamma \cap Q^\prime$, there exist finite  $B_{\mathcal{R}_{j}/2}(x_j)$, $j=1, 2, \, \cdots\,, K$, covering  $  \Gamma \cap Q^\prime $. Let $C_{j}$ be the constant in \eqref{ineq-boundary-estimates} corresponding to $x_j$. Set 
\[\widehat{C}:=\max\limits_{1\leq j\leq K}\big\{C_{j}\big\},\quad\widehat{\mathcal{R}}:=\min\limits_{1\leq j\leq K}\big\{\frac{\mathcal{R}_{j}}{2}\big\}.\]
Thus, for any $x_0\in  \Gamma \cap Q^\prime$, there exists $j_0\in \{1, 2, \, \cdots\, ,K \}$ such that
$B_{\widehat{\mathcal{R}}}(x_0)\subset B_{\mathcal{R}_{j_0}}(x_{j_0})$ and
\begin{align}\label{ineq-local-estimates}
\big[\nabla \widetilde{w}\big]_{\gamma,\, B_{\widehat{\mathcal{R}}}(x_0)\cap Q^\prime} 
\leq \widehat{C} \left(\widehat{\mathcal{R}}^\gamma [\nabla \widetilde{w}]_{\gamma, \, Q^\prime }+\frac{1}{\widehat{\mathcal{R}}^{1+\gamma}}\|\widetilde{w}\|_{L^\infty( Q)}+ [\widetilde{F}]_{\gamma,\, Q} \right).
\end{align}

Finally, we give the estimates  on $Q^\prime$. Denote $\widetilde{C}$ be the constant in \eqref{ineq-interior-estimate} of Corollary \ref{lem-interior-estimate}. Let
\[\overline{C}:=\max\{\widehat{C}, \widetilde{C}\}\quad \text{and}\quad \overline{\mathcal{R}}:=\min\{(3\overline{C})^{-1/\gamma}, \widehat{\mathcal{R}}\}.\]
For any $x^1, \, x^2\in Q^\prime $, there are three cases to occur:
\begin{itemize}
\item[(i)\,\,] $|x^1- x^2|\geq \frac{\overline{\mathcal{R}}}{2}$;

\item[(ii)\,] there exists $1\leq j_0\leq K$ such that $x^1, \, x^2\in B_{\overline{\mathcal{R}}/2}(x_{j_0})\cap Q^\prime $;

\item[(iii)]  $x^1, \, x^2\in B_{\overline{\mathcal{R}}/2}\subset
Q^\prime $.
\end{itemize}
For case (i), we have
\[\frac{|\nabla \widetilde{w}(x^1)-\nabla \widetilde{w}(x^2)|}{|x^1- x^2|^{\gamma}}\leq \frac{4}{\overline{\mathcal{R}}^\gamma}\|\nabla \widetilde{w}\|_{\infty, \, Q^\prime }.\]
For case (ii), it follows from (\ref{ineq-local-estimates}) that
\[
\begin{aligned}
\frac{|\nabla \widetilde{w}(x^1)-\nabla \widetilde{w}(x^2)|}{|x^1- x^2|^{\gamma}}& \leq C [\nabla \widetilde{w}]_{\gamma, \, B_{\overline{\mathcal{R}}/2}(x_{j_0})\cap Q^\prime }\leq C[\nabla \widetilde{w}]_{\gamma, \, B_{\overline{\mathcal{R}}}(x_{j_0})\cap Q^\prime }\\
& \leq \overline{C} \left(\overline{\mathcal{R}}^\gamma [\nabla \widetilde{w}]_{\gamma, \, Q^\prime}
+\frac{1}{\overline{\mathcal{R}}^{1+\gamma}}\|\widetilde{w}\|_{L^\infty(Q)}+ [\widetilde{F}]_{\gamma,\, Q} \right).
\end{aligned}
\]
For case (iii), by using Corollary \ref{lem-interior-estimate}, one has
\[
\begin{aligned}
\frac{|\nabla \widetilde{w}(x^1)-\nabla \widetilde{w}(x^2)|}{|x^1- x^2|^{\gamma}}& \leq C [\nabla \widetilde{w}]_{\gamma, \, B_{\overline{\mathcal{R}}/2}}
\leq \overline{C} \left(\frac{1}{\overline{\mathcal{R}}^{1+\gamma}}\|\widetilde{w}\|_{L^\infty( Q)}+ [\widetilde{F}]_{\gamma,\, Q} \right).
\end{aligned}
\]
Hence, in either case, we obtain
\begin{equation*}
[\nabla \widetilde{w}]_{\gamma, \, Q^\prime } \leq \overline{C}\left(\overline{\mathcal{R}}^\gamma [\nabla \widetilde{w}]_{\gamma, \, Q^\prime  }+\frac{1}{\overline{\mathcal{R}}^{1+\gamma}}\|\widetilde{w}\|_{L^\infty( Q)}+[\widetilde{F}]_{\gamma,\, Q}\right)  +  \frac{4}{\overline{\mathcal{R}}^\gamma}\|\nabla \widetilde{w}\|_{L^\infty( Q^\prime )}.
\end{equation*} 
By the interpolation inequality, see e.g. \cite[lemma 6.32]{gt},
\[
\begin{aligned}
\frac{4}{\overline{\mathcal{R}}^\gamma}\|\nabla \widetilde{w}\|_{L^\infty( Q^\prime )} &\leq \frac{1}{3} [\nabla \widetilde{w}]_{\gamma, \, Q^\prime } + \frac{C}{\overline{\mathcal{R}}^{1+\gamma}}\|\widetilde{w}\|_{L^\infty( Q^\prime )}\\
&\leq \frac{1}{3} [\nabla \widetilde{w}]_{\gamma, \, Q^\prime } + \frac{C}{\overline{\mathcal{R}}^{1+\gamma}}\|\widetilde{w}\|_{L^\infty(Q)},
\end{aligned}
\]
where $C=C(n, \, \gamma)$. Since $\overline{\mathcal{R}}\leq (3\overline{C})^{-1/\gamma}$, we get
\begin{equation*}
\big[\nabla \widetilde{w}\big]_{\gamma,\, Q^\prime } \leq C\left( \|\widetilde{w}\|_{L^\infty( Q)}+[\widetilde{F}\,]_{\gamma,\, Q}\right),
\end{equation*}
where $C=C(n, \gamma, Q^\prime, Q)$. By using the interpolation inequality, we obtain \eqref{ineq-global-C1alp-estimates}.
\end{proof}

\subsection{ $W^{1, p}$  estimates } 
\begin{proof}[Proof of Theorem \ref{lem-lp-esti}]
 First, we give the $W^{1, p}$ interior estimates. For any ball $B_R:=B_R(x_0) \subset Q$ with $R \leq 1$, 
since $\widetilde{w}\neq 0$ on $\ptl B_R$, we choose a cut-off function $\eta\in C_0^\infty(B_R)$ such that 
\begin{equation*}
0 \leq \eta \leq 1 \quad \eta= 1 \quad \text{in} ~~B_{\rho}, \quad |\nabla \eta| \leq \frac{C}{R-\rho}.
\end{equation*}
We have $\eta \widetilde{w}$ satisfies
 \begin{equation*}
 \int_{B_R} C_{ijkl}\ptl_l(\eta \widetilde{w}^{(k)}) \ptl_j\varphi^{(i)} \ dx= \int_{B_R} G_i \varphi^{(i)} \ dx + \int_{B_R} \widehat{F}_{ij} \ptl_j \varphi^{(i)} \ dx,\quad \forall ~~\varphi \in C_0^\infty(B_R; \R^d),
 \end{equation*}
where 
\begin{align*}
G_i:=\widetilde{f}_{ij}\ptl_j\eta-C_{ijkl}\ptl_l \widetilde{w}^{(k)} \ptl_j \eta,\quad\text{and}\quad
\widehat{F}_{ij}:= \widetilde{f}_{ij} \eta + C_{ijkl}\widetilde{w}^{(k)}\ptl_l \eta.
\end{align*} 
 Let $v\in H_0^1(B_R; \R^d)$ be the weak solution of 
 \begin{equation}\label{equ-laplas}
 -\laplas v^{(i)}=G_i.
 \end{equation}
 We conclude that $\eta \widetilde{w}$ satisfies 
 \begin{equation*}
 \int_{B_R}C_{ijkl}\ptl_l(\eta \widetilde{w}^{(k)})\ptl_j \varphi^{(i)} \ dx= \int_{B_R} \check{F}_{ij} \ptl_j \varphi^{(i)} \ dx \quad \forall ~~\varphi \in C_0^\infty (B_R; \R^d),
 \end{equation*}
where $\check{F}_{ij}:=\widehat{F}_{ij}+\ptl_j v^{(i)}$.

Since $\widetilde{f}_{ij}k\in C^{\gamma}$, then $\widetilde{f}_{ij}\in L^p(B_R)$ for any $d\leq p<\infty$. We firstly assume that $\widetilde{w}\in W^{1, q}(B_R; \R^d)$, $q\geq 2$. Then we have
\begin{align}
&G_i\in L^{p\wedge q}(B_R),\quad \text{where}\quad p\wedge q:=\min\{p,\,q\},\label{integral-G}\\
&\widehat{F}_{ij}\in L^{p\wedge q^*}(B_R),\quad \text{where}\quad  q^*:=\frac{dq}{d-q}.\label{integral-F}
\end{align}
On the account of \eqref{equ-laplas} and $L^2$ theory, $\nabla^2v \in L^2(B_R)$ and
\begin{equation*}
-\laplas (\ptl_j v^{(i)})=\ptl_j G_i.
\end{equation*}
Then, combining with \eqref{integral-G},  theorem 7.1 in \cite{gm} yields $\nabla (\ptl_j v^{(i)})\in L^{p\wedge q}(B_R)$ and it follows from the Sobolev embedding theorem that $\ptl_j v^{(i)} \in L^{(p\wedge q)^*}(B_R)$. Thus, from \eqref{integral-F}, we have $\check{F}_{ij} \in L^{p\wedge q^*}(B_R)$. Furthermore, using theorem 7.1 in \cite{gm} again, we have 
\begin{align*}
\|\nabla (\eta \widetilde{w})\|_{L^{p\wedge q^*}(B_R)}\leq C\|\check{F}\|_{L^{p\wedge q^*}(B_R)},
\end{align*}
where $C=C(d, \lam, \mu, p, q)$ and $\check{F}:=(\check{F}_{ij})$, $i,\, j=1,\, \cdots, d$. Thus, from \eqref{integral-G}, \eqref{integral-F} and the definition of $G_i$ and $\widehat{F}_{ij}$, one has 
\begin{equation}\label{ineq-p-w-q}
\|\nabla \widetilde{w}\|_{L^{p\wedge q^*}(B_\rho)}\leq \frac{C}{R-\rho}(\|\widetilde{w}\|_{W^{1, q}(B_R)}+\|\widetilde{F}\|_{L^p(B_R)}),
\end{equation}
where $C=C(d, \lam, \mu, p, q)$.

Next, we prove that $\nabla \widetilde{w}\in L^p(B_{R/2})$. Choose a series of balls with radii
\begin{equation*}
\frac{R}{2}<\cdots<R_k<\cdots<R_2<R_1<R.
\end{equation*}
In \eqref{ineq-p-w-q}, we firstly take $\rho=R_1$ and $q=2$, then
\begin{equation*}
\|\nabla \widetilde{w}\|_{L^{p\wedge 2^*}(B_{R_1})}\leq\frac{C}{R-R_1}(\|\widetilde{w}\|_{W^{1, 2}(B_R)}+\|\widetilde{F}\|_{L^p(B_R)}).
\end{equation*}
If $p\leq 2^*$, the proof is completed. If $p>2^*$, then $\nabla \widetilde{w} \in L^{2^*}(B_{R_1})$ and 
\begin{equation}\label{ineq-2*}
\|\nabla \widetilde{w}\|_{L^{ 2^*}(B_{R_1})}\leq\frac{C}{R-R_1}(\|\widetilde{w}\|_{W^{1, 2}(B_R)}+\|\widetilde{F}\|_{L^p(B_R)}).
\end{equation}
By taking $R=R_1$, $\rho=R_2$ and $q=2^*$ in \eqref{ineq-p-w-q} and combining with \eqref{ineq-2*}, one has
\begin{align*}
\|\nabla \widetilde{w}\|_{L^{ p\wedge 2^{**}}(B_{R_2})}&\leq\frac{C}{R_2-R_1}(\|\widetilde{w}\|_{W^{1, 2^*}(B_{R_1})}+\|\widetilde{F}\|_{L^p(B_{R_1})})\\
&\leq \frac{C}{(R-R_1)(R_1-R_2)}(\|\widetilde{w}\|_{W^{1, 2}(B_R)}+\|\widetilde{F}\|_{L^p(B_R)}).
\end{align*}
If $p\leq 2^{**}$, then the proof is completed. If $p>2^{**}$, continuing the above argument within finite steps, one has
$\nabla \widetilde{w}\in L^p(B_{R/2})$ and
\begin{equation}\label{ineq-lp-interior-esti}
\|\nabla \widetilde{w}\|_{L^p(B_{R/2})}\leq C(\|\widetilde{w}\|_{H^{1}(B_R)}+\|\widetilde{F}\|_{L^p(B_R)}),
\end{equation}
where $C=C(d, \lam, \mu, p, \dist(B_R, \ptl Q))$.

Now, we prove the $W^{1, p}$ estimates near boundary $\Gamma$ by using the technology of locally flattening the boundary, which is the same to the proof in Theorem \ref{lem-global-C1alp-estimates}. For simplicity, we use the same notation. Hence, we have that $\mathcal{W}(y):=\widetilde{w}(\Psi ^{-1}(y)) \in H^1(\mathcal{B}_R^+, \R^d)$ satisfies    
\begin{equation*}
\int_{\mathcal{B}_R^+}\mathcal{A}_{ijkl}(y)\ptl_l \mathcal{W}^{(k)}\ptl_j \varphi^{(i)} \ dy=\int_{\mathcal{B}_R^+} \mathcal{F}_{ij}\ptl_j \varphi^{(i)} \ dy,\quad \forall ~~ \varphi \in H_0^1(\mathcal{B}_R^+, \R^d).
\end{equation*}
Following the proof of theorem 7.2 in \cite{gm}, we obtain that for any $d\leq p<\infty$,
\begin{equation*}
\|\nabla \mathcal{W}\|_{L^p(B_{R/2}^+)} \leq C(\|\mathcal{W}\|_{H^1(B_R^+)} + \|\mathcal{F}\|_{L^p(B_R^+)}),
\end{equation*}
 where $C=C(\lam, \mu, p, R, \Psi)$. Changing back to the original variable $x$, we obtain 
 \begin{equation*}
 \|\nabla \widetilde{w}\|_{L^p(\mathcal{N}')}\leq C(\|\widetilde{w}\|_{H^1(\mathcal{N})}+\|\widetilde{F}\|_{L^p(\mathcal{N})}),
 \end{equation*}
 where $\mathcal{N}^\prime=\Psi^{-1}(\mathcal{B}^+_{R/2})$, $\mathcal{N}=\Psi ^{-1}(\mathcal{B}^+_{R})$ and $C=C(\lam, \mu, p, R, \Psi)$. Furthermore, there exists a constant $0<\sigma<1$, independent on $R$, such that $ B_{\sigma R}(x_0)\cap Q \subset \mathcal{N}^\prime$. 
 
 Therefore, for any $x_0 \in Q'\cap \Gamma $, there exists $R_0:=R_0(x_0)>0$ such that 
\begin{equation}\label{ineq-lp-boundary}
\|\nabla \widetilde{w}\|_{L^p(B_{\sigma R_0}(x_0)\cap Q')}\leq C(\|\widetilde{w}\|_{H^1(Q)}+\|\widetilde{F}\|_{L^p(Q)}),
\end{equation}
where $C=C(\lam, \mu, p, x_0, R)$. Combining \eqref{ineq-lp-interior-esti} and \eqref{ineq-lp-boundary} and making use of the finite covering theorem, we obtain that
\begin{equation*}
\|\nabla \widetilde{w} \|_{L^p(Q')}\leq C(\|\widetilde{w}\|_{H^1(Q)}+ \|\widetilde{F}\|_{L^p(Q)}),
\end{equation*}
where $C=C(\lam, \mu, p, Q', Q)$. Thus, \eqref{ineq-w1p} follows from the interpolation inequality. 

In particular, since $\widetilde{w}$ still satisfies \eqref{equ-w-divf-q1} with $\widetilde{F}$ replacing by $\widetilde{F}-\mathcal{M}$,  for any constant matrix $\mathcal{M}=(\mathfrak{a}_{ik})$, $i,\, k=1,\,2,\cdots d$,  then, noticing that $W^{1, p} \hookrightarrow C^{0, \tau}$ for $0<\tau\leq 1-d/p$, \eqref{ineq-c-alp} is proved.
\end{proof}

\noindent
{\bf{\large Acknowledgements.}}
The authors would like to thank Prof. Hongjie Dong and Longjuan Xu for valuable discussions. Y. Chen was partially supported by PSF in China No. 2018M631369. International Postdoctoral Exchange Program No. 20190026 and NSF in China No. 11901036. H.G. Li was partially supported by NSF in China No. 11631002, 11971061, and Beijing NSF No. 1202013.


\end{document}